\documentclass[a4paper, 11pt]{article}

\usepackage{a4wide}
\usepackage[english]{babel}
\usepackage{amsmath,caption}
\usepackage{psfrag}
\usepackage[latin1]{inputenc}
\usepackage[T1]{fontenc}
\usepackage{amsthm} 
\usepackage{amssymb,amsxtra}
\usepackage[usenames]{color}
\usepackage{stmaryrd}
\usepackage{mathrsfs}
\usepackage[table]{xcolor}
\usepackage{young,youngtab,ytableau}
\usepackage{tikz}
\usepackage{tikz-cd,circuitikz}

\usetikzlibrary{decorations.pathreplacing}
\usetikzlibrary{decorations.markings, arrows.meta}
\usetikzlibrary{arrows}

\usepackage{color}
\definecolor{brown}{rgb}{.6,0,0}
\usepackage[
        colorlinks,
        linkcolor=brown, filecolor=brown,  citecolor=brown, 
        pagebackref,
        pdfauthor={},
  pdftitle={},
]{hyperref}

\newcommand{\NN}{\mathbb{N}}

\newcommand{\fS}{\mathfrak{S}}
\newcommand{\blue}{\color{\blue}}

\DeclareMathOperator{\sgn}{sgn}

\DeclareMathOperator{\res}{Res}
\DeclareMathOperator{\ind}{Ind}
\DeclareMathOperator{\Soc}{Soc}
\DeclareMathOperator{\Rad}{Rad}
\DeclareMathOperator{\Hd}{Hd}

\DeclareMathOperator{\Ext}{Ext}

\renewcommand{\leq}{\leqslant}
\renewcommand{\geq}{\geqslant}

\renewcommand{\unrhd}{\trianglerighteqslant}

\tikzcdset{
arrows = dash, start anchor=center, end anchor=center
}

\tikzset{
    gt/.style={
        -,
        postaction={decorate,-},
        decoration={
            markings,
            mark=at position .5 with {
                \arrow{>[width=5pt]}
            }
        }
    }
}

\tikzset{
    lt/.style={
        -,
        postaction={decorate,-},
        decoration={
            markings,
            mark=at position .5 with {
                \arrow{<[width=5pt]}
            }
        }
    }
}

\tikzset{
  mat/.style={                 % style for the matrix layout of the graph
    matrix of nodes,           % each cell contains the text of a node
    row sep=2cm,               % distance between rows
    column sep=2cm,            % distance between columns
    nodes=dot,                 % apply the `dot' style to each node in the matrix
    nodes in empty cells=false % do not draw empty cells in the matrix
  },
  dot/.style={    % style for the dots in the matrix
    fill,         % fill the node shape
    circle,       % circle node shape
    inner sep=0pt % tightly fitted to node text
  },
  con/.style={                              % style for the connecting lines
    postaction={decorate},                  % decorate path after drawing
    decoration={                            % specify the decoration
      markings,                             % mark the path
      mark=at position 0.9 with {\arrow{>}} % mark halfway along with arrow
    }
  },
  lab/.style={ % style for the path labels
    midway,    % label midway along path
    auto       % auto position, use `swap' to override
  }
}

\begin{document}

\swapnumbers
\theoremstyle{definition}
\newtheorem{defi}{Definition}[section]
\newtheorem{rem}[defi]{Remark}
\newtheorem{ques}[defi]{Question}
\newtheorem{expl}[defi]{Example}
\newtheorem{conj}[defi]{Conjecture}
\newtheorem{claim}[defi]{Claim}
\newtheorem{nota}[defi]{Notation}
\newtheorem{noth}[defi]{}
\newtheorem{hypo}[defi]{Hypotheses}

\theoremstyle{plain}
\newtheorem{prop}[defi]{Proposition}
\newtheorem{lemma}[defi]{Lemma}
\newtheorem{cor}[defi]{Corollary}
\newtheorem{thm}[defi]{Theorem}

\renewcommand{\proofname}{\textsl{\textbf{Proof}}}

\setcounter{MaxMatrixCols}{20}

\begin{center}
{\bf\Large On  Ext-Quivers of Blocks of Weight Two\\ for Symmetric Groups}

Susanne Danz and Karin Erdmann
\medskip

\today
\begin{abstract}
\noindent
In this paper we investigate blocks of symmetric groups of weight 2 over fields of odd characteristic $p$.
We develop an algorithm that relates the quivers of two such blocks forming  a $(2:1)$-pair,  as introduced by Scopes. 
We then apply the algorithm to blocks whose $p$-cores are hook partitions, in order to explicitly determine the quivers of these blocks.
As a consequence we conclude that two $p$-blocks $B_1$ and $B_2$ of symmetric groups of weight $2$ whose $p$-cores are hook partitions 
are Morita equivalent if and only if $B_2$ or its conjugate block is in the same Scopes class as $B_1$.

\smallskip


\smallskip
\noindent
{\bf MR Subject Classification:} primary 20C20, 20C30; secondary 16D90
\end{abstract}

\end{center}

\section{Introduction}\label{sec intro}

In this paper, we study $p$-blocks of symmetric groups  $\fS_n$ over $F$, where $F$ is an algebraically closed field of characteristic $p>0$. 
As is well known, each such block is 
parametrized by its ($p$-)weight $w\geq 0$ and its ($p$-)core, the latter being  a partition 
of $n-pw$. 
It was proved by Scopes in \cite{Scopes1991} that there are only finitely many Morita equivalence classes of $p$-blocks for symmetric groups with
any given weight. 
More precisely, given a non-negative integer $w$, there is  a minimal list of  $p$-blocks of symmetric groups of weight $w$,
described in terms  of $p$-cores, such that every block of weight $w$ is Morita equivalent to one
of the blocks in this list; see \cite[3.10]{Richards1996}. This list has 
size 
$\frac{1}{2p}{2p \choose p-1}  + \frac{1}{2}{ \lfloor{pw/2}\rfloor \choose \lfloor{p/2}\rfloor}$), 
and it  has recently been conjectured by Sambale in \cite{Sambale2018} that when $w=2$, no two blocks in this list are Morita equivalent.
(Sambale's Conjecture is stated for $p$-blocks over $\mathbb{Z}$, and he has verified his conjecture computationally, for $p\leq 11$.)

\smallskip

Here we now focus on $p$-blocks of symmetric groups of weight $w=2$, over $F$, for $p\geq 3$. Our aim is to give a precise graph-theoretic description
of the (Gabriel or Ext-) quivers of such blocks, provided their $p$-cores are hook partitions. Since Morita equivalent blocks have isomorphic quivers, we shall
also gain more evidence for Sambale's Conjecture. The case $p=3$ is somewhat special, as we shall explain in \ref{noth p=3}. Therefore, we shall
mostly consider the case where $p\geq 5$.

\smallskip

By work of Scopes \cite{Scopes1995}, Richards \cite{Richards1996}, Chuang--Tan \cite{ChuangTan2001}, Martin \cite{Martin1989,Martin1990}, and
others, much is known about the structure of weight-2 blocks in odd characteristic. In particular, Scopes \cite{Scopes1995} has shown that, for every $p$-block $B$ of weight $w=2$ in the minimal list  mentioned above, 
 there is a finite sequence
of blocks of weight $2$ starting with the principal block of
$F\fS_{2p}$ such that two consecutive 
blocks in the sequence are related by
what is called a $(2:1)$-pair. 
Scopes has also shown that, whenever $(B,\bar{B})$ is a $(2:1)$-pair 
of weight-2 blocks of $F\mathfrak{S}_n$ and $F\mathfrak{S}_{n-1}$, respectively, 
a substantial part of the representation theory of $B$ can already be determined from that of $\bar{B}$; in particular,
the quivers of $B$ and $\bar{B}$ almost coincide, except for one vertex and its adjacent edges. Exploiting this, she has further shown that, for every
block $B$  of $F\mathfrak{S}_n$ of weight 2 and all simple $B$-modules $S$ and $T$, the dimension ${\rm Ext}^1_B(S, T)$ is at most 1.
Since simple modules of symmetric groups are self-dual (see \cite[(7.1.14)]{JK1981}), one also has ${\rm Ext}^1_B(S, T)\cong {\rm Ext}^1_B(T,S)$. Hence, if
${\rm Ext}^1_B(S, T)\neq \{0\}$, then
one usually simply connects $S$ and $T$ by a single edge.

Improvements of Scopes's results due to
Richards \cite{Richards1996} and 
Chuang--Tan \cite{ChuangTan2001}
lead to a 
general description of the quiver of any such weight-2 block. 
We now turn this into a combinatorial algorithm, which is the content of Theorem~\ref{thm Loewy except}. With this, we obtain our first main result of this paper:

\begin{thm}\label{thm main1}
Let $p\geq 5$, and let $B_{k,l}(n,p)$ be a block of $F\mathfrak{S}_n$ of $p$-weight $2$ and $p$-core $(k,1^l)$, for
some $k,l\in\NN_0$. With the graphs defined in Appendix~\ref{sec quiv}, the Ext-quiver of $B_{k,l}(n,p)$ is isomorphic, as an undirected graph, to

\begin{enumerate}

\item[{\rm (a)}] $Q_{0,0}(p)$, if $k=l=0$,
\item[{\rm (b)}] $Q_{k,l}(p)$, if $1\leq k+l\leq p-1$,
\item[{\rm (c)}] $Q_{k-1,l-1}(p)$, if $p+1\leq k+l\leq 2p-1$.
\end{enumerate}
\end{thm}

In fact, in Section~\ref{sec main1}, we shall prove a more detailed version of this statement. 
To do so, we  start with the quiver of the 
principal block of $F\fS_{2p}$, which has already been known by work of Martin \cite{Martin1989}.
In Section~\ref{sec B0} we shall give an elementary and self-contained proof of this result in the case $p\geq 5$. Our most important 
ingredient here will be the decomposition matrix of the principal block of $F\fS_{2p}$, a precise description of the Loewy structures of the Specht modules
in this block, and the results of Chuang--Tan \cite{ChuangTan2001}.
The quiver of the principal block of $F\fS_{2p}$ at hand, we shall then be able to apply Theorem \ref{thm Loewy except}, which gives an algorithm that describes how the quivers of two blocks labelled by hook partitions and forming
a $(2:1)$-pair are related. 
In principle, using this algorithm, one can compute the quiver of an arbitrary block of weight $2$. 
However, the number of quivers one has to consider increases with the prime, so that finding a general description
seems to be rather difficult. We expect that Sambale's Conjecture might follow if one knew
the precise graph structure of the quivers of all weight-2 blocks.

\smallskip

Note that, in the notation of Theorem~\ref{thm main1}, the blocks $B_{k,0}(n,p)$, for $k\in\{0,\ldots,p-1\}$
are principal blocks.
It should be emphasized that the quivers of these have already been computed by Martin~\cite{Martin1989,Martin1990}, although the graphs
are not all
drawn correctly there, since they are not bipartite as they should be, by \cite{ChuangTan2001}.

\smallskip

The information on the quivers provided by Theorem~\ref{thm main1} is sufficient to distinguish  Morita
equivalence classes of blocks of weight $w=2$ whose cores are hook partitions.
The  following result 
is a direct consequence of Proposition~\ref{prop main2}, which shows that, for such blocks, there are only the known Morita equivalences, that is, the isomorphism between a block
and its conjugate, and the Scopes equivalences.

\begin{thm}\label{thm main2}
Let $p\geq 5$ be a prime. Then there are precisely $(p-1)p/2+1$ Morita equivalence classes of $p$-blocks of symmetric groups
of $p$-weight $2$ whose $p$-cores are hook partitions. Representatives of these are labelled by those $p$-core partitions
$(k,1^l)$ satisfying one of the following conditions:
\begin{enumerate}
\item[{\rm (a)}] $k=l=0$, or
\item[{\rm (b)}] $k>l$ and $1\leq k+l\leq p-1$, or
\item[{\rm (c)}] $k>l$ and $p+1 <  k+l\leq 2p-1$.
\end{enumerate}
\end{thm}

We shall give a proof of Proposition~\ref{prop main2} at the end of Section~\ref{sec main1}, by applying Proposition~\ref{prop graph isos}. The arguments will be completely graph-theoretic.

\smallskip

A remark on the cases $p\in\{2,3\}$ seems to be in order. There are five Scopes equivalence classes of $3$-blocks of symmetric groups of weight 2. Their
structure is completely understood, and their quivers easily determined. We shall present them in \ref{noth p=3}. The case $p=2$ is
not covered by \cite{Scopes1995}, and behaves differently. However, by \cite{Scopes1991}, every $2$-block of a symmetric
group of weight 2 is Morita equivalent to $F\fS_4$, or to the principal block of $F\fS_5$. Both are well known, and their quivers can be found in \cite[Appendix D(2B), D(2A)]{Erdmann1990}.

\smallskip

The present paper is organized as follows:
In Section 2  we summarize background on representations of symmetric groups and fix some general notation.
Section 3 recalls  relevant results on blocks of weight $2$ from \cite{Scopes1995}, \cite{Richards1996}, and \cite{ChuangTan2001}. Furthermore, 
we establish Theorem \ref{thm Loewy except}, which will be the key ingredient in our inductive proof of Theorem~\ref{thm main1} in Section~\ref{sec main1}.

In Section 4 we prove our above main results on
blocks whose cores are hook partitions, that is, Theorems \ref{thm main1} and \ref{thm main2}. 
In the Appendix, we collect some useful abacus combinatorics that is used extensively throughout this paper.  Furthermore, for convenience, we discuss  the principal blocks of $F\fS_{2p}$ and $F\fS_{2p+1}$ and their quivers. The results presented in Appendix~\ref{sec B0} are not new,
but not too easily available in the literature. In the last part of the appendix, we introduce the graphs appearing in Theorem~\ref{thm main1}, and prove the combinatorial details for Theorem \ref{thm main2}.

\medskip 

\noindent
{\bf Acknowledgements:} We are grateful to the Mathematical Institute of the University of Oxford and the Department of Mathematics of the University of Eichst\"att-Ingolstadt
for their kind hospitality 
during mutual visits. Moreover, we gratefully acknowledge financial support through a 
Scheme 4 grant of the London Mathematical Society and a proFOR+ grant of the University of Eichst\"att-Ingolstadt.
Lastly, we should like to thank Tommy Hofmann and David Craven for their help with TikZ, and the referee for their careful reading of 
an earlier version of the manuscript.

\section{Notation and Preliminaries}\label{sec pre}

Throughout this section, let $F$ be an algebraically closed field of characteristic $p>0$.
Whenever $G$ is a finite group, by an $FG$-module we understand a finitely generated left $FG$-module.
For background on general representation theory of finite groups we refer to \cite{NT}, for the standard notation and results
concerning representations of symmetric groups we refer to \cite{James1978,JK1981}.

\begin{noth}{\bf General notation.}\,\label{noth block nota}
(a)\, Suppose that $M$ and $N$ are $FG$-modules.
If $N$ is isomorphic to a direct summand of $M$, we write
$N\mid M$. 

If $N$ and $M$ have the same composition factors, that is, represent the same element in the Grothendieck group
of $FG$, then we write $M\sim N$. 
If $M$ and $N$ have no common composition factor, then we say that $M$ and $N$ are \textit{disjoint}.
For every simple $FG$-module $D$, we denote by $[M:D]$ the multiplicity of $D$ as a composition factor of $M$. 

The $F$-linear dual of $M$ will be denoted by $M^*$.

(b)\, Let $G$ be a finite group and $H\leq G$. Let further $B$ be a block of $FG$ and $b$ a block of $FH$. We have
the usual (block) restriction and (block) induction functors
\begin{alignat*}{2}
\res_H^G&: FG-\textbf{mod}\to FH-\textbf{mod}\,,& \quad \ind_H^G&: FH-\textbf{mod}\to FG-\textbf{mod}\,,\\
\res_b^B&: B-\textbf{mod}\to b-\textbf{mod}\,, &\quad  \ind_b^B&: b-\textbf{mod}\to B-\textbf{mod}\,.\\
\end{alignat*}
For ease of notation, we shall also write $M\downarrow_b:=\res_b^B(M):=b\cdot \res_H^G(M)$ and $N\uparrow^B:=\ind_b^B(N):=B\cdot \ind_H^G(N)$, for every
$B$-module $M$ and every $b$-module $N$.

\smallskip

(c)\, If $M$ is an $FG$-module and
 $i\geq 0$, then we denote the $i$th radical of $M$ by $\Rad^i(M)$ and the $i$th socle of $M$ by $\Soc_i(M)$. Moreover, we denote by
  $\Hd(M):=M/\Rad(M)$
the \textit{head} of $M$.

Suppose that $M$ has Loewy length $l\geq 1$ with
Loewy layers $\Rad^{i-1}(M)/\Rad^i(M)\cong D_{i1}\oplus\cdots \oplus D_{i r_i}$, for $i\in\{1,\ldots,l\}$, $r_1,\ldots,r_l\in \NN$
and simple $FG$-modules
$D_{i1},\ldots,D_{ir_i}$. Then we write
\begin{equation}\label{eqn Loewy}
M \ \approx \ \ \begin{matrix}   D_{11}\oplus\cdots \oplus D_{1r_1}\\\vdots\\D_{l1}\oplus\cdots \oplus D_{lr_l} \end{matrix}\,,
\end{equation}
and say that $M$ has \textit{Loewy structure} (\ref{eqn Loewy}).
\end{noth}

\begin{noth}\label{noth Ext quiver}{\bf The Ext-quiver.}\,
Suppose that $G$ is a finite group and that $A$ is the group algebra $FG$ or a block of $FG$. Let further $D_1,\ldots,D_r$ be
representatives of the isomorphism classes of simple $A$-modules with projective covers
$P_1,\ldots,P_r$. The \textit{Ext-quiver} of $A$ is the directed graph with 
vertices $D_1,\ldots,D_r$, and, for $i,j\in\{1,\ldots,r\}$, the number of arrows from $D_i$ to $D_j$ equals
$$[\Rad(P_i)/\Rad^2(P_i)):D_j]=\dim_F(\mathrm{Ext}^1_A(D_i,D_j))=
[\Soc_2(P_j))/\Soc(P_j):D_i];
$$
see, \cite[I.6.3]{Erdmann1990}.

Recall that, for $i,j\in\{1,\ldots,r\}$, one also has $\Ext^1_A(D_i,D_j)\cong \Ext^1 _A(D_j^*,D_i^*)$. We want to apply this
when $A$ is a block of  a group algebra of a symmetric group. For these, all simple modules
are self-dual (see [(7.1.14)]\cite{JK1981}), so that we shall simply connect $D_i$ and $D_j$ by $\dim_F(\mathrm{Ext}^1_A(D_i,D_j))$ undirected edges, for all $i,j\in\{1,\ldots,r\}$.
In fact for a block $A$ of weight $w=2$ in characteristic $p>2$, the dimension of ${\rm Ext}^1_A(D_i, D_j)$ is at most $1$, by  \cite{Scopes1995}.
\end{noth}

\begin{noth}{\bf Partitions, modules and blocks of $F\mathfrak{S}_n$.}\,\label{noth part}
(a)\, We write $\mu\vdash n$, for every partition $\mu$ of $n$, and $\lambda\vdash_p n$, for every $p$-regular partition $\lambda$ of $n$.
By $\mu'$ we denote the conjugate of $\mu$, that is, the Young diagram of $\mu'$ is the transposed of the Young diagram of $\mu$.
Recall that if $\mu$ is $p$-regular, then $\mu'$ is $p$-restricted.

As usual, the dominance ordering on partitions of $n$ will be denoted by $\unrhd$, the lexicographic ordering on partitions of $n$ by $\geq $.

\smallskip

(b)\, The Specht $F\mathfrak{S}_n$-module labelled by $\mu\vdash n$
and the simple $F\mathfrak{S}_n$-module labelled by $\lambda\vdash_p n$ will be denoted by $S^\mu$ and $D^\lambda$, respectively.
Recall that $D^\lambda$ is self-dual.

Suppose that $p\geq 3$. Then, for every $p$-regular partition $\lambda$ of $n$ we denote by $\mathbf{m}(\lambda)$ its \textit{Mullineux conjugate}, that is, the $p$-regular partition of $n$ such that $D^{\mathbf{m}(\lambda)}\cong D^\lambda\otimes \sgn$. 
Recall from \cite[Theorem 8.15]{James1978} that one has $S^{\lambda'}\cong (S^\lambda\otimes \sgn)^*$; in particular, the socle of $S^{\lambda'}$
is isomorphic to $D^{\mathbf{m}(\lambda)}$.

\smallskip

(c)\, Given a block $B$ of $F\mathfrak{S}_n$, we denote by $\kappa_B$ its $p$-core and by $w_B$ its $p$-weight.
As well, for every partition $\lambda\vdash n$, we denote by $\kappa_\lambda$ and $w_\lambda$ its $p$-core
and its $p$-weight, respectively. We say that $\lambda$ \textit{belongs to} $B$ (or \textit{$B$ contains $\lambda$}), if 
$\kappa_\lambda=\kappa_B$.

Conversely, if $\kappa$ is some $p$-core partition and $w\geq 0$ is an integer, then $F\mathfrak{S}_{|\kappa|+pw}$ has a 
block with $p$-core $\kappa$ and $p$-weight $w$. 
\end{noth}

\begin{noth}\label{noth Specht filtration}
{\bf Specht filtrations.}\, An $F\mathfrak{S}_n$-module $M$ is said to admit a \textit{Specht filtration} if there are a series 
of $F\mathfrak{S}_n$-submodules 
$$\{0\}=M_0\subset M_1\subset\cdots\subset M_r\subset M_{r+1}=M$$ 
and partitions $\rho_1,\ldots,\rho_{r+1}$ of $n$ such that $M_i/M_{i-1}\cong S^{\rho_i}$, for $i\in\{1,\ldots,r+1\}$.
In general, $M$ may have several Specht filtrations.
Moreover, if $p\in\{2,3\}$, then the number of factors isomorphic to a given Specht $F\mathfrak{S}_n$-module $S^\lambda$ may depend on the chosen filtration; this has been shown by Hemmer and Nakano in \cite{HN2004}. If, however, $p\geq 5$, then, by \cite{HN2004}
again,
the number of factors isomorphic to a Specht $F\mathfrak{S}_n$-module $S^\lambda$ 
is the same for every Specht filtration of $M$; we shall denote this multiplicity by $(M:S^\lambda)$.

\smallskip

Every projective $F\mathfrak{S}_n$-module admits a Specht filtration; this is well known, see for example \cite[(2.6)]{Donkin}, or \cite[(6.1)]{Green}. If $p\geq 5$ and if $\lambda$ is a $p$-regular partition of $n$, then
one has $(P^\lambda:S^\lambda)=1$. Moreover, if $\mu\neq \lambda$ is any partition of $n$ such that $(P^\lambda:S^\mu)>0$,
then $\mu\rhd\lambda$, by Brauer Reciprocity.
\end{noth}

\begin{noth}{\bf Abacus displays.}\,\label{noth abacus}
Throughout this article, we shall employ some standard combinatorial methods to identify partitions with suitable abacus displays; see \cite[Section 2.7]{JK1981}
and \cite{Scopes1995}. 

\smallskip

(a)\, 
Given a partition $\lambda=(\lambda_1,\ldots,\lambda_s)$ of $n$ and any integer $t\geq s$, we can display $\lambda$ on an abacus 
 $\Gamma_\lambda:=\Gamma_{\lambda,t}$ with $p$ runners and $t$ beads,  one at each of the positions $\beta_i:= \lambda_i-i+t$,
for $i\in\{1,\ldots,s\}$ and $\beta_i:=  -i+t$  if $i>s$. Here we label the positions  
  from left to right, then top to bottom, starting with $0$. 
   In accordance with \cite{Scopes1995}, we label the runners of a fixed abacus from $1,\ldots,p$. Then the  places on 
    runner $i$  represent the non-negative integers with residue $i-1$ modulo $p$.

Note that, given any abacus display $\Gamma_\lambda$ of $\lambda$, one can
easily read off $\lambda$ as follows: for each bead on the abacus, count the number of gaps preceding the bead. Then this number of gaps equals the corresponding part of $\lambda$. For instance, if $p=3$, then the following abaci represent the partition $\lambda=(6,3^3,2^2)$:

\begin{center}
$\Gamma_{\lambda,6}$: \quad \begin{tabular}{ccc}
$-$&$-$&$\bullet$\\
$\bullet$&$-$&$\bullet$\\
$\bullet$&$\bullet$&$-$\\
$-$&$-$&$\bullet$
\end{tabular}\,\quad \text{ and }\quad 
$\Gamma_{\lambda, 7}$: \quad 
\begin{tabular}{ccc} $\bullet$ & $-$& $-$\\
                                 $\bullet$&$\bullet$ & $-$\\
                                 $\bullet$ & $\bullet$ & $\bullet$\\
                                 $-$ & $-$ & $-$\\
                                 $\bullet$& & \end{tabular}
                                 \end{center}

This also illustrates the effect of varying the total number of beads on $\Gamma_{\lambda}$: 
inserting a bead
at position 0 of a given abacus $\Gamma_{\lambda,s}$ and moving every other bead to the next position gives
an abacus display $\Gamma_{\lambda,s+1}$.

\smallskip

Recall further that moving a bead on some runner of $\Gamma_\lambda$ one position up corresponds to removing a rim $p$-hook from
the Young diagram $[\lambda]$, while moving a bead from some runner of $\Gamma_\lambda$ one position to the left (respectively, to the right)
corresponds to removing (respectively, adding) a node to $[\lambda]$. This describe the branching rules. 
Moving all beads on all runners as far up as possible, one obtains an
abacus display of the $p$-core $\kappa_\lambda$.

Lastly, suppose that $B$ is any block of $F\mathfrak{S}_n$ and $\Gamma_{\kappa_B}$ is any abacus display of 
$\kappa_B$ and suppose further that there are $k>0$ more beads on some runner $i>1$ than on runner $i-1$ of
$\Gamma_{\kappa_B}$. One may interchange runners $i$ and $i-1$ to get an abacus display of $\kappa_{\bar{B}}$, where
$\bar{B}$ is a block of $F\mathfrak{S}_{n-k}$ with $w_{\bar{B}}=w_B=:w$; in this case one says that 
$(B,\bar{B})$ is a 
\textit{$(w:k)$-pair.}
In the above example we get, for instance, $\kappa_\lambda=\kappa_B=(3,1)$, and the block $B$ of $F\mathfrak{S}_{19}$
forms a $(5:2)$-pair with the block $\bar{B}$ of $F\mathfrak{S}_{17}$ with $\kappa_{\bar{B}}=(2)$.

As Scopes has shown in \cite{Scopes1991}, if $k\geq w$, then $B$ and $\bar{B}$ are Morita equivalent; the particular Morita equivalence
between $B$ and $\bar{B}$ established in \cite{Scopes1991} is called \textit{Scopes equivalence}. 
Moreover, she proved that for a fixed $w$, there is a finite list of blocks $\bar{B}$ such that every other block can be obtained from 
some block in this list by a sequence of $(w:k)$ pairs, for some $k\geq w$.
We shall come back to 
the notion of Scopes equivalence later in \ref{noth Scopes} and \ref{noth partial Scopes}.

\smallskip

(b)\, In this paper we shall focus on partitions and blocks of weight 2. 
To this end, we recall one last bit of notation from
\cite{Scopes1995}. Suppose that $\lambda=(\lambda_1,\ldots,\lambda_r)$ is a  partition of $n$ of $p$-weight $2$. As above, let
$\Gamma_\lambda:=\Gamma_{\lambda,s}$ be an abacus display of $\lambda$, for some $s\geq r$. 
Suppose that, for $i\in\{1,\ldots,p\}$
we have $m_i$ beads on runner $i$ of $\Gamma_\lambda$. 
The we shall say that we represent $\lambda$ on an  \textit{$[m_1,\ldots,m_p]$-abacus}.

For  $\lambda$ of weight 2, there are exactly two beads that can be moved up on their respective runner in  
$\Gamma_\lambda$, and there are three possible constellations: 

\smallskip

\quad (i)\, There is a bead on some runner $i$ that can be moved two positions up. Then we denote $\lambda$ by $\langle i\rangle$.

\quad (ii)\, There are $1\leq i<j\leq p$ such that there is a movable bead on runner $i$ and a movable bead on runner $j$. Then we denote
$\lambda$ by $\langle j,i\rangle$ or $\langle i,j\rangle$.

\quad (i)\, There is some runner $i$ that has a gap followed by two consecutive beads. Thus one can first move the upper bead one position up, then
the lower bead. In this case, we denote $\lambda$ by $\langle i,i\rangle$.

\medskip

Note that this labelling depends on the fixed choice of an abacus. 
 We shall always state clearly which abacus 
 is used.
\end{noth}

%%%%%%%%%%%%%%%%%%%%%%%%%%%%%%%%%%%%%%%%%%%%%%%%%%%%%%%%%%%%%%%%%%%%%%
\section{Blocks of Weight Two}\label{sec weight 2}

Throughout this section, let $p\geq 3$ be a prime and let $F$ be an algebraically closed field of characteristic $p$.
We begin by recalling some crucial notation from \cite[Section 4]{Richards1996} and \cite[Section 2]{ChuangTan2001}.

\begin{noth}{\bf Colours and $\partial$-values.}\,\label{noth partial}
Suppose that $\lambda$
 is a partition of $n$ with $p$-weight $2$ and $p$-core $\kappa=(\kappa_1,\ldots,\kappa_t)$.

\smallskip

(a)\, One can remove exactly two rim $p$-hooks from $[\lambda]$ to obtain $[\kappa]$. Although there is, in general, not a unique way to do so,
the absolute value of the difference in the leg lengths of the two rim hooks is well defined, and is denoted by $\partial(\lambda)$; see \cite[Lemma~4.1]{Richards1996}.

\smallskip

(b)\, Consider the hook diagram $H_\lambda$ of $\lambda$.
Since $\lambda$ has $p$-weight $2$, there are exactly two entries in $H_\lambda$ that are divisible by $p$; see \cite[2.7.40]{JK1981}.
There are two possibilities for these entries: either both are equal to $p$, or one equals $p$ and the other equals $2p$. Suppose further that
$\partial(\lambda)=0$. If $H_\lambda$ has two entries equal to $p$, then the leg lengths of the corresponding hooks differ by $1$; see \cite[p. 397]{Richards1996}. If the larger leg length is even, one calls $\lambda$ \textit{black}, otherwise \textit{white}. 
If $H_\lambda$ has an entry equal to $2p$ and if the leg length of the corresponding hook has residue $0$ or $3$ modulo $4$, then one also calls
$\lambda$ \textit{black}, otherwise \textit{white}.

\smallskip

(c)\, If $\lambda$ is $p$-restricted, then there is a ($p$-regular) partition $\lambda_+$ of $n$
that is the lexicographically smallest partition with the following properties: $\lambda_+>\lambda$, $\lambda_+$ has $p$-core
$\kappa$, $p$-weight $2$ as well as the same $\partial$-value and (if $\partial(\lambda)=0$) the same colour as $\lambda$.  
As well, if $\lambda$ is $p$-restricted, then $\lambda'$ is $p$-regular, and one has $\lambda_+=\mathbf{m}(\lambda')$.
If $\lambda$ is not $p$-restricted, then $\lambda_+$ does not exist; see \cite[Remarks 2.1]{ChuangTan2001}.

If $\lambda$ is $p$-regular, then there is a ($p$-restricted) partition $\lambda_-$ of $n$ that is the lexicographically largest partition
with the following properties: $\lambda_-<\lambda$, $\lambda_-$ has $p$-core
$\kappa$, $p$-weight $2$ as well as the same $\partial$-value and (if $\partial(\lambda)=0$) the same colour as $\lambda$.  
Moreover, one then has $\lambda_-=\mathbf{m}(\lambda)'$.
If $\lambda$ is not $p$-regular, then $\lambda_-$ does not exist; see \cite[Remarks 2.1]{ChuangTan2001}.

\smallskip

\end{noth}

Next we recall from \cite[Section 3]{Scopes1995} some properties of $(2:1)$-pairs of blocks that will be 
fundamental later in this article.

\begin{noth}{\bf Exceptional partitions of $(2:1)$-pairs.}\,\label{noth (2:1)}
Suppose that $B$ is a block of $F\mathfrak{S}_n$ of weight $2$ with $\kappa_B=(\kappa_1,\ldots,\kappa_t)$ where $\kappa_t\neq 0$.
Furthermore, let $\bar{B}$ be a block of $F\mathfrak{S}_{n-1}$ of weight $2$ such that $(B,\bar{B})$ is a $(2:1)$-pair.
Let $s\geq t$.
As in \ref{noth abacus}, we display $\kappa_B$, $\kappa_{\bar{B}}$ as well as all partitions of $B$ and $\bar{B}$, respectively, on an $[m_1, \ldots, m_p]$-abacus with $2p+s$ beads. 
For ease of notation we identify partitions with their abacus displays as explained in \ref{noth abacus}.
With a suitable choice of $s$,  there is a unique $i\in\{2,\ldots,p\}$ such that $\kappa_{\bar{B}}$ is obtained from $\kappa_B$ by interchanging the $i$th and the $(i-1)$st runner.
Following \cite[Definition 3.1, Definition 3.2]{Scopes1995}, we consider the following partitions of $B$ and $\bar{B}$, respectively:
\begin{align*}
\alpha&:=\alpha(B,\bar{B}):=\langle i,i\rangle\,,\quad \beta:=\beta(B,\bar{B}):=\langle i,i-1\rangle\,,\quad \gamma:=\gamma(B,\bar{B}):=\langle i-1\rangle\,,\\
\bar{\alpha}&:=\bar{\alpha}(B,\bar{B}):=\langle i\rangle\,,\quad \bar{\beta}:=\bar{\beta}(B,\bar{B}):=\langle i,i-1\rangle\,,\quad \bar{\gamma}:=\bar{\gamma}(B,\bar{B}):=\langle i-1,i-1\rangle\,.
\end{align*}

From now on we shall refer to $\alpha,\beta$ and $\gamma$ as the \textit{exceptional partitions} of $B$ (with respect to the pair $(B,\bar{B})$).
Analogously, we shall call $\bar{\alpha},\bar{\beta}$ and $\bar{\gamma}$ the \textit{exceptional partitions} of $\bar{B}$ (with respect to the pair $(B,\bar{B})$). Every partition of $B$ and $\bar{B}$, respectively, that is not exceptional will be called \textit{good} (with respect to the pair $(B,\bar{B})$). 
A $B$-module (respectively, $\bar{B}$-module) will be called \textit{good} if all its composition factors are labelled by good partitions.

By \cite[Lemma~3.5]{Scopes1995}, one has a bijection $\Phi:=\Phi(B,\bar{B})$ between the set of good partitions of $B$ and the set
of good partitions of $\bar{B}$ that preserves the lexicographic ordering as well as $p$-regularity and $p$-singularity. Given a good partition $\lambda$ of $B$, one obtains $\Phi(\lambda)$ by interchanging the 
$i$th and $(i-1)$st runner of the abacus. We shall often denote $\Phi(\lambda)$ by $\bar{\lambda}$, for every good partition of $B$, 
and $\Phi^{-1}(\mu)$ by $\hat{\mu}$, for every good partition $\mu$ of $\bar{B}$.

It should be emphasized that neither the exceptional partitions of $B$ and $\bar{B}$ nor the bijection $\Phi$ depends on the chosen
abacus displays.
\end{noth}

\begin{rem}\label{rem l and r}
Suppose that $B$ and $\bar{B}$ are blocks of $F\mathfrak{S}_n$ and $F\mathfrak{S}_{n-1}$, respectively, of weight $2$ that form a 
$(2:1)$-pair. Moreover, let $\kappa_B=(\kappa_1,\ldots,\kappa_t)$ with $\kappa_t\neq 0$. As in \ref{noth (2:1)}, we display
$\kappa_B$, $\kappa_{\bar{B}}$ and all partitions of $B$ and $\bar{B}$ on an abacus with $s+2p$ beads, for a fixed $s\geq t$.  
Suppose that $\kappa_{\bar{B}}$ is obtained from $\kappa_B$ by swapping runners $i$ and $i-1$.
Again in the notation
of \ref{noth (2:1)}, we consider the $i$th and $(i-1)$st runner of the abaci displaying the exceptional partitions of $B$ and $\bar{B}$. Then we 
have the following constellations, where, in each case, $l_1,l_2,r_1,r_2$ are understood to be the numbers of beads in the respective parts of the 
abacus under consideration, as shown in the diagrams below.

\smallskip

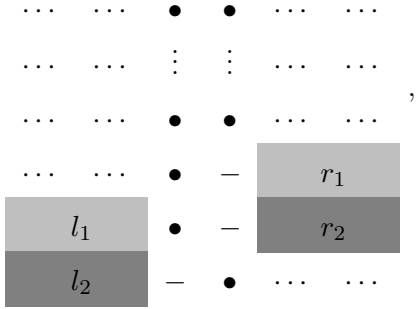
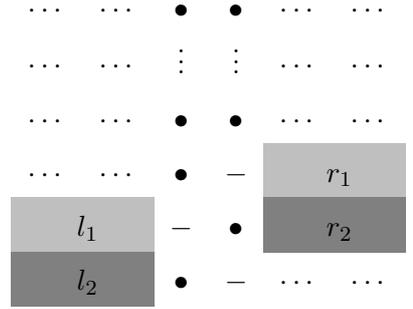
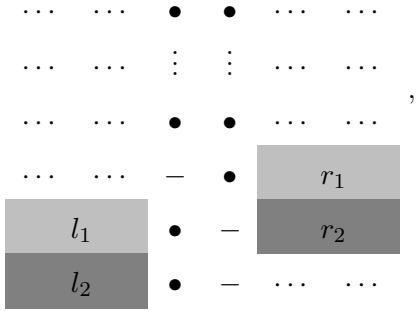
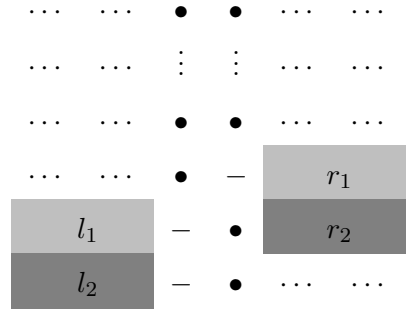
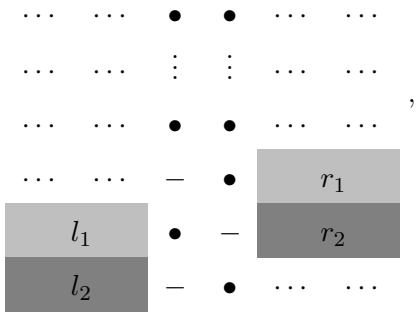
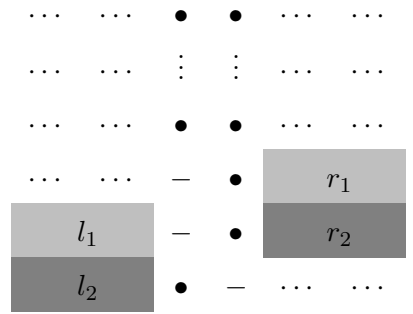
\begin{center}
\begin{tabular}{cccccc}
\multicolumn{6}{c}{$\bar{\alpha}=\langle i\rangle$:}\\
&&&&&\\
$\cdots$&$\cdots$&$\bullet$&$\bullet$&$\cdots$&$\cdots$\\
$\cdots$&$\cdots$&$\vdots$&$\vdots$&$\cdots$&$\cdots$\\
$\cdots$&$\cdots$&$\bullet$&$\bullet$&$\cdots$&$\cdots$\\
$\cdots$&$\cdots$&$\bullet$&$-$&\multicolumn{2}{c}{\cellcolor{lightgray} $r_1$}\\
\multicolumn{2}{c}{\cellcolor{lightgray} $l_1$}&$\bullet$&$-$&\multicolumn{2}{c}{\cellcolor{gray} $r_2$}\\
\multicolumn{2}{c}{\cellcolor{gray} $l_2$}&$-$&$\bullet$&$\cdots$&$\cdots$\\
\end{tabular}\;, \quad \quad
\begin{tabular}{cccccc}
\multicolumn{6}{c}{$\bar{\beta}=\langle i-1,i\rangle$:}\\
&&&&&\\
$\cdots$&$\cdots$&$\bullet$&$\bullet$&$\cdots$&$\cdots$\\
$\cdots$&$\cdots$&$\vdots$&$\vdots$&$\cdots$&$\cdots$\\
$\cdots$&$\cdots$&$\bullet$&$\bullet$&$\cdots$&$\cdots$\\
$\cdots$&$\cdots$&$\bullet$&$-$&\multicolumn{2}{c}{\cellcolor{lightgray} $r_1$}\\
\multicolumn{2}{c}{\cellcolor{lightgray} $l_1$}&$-$&$\bullet$&\multicolumn{2}{c}{\cellcolor{gray} $r_2$}\\
\multicolumn{2}{c}{\cellcolor{gray} $l_2$}&$\bullet$&$-$&$\cdots$&$\cdots$\\
\end{tabular}\;, 
\end{center}

\bigskip
\bigskip

\begin{center}
\begin{tabular}{cccccc}
\multicolumn{6}{c}{$\bar{\gamma}=\langle i-1,i-1\rangle$:}\\
&&&&&\\
$\cdots$&$\cdots$&$\bullet$&$\bullet$&$\cdots$&$\cdots$\\
$\cdots$&$\cdots$&$\vdots$&$\vdots$&$\cdots$&$\cdots$\\
$\cdots$&$\cdots$&$\bullet$&$\bullet$&$\cdots$&$\cdots$\\
$\cdots$&$\cdots$&$-$&$\bullet$&\multicolumn{2}{c}{\cellcolor{lightgray} $r_1$}\\
\multicolumn{2}{c}{\cellcolor{lightgray} $l_1$}&$\bullet$&$-$&\multicolumn{2}{c}{\cellcolor{gray} $r_2$}\\
\multicolumn{2}{c}{\cellcolor{gray} $l_2$}&$\bullet$&$-$&$\cdots$&$\cdots$\\
\end{tabular}\;, \quad\quad
\begin{tabular}{cccccc}
\multicolumn{6}{c}{$\alpha=\langle i,i\rangle$:}\\
&&&&&\\
$\cdots$&$\cdots$&$\bullet$&$\bullet$&$\cdots$&$\cdots$\\
$\cdots$&$\cdots$&$\vdots$&$\vdots$&$\cdots$&$\cdots$\\
$\cdots$&$\cdots$&$\bullet$&$\bullet$&$\cdots$&$\cdots$\\
$\cdots$&$\cdots$&$\bullet$&$-$&\multicolumn{2}{c}{\cellcolor{lightgray} $r_1$}\\
\multicolumn{2}{c}{\cellcolor{lightgray} $l_1$}&$-$&$\bullet$&\multicolumn{2}{c}{\cellcolor{gray} $r_2$}\\
\multicolumn{2}{c}{\cellcolor{gray} $l_2$}&$-$&$\bullet$&$\cdots$&$\cdots$\\
\end{tabular}\;,\quad\quad 
\end{center}

\bigskip

\bigskip

\begin{center}
\begin{tabular}{cccccc}
\multicolumn{6}{c}{$\beta=\langle i-1,i\rangle$:}\\
&&&&&\\
$\cdots$&$\cdots$&$\bullet$&$\bullet$&$\cdots$&$\cdots$\\
$\cdots$&$\cdots$&$\vdots$&$\vdots$&$\cdots$&$\cdots$\\
$\cdots$&$\cdots$&$\bullet$&$\bullet$&$\cdots$&$\cdots$\\
$\cdots$&$\cdots$&$-$&$\bullet$&\multicolumn{2}{c}{\cellcolor{lightgray} $r_1$}\\
\multicolumn{2}{c}{\cellcolor{lightgray} $l_1$}&$\bullet$&$-$&\multicolumn{2}{c}{\cellcolor{gray} $r_2$}\\
\multicolumn{2}{c}{\cellcolor{gray} $l_2$}&$-$&$\bullet$&$\cdots$&$\cdots$\\
\end{tabular}\;, \quad\quad
\begin{tabular}{cccccc}
\multicolumn{6}{c}{$\gamma=\langle i-1\rangle$:}\\
&&&&&\\
$\cdots$&$\cdots$&$\bullet$&$\bullet$&$\cdots$&$\cdots$\\
$\cdots$&$\cdots$&$\vdots$&$\vdots$&$\cdots$&$\cdots$\\
$\cdots$&$\cdots$&$\bullet$&$\bullet$&$\cdots$&$\cdots$\\
$\cdots$&$\cdots$&$-$&$\bullet$&\multicolumn{2}{c}{\cellcolor{lightgray} $r_1$}\\
\multicolumn{2}{c}{\cellcolor{lightgray} $l_1$}&$-$&$\bullet$&\multicolumn{2}{c}{\cellcolor{gray} $r_2$}\\
\multicolumn{2}{c}{\cellcolor{gray} $l_2$}&$\bullet$&$-$&$\cdots$&$\cdots$\\
\end{tabular}\;.
\end{center}

Note that we have $l_1\geq l_2$ and $r_1\geq r_2$, since none of the above 
partitions has a movable bead on any runner different from $i$ and $i-1$.

For our subsequent considerations, in particular those in Theorem~\ref{thm Loewy except}, it will turn out to be 
useful to distinguish the following cases:

\begin{itemize}
\item[\rm{(1)}] $l_1+r_1=0=l_2+r_2$;
\item[\rm{(2)}] $0=l_2+r_2< l_1+r_1<p-2$;
\item[\rm{(3)}] $0=l_2+r_2<l_1+r_1=p-2$;
\item[\rm{(4)}] $0<l_2+r_2< l_1+r_1=p-2$;
\item[\rm{(5)}] $l_2+r_2=p-2=l_1+r_1$;
\item[\rm{(6)}] $0<l_2+r_2\leq l_1+r_1<p-2$.
\end{itemize}

Observe that case (3) occurs precisely when $B$ is the principal block of $F\mathfrak{S}_{2p+1}$, which has $p$-core $(1)$, and
when $\bar{B}$ is the principal block of $F\mathfrak{S}_{2p}$, which has $p$-core $\emptyset$. 
In this case, we further have $\beta=(p+1,1^p)$ and $\alpha=(p+1,2,1^{p-2})$. By \cite[Theorem 23.7]{James1978}, $S^\beta$ is simple, and by \cite[Lemma 4.3]{Scopes1995},
one has $[S^\beta:D^\alpha]\neq 0$. Thus $S^\beta\cong D^\alpha$, in this case.
\end{rem}

\bigskip

The next lemma shows in which of the six cases of Remark~\ref{rem l and r}
the exceptional partitions of a $(2:1)$-pair of
weight-$2$ blocks are $p$-regular or $p$-restricted. We also record the $\partial$-values
of the partitions in question.
The result is an easy consequence of the abacus combinatorics in Appendix~\ref{sec abacus}, and will be important
for our proof of Theorem~\ref{thm main1}.

\begin{lemma}\label{lemma l and r}
Retain the hypotheses and notation as in Remark~\ref{rem l and r}. Moreover, set $d:=l_1+r_1-l_2-r_2$. Then one has the following

\begin{center}
\begin{tabular}{|c|c|c|c|}\hline
partition & $p$-regular &$p$-restricted& $\partial$\\\hline\hline
$\bar{\alpha}$& {\rm (1), (2), (3),} & {\rm (4), (5), (6)}&$d+1$\\
                       & {\rm (4), (5), (6)}&&\\\hline
$\bar{\beta}$& {\rm (1), (2), (3),} & {\rm (2), (3), (4),}&$d$\\
                     & {\rm (4), (6)}.        & {\rm (5), (6)} &\\\hline
$\bar{\gamma}$& {\rm (1), (2), (6)} & {\rm (1), (2), (3),}&$d+1$\\
                          &                            & {\rm (4), (5), (6)} &\\\hline
$\alpha$& {\rm (1), (2), (3),} & {\rm (2), (3), (4)}&$d$ \\
              & {\rm (4), (5), (6)}              & {\rm (5), (6)} &\\\hline
$\beta$& {\rm (1), (2), (6)}& {\rm (4), (5), (6)}&$d+1$\\\hline
$\gamma$& {\rm (1), (2), (3),}& {\rm (1), (2), (3),}&$d$\\
                 & {\rm (4), (6)} &      {\rm (4), (5), (6)} &\\\hline
\end{tabular}
\end{center}

In particular, one has $\partial(\beta)=\partial(\bar{\alpha})=\partial(\bar{\gamma})>0$.
Furthermore, $\partial(\bar{\beta})=\partial(\alpha)=\partial(\gamma)=0$ if and only if $l_1=l_2$ and $r_1=r_2$; in this case,
$\bar{\beta}$, $\alpha$ and $\gamma$ have the same colour, which is black if and only if $l_2+r_2$ is odd.
\end{lemma}

\begin{proof}
We represent the respective partitions on an abacus with $s+2p$ beads as before.
The assertions concerning $p$-regularity and $p$-restrictedness then follow from \ref{noth abacus}(a).
Note that part of this already appears in \cite[Lemma~4.4]{Scopes1995}.
The assertions concerning $\partial$-values and colours are immediate from \ref{noth hooks diagram} and \ref{noth colour weight 2}.
\end{proof}

\begin{lemma}\label{lemma colour good}
Suppose that $B$ and $\bar{B}$ are blocks of $F\mathfrak{S}_n$ and $F\mathfrak{S}_{n-1}$, respectively,
with $w_B=2=w_{\bar{B}}$ that form a $(2:1)$-pair. If $\lambda$ is a good partition of $B$ and $\Phi(\lambda)$ its corresponding good partition of $\bar{B}$, then $\partial(\lambda)=\partial(\Phi(\lambda))$.
If $\partial(\lambda)=\partial(\Phi(\lambda))=0$, then $\lambda$ and $\Phi(\lambda)$ have the same colour.
Moreover, if $\partial(\bar{\beta})=0=\partial(\gamma)$, then also the exceptional partitions
$\bar{\beta}$ and $\gamma$ have the same colour.
\end{lemma}

\begin{proof}
Suppose that $\kappa_B=(\kappa_1,\ldots,\kappa_t)$, for some $t\geq 1$, $\kappa_t\geq 1$. 
For convenience, we set $\bar{\lambda}:=\Phi(\lambda)$. We consider an abacus display $\Gamma_\lambda$
of $\lambda$ and an $[m_1,\ldots,m_p]$-abacus display $\Gamma_{\bar{\lambda}}$ of $\bar{\lambda}$ with $s+2p$ beads, for some $s\geq t$, such that  
$\Gamma_{\bar{\lambda}}$ is obtained from $\Gamma_\lambda$
by swapping runners $i$ and $i-1$, where $i\in \{ 2, \ldots, p\}$.

We first observe the following: suppose that there is a runner $j\neq i$ of $\Gamma_\lambda$
such that there is a bead in row $x$ on runner $j$ and a gap in some row $y<x$ on runner $j$. Then the same holds for
$\Gamma_{\bar{\lambda}}$. Moreover, the number of beads passed when moving the bead to position $(y,j)$ in $\Gamma_\lambda$
equals the number of beads passed when moving the bead to position $(y,j)$ in $\Gamma_{\bar{\lambda}}$.

Next, in the notation of \ref{noth abacus}(b), there are three possibilities: $\lambda=\langle j\rangle$, $\lambda=\langle j,j\rangle$
or $\lambda=\langle k,j\rangle$, for some $j,k\in\{1,\ldots,p\}$, $j<k$. 
If $j,k\notin\{i,i-1\}$, then, in each of the three possible cases, the assertion of the lemma is an easy consequence
of the abacus combinatorics in \ref{noth colour weight 2} and our above observation. Thus, since 
$\lambda\notin\{\alpha,\beta,\gamma\}=\{\langle i,i\rangle,\langle i-1,i\rangle, \langle i-1\rangle\}$, it remains to treat the following four cases, where we 
draw runners $i-1$ and $i$,  and where $i-1\neq l\neq i$. 
As well,
$a,b,c,d$ are the numbers of beads in the respective parts of the abacus.

\smallskip

\begin{center}
\begin{tabular}{|c|c|c|}\hline
case & $\lambda$ & $\bar{\lambda}$\\\hline\hline
(a) & $\langle i-1,i-1\rangle$ &$\langle i,i\rangle$\\\hline
&\begin{tabular}{cccc}
$\cdots$&$\bullet$&$\bullet$& $\cdots$\\
$\vdots$&$\vdots$&$\vdots$&$\vdots$\\
$\cdots$& $\bullet$&$\bullet$&$\cdots$\\
$\cdots$& $-$&$\bullet$&{\cellcolor{lightgray} $c$}\\
{\cellcolor{lightgray} $a$} & $\bullet$&$\bullet$&{\cellcolor{gray} $d$}\\
{\cellcolor{gray} $b$}&$\bullet$&$\bullet$&$\cdots$
\end{tabular}
&
\begin{tabular}{cccc}
$\cdots$&$\bullet$&$\bullet$& $\cdots$\\
$\vdots$&$\vdots$&$\vdots$&$\vdots$\\
$\cdots$& $\bullet$&$\bullet$&$\cdots$\\
$\cdots$& $\bullet$&$-$&{\cellcolor{lightgray} $c$}\\
{\cellcolor{lightgray} $a$} & $\bullet$&$\bullet$&{\cellcolor{gray} $d$}\\
{\cellcolor{gray} $b$}&$\bullet$&$\bullet$&$\cdots$
\end{tabular}\\\hline
\end{tabular}

\begin{tabular}{|c|c|c|}\hline
(b) & $\langle i\rangle$ &$\langle i-1\rangle$\\\hline
&
\begin{tabular}{cccc}
$\cdots$&$\bullet$&$\bullet$& $\cdots$\\
$\vdots$&$\vdots$&$\vdots$&$\vdots$\\
$\cdots$&$\bullet$&$\bullet$&$\cdots$\\
$\cdots$& $-$&$-$&{\cellcolor{lightgray} $c$}\\
{\cellcolor{lightgray} $a$} & $-$&$-$&{\cellcolor{gray} $d$}\\
{\cellcolor{gray} $b$}&$-$&$\bullet$&$\cdots$
\end{tabular}
&
\begin{tabular}{cccc}
$\cdots$&$\bullet$&$\bullet$& $\cdots$\\
$\vdots$&$\vdots$&$\vdots$&$\vdots$\\
$\cdots$&$\bullet$&$\bullet$&$\cdots$\\
$\cdots$& $-$&$-$&{\cellcolor{lightgray} $c$}\\
{\cellcolor{lightgray} $a$} & $-$&$-$&{\cellcolor{gray} $d$}\\
{\cellcolor{gray} $b$}&$\bullet$&$-$&$\cdots$
\end{tabular}\\\hline
\end{tabular}

\begin{tabular}{|c|c|c|}\hline
(c) & $\langle i-1,l\rangle$ &$\langle i,l\rangle$\\\hline
&\begin{tabular}{cccc}
$\cdots$&$\bullet$&$\bullet$& $\cdots$\\
$\vdots$&$\vdots$&$\vdots$&$\vdots$\\
$\cdots$& $\bullet$&$\bullet$&{\cellcolor{lightgray} $c$}\\
{\cellcolor{lightgray} $a$} & $-$&$\bullet$&{\cellcolor{gray} $d$}\\
{\cellcolor{gray} $b$}&$\bullet$&$\bullet$&$\cdots$
\end{tabular}
&
\begin{tabular}{cccc}
$\cdots$&$\bullet$&$\bullet$& $\cdots$\\
$\vdots$&$\vdots$&$\vdots$&$\vdots$\\
$\cdots$& $\bullet$&$\bullet$&{\cellcolor{lightgray} $c$}\\
{\cellcolor{lightgray} $a$} & $\bullet$&$-$&{\cellcolor{gray} $d$}\\
{\cellcolor{gray} $b$}&$\bullet$&$\bullet$&$\cdots$
\end{tabular}\\\hline\hline
(d) & $\langle i,l\rangle$ &$\langle i-1,l\rangle$\\\hline
&\begin{tabular}{cccc}
$\cdots$&$\bullet$&$\bullet$& $\cdots$\\
$\vdots$&$\vdots$&$\vdots$&$\vdots$\\
$\cdots$& $\bullet$&$\bullet$&{\cellcolor{lightgray} $c$}\\
{\cellcolor{lightgray} $a$} & $-$&$-$&{\cellcolor{gray} $d$}\\
{\cellcolor{gray} $b$}&$-$&$\bullet$&$\cdots$
\end{tabular}
&
\begin{tabular}{cccc}
$\cdots$&$\bullet$&$\bullet$& $\cdots$\\
$\vdots$&$\vdots$&$\vdots$&$\vdots$\\
$\cdots$& $\bullet$&$\bullet$&{\cellcolor{lightgray} $c$}\\
{\cellcolor{lightgray} $a$} & $-$&$-$&{\cellcolor{gray} $d$}\\
{\cellcolor{gray} $b$}&$\bullet$&$-$&$\cdots$
\end{tabular}\\\hline
\end{tabular}
\end{center}

\smallskip

We use \ref{noth hooks diagram}.
In case (a), we thus get $\partial(\lambda)=|(a+c+1)-(b+d+1)|=\partial(\bar{\lambda})$.
The hook diagram of $\lambda$ and $\bar{\lambda}$, respectively, has an entry equal to $2p$.
 If $\partial(\lambda)=0$,
then the leg length of the $2p$-hook of $\lambda$ and  $\bar{\lambda}$ is $a+b+c+d+3$, so that $\lambda$ and $\bar{\lambda}$ have the same colour, by \ref{noth partial}.

In case (b), we get $\partial(\lambda)=|(b+d)-(a+c)|=\partial(\bar{\lambda})$. The hook diagram of $\lambda$ and $\bar{\lambda}$, respectively, has an entry equal to $2p$.
 If $\partial(\lambda)=0$,
then the leg length of the $2p$-hook of $\lambda$ and  $\bar{\lambda}$ is $a+b+c+d$. So
$\lambda$ and $\bar{\lambda}$ have the same colour, by \ref{noth partial}.

In case (c), $\lambda$ and $\bar{\lambda}$ both have a movable bead on runner $l$. The number $m$ of beads passed when moving
this bead one position up is the same for $\lambda$ and $\bar{\lambda}$, by our initial observation.
Moreover, the hook diagrams of $\lambda$ and $\bar{\lambda}$ have two entries equal to $p$. The
leg lengths of the 
corresponding hooks are $b+d+1$ and $m$, both for $\lambda$ and $\bar{\lambda}$.
As for the $\partial$-values, we may first move the bead on runner $i-1$ of $\Gamma_\lambda$ one position up.
The leg length of the corresponding (rim) $p$-hook equals $b+d+1$. Then we move the movable bead on runner $l$ one position up.
The leg length of the corresponding (rim) $p$-hook equals $m'$, for some $m'\geq 0$. So we have $\partial(\lambda)=|(b+d+1)-m'|$.
On the other hand, we first move the bead on runner $i$ of $\Gamma_{\bar{\lambda}}$ one position up.
The leg length of the corresponding (rim) $p$-hook equals $b+d+1$.
Then we move the movable bead on runner $l$ one position up.
The leg length of the corresponding rim $p$-hook equals $m'$ again. Thus also 
$\partial(\bar{\lambda})=|(b+d+1)-m'|$.
Hence
we have $\partial(\lambda)=\partial(\bar{\lambda})$ and if $\partial(\lambda)=0$, then
$\lambda$ and $\bar{\lambda}$ have the same colour.

In case (d), $\lambda$ and $\bar{\lambda}$ both have a movable bead on runner $l$. The number $m$ of beads passed when moving
this bead one position up is the same for $\lambda$ and $\bar{\lambda}$, by our initial observation.
Moreover, the hook diagrams of $\lambda$ and $\bar{\lambda}$ have two entries equal to $p$. The leg length of the corresponding rim $p$-hook equals $b+d$. Then we move the bead on runner $l$ one position up, and the leg lengths of the
corresponding hooks are $b+d$ and $m$, both for $\lambda$ and $\bar{\lambda}$.
To determine the $\partial$-values, we proceed as in case (3), that is, here we first move the bead on runner $i$ and then the bead on runner $l$ of
$\Gamma_\lambda$ one position up. 
Analogously, we first move the bead on runner $i-1$  and then the bead on runner $l$ of $\Gamma_{\bar{\lambda}}$ one position up.
Then there is some $m'\geq 0$ such that
$\partial(\lambda)=|(b+d)-m'|=\partial(\bar{\lambda})$ and if $\partial(\lambda)=0$, then
$\lambda$ and $\bar{\lambda}$ have the same colour.

The assertion concerning $\bar{\beta}$ and $\gamma$ has been established in Lemma~\ref{lemma l and r}. 
This completes the proof of the lemma.
\end{proof}

\begin{noth}\label{noth Scopes}
{\bf Scopes equivalences and $\partial$-values.}\,
Suppose that $k\geq 2$ and that $B$ and $\bar{B}$ are blocks of $F\mathfrak{S}_n$ and $F\mathfrak{S}_{n-k}$, respectively, of weight $2$ such that
$(B,\bar{B})$ is a $(2:k)$-pair. Let $\kappa_B=(\kappa_1,\ldots,\kappa_t)$, and display $\kappa_B,\kappa_{\bar{B}}$ as well as all partitions
of $B$ and $\bar{B}$ on an $[m_1,\ldots,m_p]$-abacus with $s+2p$ beads, for some $s\geq t$. By \cite{Scopes1991}, we may choose $s$ such that there is some $i>1$ such that $\Gamma_{B}$ has 
$k$ more beads on runner $i$ than on runner $i-1$, and $\Gamma_{\bar{B}}$ is obtained by interchanging runners $i$ and $i-1$ of $\Gamma_B$.
Moreover, by \cite{Scopes1991}, the blocks $B$ and $\bar{B}$ are Morita equivalent. Thus, in particular, there is a bijection
between the isomorphism classes of simple $B$-modules and the isomorphism classes of simple $\bar{B}$-modules.
As Scopes has also shown in \cite{Scopes1991}, this bijection can be described combinatorially.
More precisely,
there is a bijection, say $\Psi$, between the set of partitions of $B$ and the set of partitions of $\bar{B}$ that preserves $p$-regularity and
the lexicographic ordering. Whenever $\lambda$ is a partition of $B$ with abacus display $\Gamma_\lambda$, one obtains 
$\Gamma_{\Psi(\lambda)}$ by interchanging runners $i$ and $i-1$. If $\lambda$ is $p$-regular, then
the simple $\bar{B}$-module $D^{\Psi(\lambda)}$ is the Morita correspondent of $D^{\lambda}$.

We should like to emphasize that, for every partition $\lambda$ of $B$, whenever there
is some bead on runner $i-1$ of $\Gamma_\lambda$, there is also a bead on runner $i$ in the same row. This is due to
the fact that $k\geq 2$. With this, the arguments
used in the proof of Lemma~\ref{lemma colour good}, easily generalize and show that $\partial(\lambda)=\partial(\Psi(\lambda))$, for
every partition $\lambda$ of $B$. Moreover, if $\partial(\lambda)=0$, then $\lambda$ and $\Psi(\lambda)$ have the same colour.
\end{noth}

\begin{noth}\label{noth partial Scopes}
{\bf Partial Scopes equivalences and Ext-quivers.}\, 
Suppose that $B$ and $\bar{B}$ are blocks of $F\mathfrak{S}_n$ and $F\mathfrak{S}_{n-1}$, respectively, of weight $2$ such that
$(B,\bar{B})$ is a $(2:1)$-pair. As in \ref{noth (2:1)}, we denote by $\bar{\alpha},\bar{\beta},\bar{\gamma}$ and
$\alpha$, $\beta$, $\gamma$ the exceptional partitions of $B$ and $\bar{B}$, respectively. Moreover, retain the notation fixed in
\ref{noth (2:1)}. 

Furthermore, consider the pair of exact, two-sided adjoint functors
\begin{align*}
\res_{\bar{B}}^B&: B-\mathbf{mod}\to \bar{B}-\mathbf{mod}\,, M\mapsto M\downarrow_{\bar{B}}\,,\\
\ind_{\bar{B}}^B&: \bar{B}-\mathbf{mod}\to B-\mathbf{mod}\,, N\mapsto N\uparrow^B\,
\end{align*}
as in \ref{noth block nota}(b).

Whenever $\lambda$ is a good partition of $B$ and $\mu$ is a good partition of $\bar{B}$, we have
$S^\lambda\downarrow_{\bar{B}}\cong S^{\Phi(\lambda)}$ and $S^\mu\uparrow ^B\cong S^{\Phi^{-1}(\mu)}$,
by \cite[Lemma 3.3]{Scopes1995}.
Moreover, whenever $\lambda$ is a good $p$-regular partition of $B$ and $\mu$ is a good $p$-regular partition of $\bar{B}$, by \cite[Corollary 3.7]{Scopes1995}, we have $D^\lambda\downarrow_{\bar{B}}\cong D^{\Phi(\lambda)}$ and $D^\mu\uparrow ^B\cong D^{\Phi^{-1}(\mu)}$.
Moreover, by Lemma~\ref{lemma l and r} and \cite[Remark 4.4]{Scopes1995}, we also know that $\bar{\gamma}$ is $p$-regular if and only 
if $\beta$ is. If so, then $D^\beta\downarrow_{\bar{B}}\cong D^{\bar{\gamma}}$ and 
$D^{\bar{\gamma}}\uparrow ^B\cong D^{\beta}$. We set $\Phi(\beta):=\bar{\gamma}$, in this case.
Analogously, $\gamma$ is $p$-regular if and only if $\bar{\beta}$ is. If so, then 
$D^\gamma\downarrow_{\bar{B}}\cong D^{\bar{\beta}}$ and 
$D^{\bar{\beta}}\uparrow ^B\cong D^{\gamma}$. We set $\Phi(\gamma):=\bar{\beta}$, in this case.

Next let $\mathcal{M}$ be the full subcategory of $B-\mathbf{mod}$ 
whose objects do not have any composition factor isomorphic to $D^\alpha$. Analogously, let $\mathcal{N}$ be
the full subcategory of $\bar{B}-\mathbf{mod}$ whose objects do not have any composition factor isomorphic to 
$D^{\bar{\alpha}}$. It is well known that the restriction of the functors $\res_{\bar{B}}^B$ and $\ind_{\bar{B}}^B$ 
yield an equivalence between the categories $\mathcal{M}$ and $\mathcal{N}$; for some further explanations see, for instance, \cite[Section 4]{DanzErdmann2012}.
On the combinatorial level, this equivalence entails the following bijection, which generalizes the map $\Phi$ in \ref{noth (2:1)} and which
we denote by $\Phi$ as well:
$$\Phi:\{\lambda\vdash_p n: D^\lambda\in\mathcal{M}\}\to \{\mu\vdash_p n-1: D^\mu\in\mathcal{N}\}\,.$$
It should be emphasized, however, that this latter bijection now preserves the lexicographic ordering only on good partitions, since we always have $\bar{\alpha}>\bar{\beta}>\bar{\gamma}$ as well as $\alpha>\beta>\gamma$.

By \cite[Theorem I]{Scopes1995}, the dimension of the Ext-space of any pair of simple $B$-modules (respectively, $\bar{B}$-modules)
is either $0$ or $1$.
So, since block restriction and block induction are exact functors, for any pair of simple modules $D^{\lambda_1}$ and
$D^{\lambda_2}$ of $\mathcal{M}$, we get an $F$-vector space isomorphism

\begin{align*}
\Ext_{B}^1(D^{\lambda_1},D^{\lambda_2})&\cong \Ext_{\mathcal{M}}^1(D^{\lambda_1},D^{\lambda_2})\cong \Ext_{\mathcal{N}}^1(D^{\Phi(\lambda_1)},D^{\Phi(\lambda_2)})\\
&\cong \Ext_{\bar{B}}^1(D^{\Phi(\lambda_1)},D^{\Phi(\lambda_2)})\,.
\end{align*}

Consequently, suppose  we already know the Ext-quiver of $\bar{B}$. In order to determine the Ext-quiver of $B$, we may in fact proceed as follows:
first remove $\bar{\alpha}$ and all edges connected to $\bar{\alpha}$, and replace every partition $\mu\neq \bar{\alpha}$ of $\bar{B}$ by $\Phi^{-1}(\mu)$.
Then determine those $p$-regular partitions $\lambda$ of $B$ such that
$\Ext_{B}^1(D^{\lambda},D^{\alpha})\neq \{0\}$; recall that the latter Ext-space is at most $1$-dimensional. 
This yields the edges connected to $\alpha$. To this end, we record the following lemma.
\end{noth}

\begin{lemma}\label{lemma Ext alpha}
Suppose that $B$ and $\bar{B}$ are blocks of $F\mathfrak{S}_n$ and $F\mathfrak{S}_{n-1}$, respectively, of weight $2$ such that
$(B,\bar{B})$ is a $(2:1)$-pair. 

\begin{enumerate}
\item[{\rm (a)}] Let $\mu$ be a $p$-regular partition of $n-1$ with $p$-core $\kappa_{\bar{B}}$. Then $\Ext_{\bar{B}}^1(D^\mu,D^{\bar{\alpha}})\neq \{0\}$
if and only if

\begin{enumerate}
\item[\rm{(i)}] either $\mu=\bar{\beta}$,
\item[\rm{(ii)}] or $\mu>\bar{\alpha}$, $[S^{\bar{\alpha}}:D^\mu]\neq 0$ and $|\partial(\bar{\alpha})-\partial(\mu)|=1$.
\end{enumerate}

\item[\rm{(b)}] Let $\lambda$ be a $p$-regular partition of $n$ with $p$-core $\kappa_B$. Then
$\Ext_B^1(D^\lambda,D^{\alpha})\neq \{0\}$
if and only if

\begin{enumerate}
\item[\rm{(i)}] either $\lambda=\beta$,
\item[\rm{(ii)}] or $\lambda>\alpha$, $[S^{\alpha}:D^\lambda]\neq 0$ and $|\partial(\alpha)-\partial(\lambda)|=1$.
\end{enumerate}
\end{enumerate}
\end{lemma}

\begin{proof}
(a)\, Suppose first that $\bar{\alpha}\geq \mu$. Then, by \cite[Theorem 6.1]{ChuangTan2001}, we have
$\Ext_{\bar{B}}^1(D^\mu,D^{\bar{\alpha}})\neq \{0\}$ if and only if $[S^\mu:D^{\bar{\alpha}}]\neq 0$ and $|\partial(\bar{\alpha})-\partial(\mu)|=1$.
By \cite[Lemma 4.3]{Scopes1995}, we further know that the only Specht $F\mathfrak{S}_{n-1}$-modules with a composition factor isomorphic to $D^{\bar{\alpha}}$ are $S^{\bar{\alpha}}$, $S^{\bar{\beta}}$ and $S^{\bar{\gamma}}$. Since, by Lemma~\ref{lemma l and r}, we have
$\partial(\bar{\gamma})=\partial(\bar{\alpha})=\partial(\bar{\beta})+1$, we deduce that 
$\Ext_{\bar{B}}^1(D^\mu,D^{\bar{\alpha}})\neq \{0\}$ if and only if $\mu=\bar{\beta}$ and $\bar{\beta}$ is $p$-regular.

So we may suppose that $\mu>\bar{\alpha}$. By \cite[Theorem 6.1]{ChuangTan2001} again, we deduce that
$\Ext_{\bar{B}}^1(D^\mu,D^{\bar{\alpha}})\neq \{0\}$ if and only if $\mu$ satisfies condition (ii).

\smallskip

Analogously one obtains assertion (b), also by \cite[Theorem 6.1]{ChuangTan2001}, \cite[Lemma 4.3]{Scopes1995}
and Lemma~\ref{lemma l and r}.
\end{proof}

Let $(B,\bar{B})$ be a $(2:1)$-pair of blocks of $F\mathfrak{S}_n$ and
$F\mathfrak{S}_{n-1}$, respectively, of weight $2$. 
Suppose that $\kappa_B=(\kappa_1,\ldots,\kappa_t)$. As in Section~\ref{sec weight 2}, we display $\kappa_B$, $\kappa_{\bar{B}}$ as
well as all partitions of $B$ and $\bar{B}$ on a fixed $[m_1,\ldots,m_p]$-abacus with $s+2p$ beads, where $s\geq t$. 
As in \ref{noth (2:1)}, we denote by $\bar{\alpha},\bar{\beta},\bar{\gamma},\alpha,\beta,\gamma$ the exceptional partitions
associated to the $(2:1)$-pair $(B,\bar{B})$. Lastly, recall the notation introduced in Remark~\ref{rem l and r} and \ref{noth partial}. 
The main aim of the next theorem is to give detailed information on the Loewy structures of the Specht modules of
$B$ and $\bar{B}$, respectively, labelled by the exceptional partitions. This will be the crucial ingredient for our
inductive proof of Theorem~\ref{thm main1}. It will turn out that this heavily depends on
whether the exceptional partitions are $p$-regular or $p$-restricted.  To this end, we shall again distinguish the cases (1)--(6) 
as in Remark~\ref{rem l and r}.
For ease of notation, in Theorem~\ref{thm Loewy except} below, we shall often
identify simple $B$-modules and simple $\bar{B}$-modules with their labelling partitions.

We expect the Loewy structures of the Specht modules treated in Theorem~\ref{thm Loewy except} to be known, but we have not seen them in print so far.

\begin{thm}\label{thm Loewy except}
With the above notation,  the Specht $F\mathfrak{S}_{n-1}$-modules $S^{\alpha},S^{\beta}$ and $S^{\gamma}$
and the Specht $F\mathfrak{S}_n$-modules $S^{\bar{\alpha}},S^{\bar{\beta}}$ and $S^{\bar{\gamma}}$, respectively,
have the following Loewy structures.

\medskip

\begin{center}
\begin{tabular}{|c||c|c|c||c|c|c|}\hline
case & $S^{\bar{\alpha}}$ & $S^{\bar{\beta}}$& $S^{\bar{\gamma}}$&$S^\alpha$ &$S^\beta$& $S^\gamma$\\\hline\hline
{\rm (1)}& $\begin{matrix} \bar{\alpha}\\\bar{Z}\\ \\\end{matrix}$ & $\begin{matrix} \bar{\beta}\\\bar{\alpha}\oplus\bar{Y}\\ \\\end{matrix}$ &$\begin{matrix}\bar{\gamma}\\\bar{\beta}\oplus\bar{Z}\\\bar{\alpha}\end{matrix}$ &   $\begin{matrix} \alpha\\ Y\\ \\\end{matrix}$ & $\begin{matrix} \beta \\\alpha\oplus Z\\ \\\end{matrix}$ &$\begin{matrix}\gamma\\\beta\oplus Y\\\alpha\end{matrix}$   \\\hline\hline
{\rm (2)} & $\begin{matrix} \bar{\alpha}\\\bar{\beta}_+\oplus \bar{Z}\\ \\\end{matrix}$ &$\begin{matrix} \bar{\beta}\\\bar{\alpha}\oplus\bar{Y}\\ \bar{\beta}_+\end{matrix}$ &$\begin{matrix}\bar{\gamma}\\\bar{\beta}\oplus\bar{Z}\\\bar{\alpha}\end{matrix}$&$\begin{matrix} \alpha\\ Y\\ \alpha_+\\\end{matrix}$ &$\begin{matrix} \beta\\\alpha\oplus Z\\ \\\end{matrix}$ &$\begin{matrix}\gamma\\\beta\oplus Y\\\alpha\end{matrix}$\\\hline\hline
{\rm (3)} & $\begin{matrix} \bar{\alpha}\\\bar{\beta}_+\\ \\\end{matrix}$ &$\begin{matrix} \bar{\beta}\\\bar{\alpha}\oplus\bar{Y}\\ \bar{\beta}_+\end{matrix}$ & $\begin{matrix}\\ \bar{\beta}\\\bar{\alpha}\end{matrix}$& $\begin{matrix}\alpha\\ Y\\ \alpha_+\end{matrix}$&$\begin{matrix} \alpha\end{matrix}$&$\begin{matrix}\gamma\\ Y \\ \alpha\end{matrix}$\\\hline\hline
{\rm (4)} &$\begin{matrix} \bar{\alpha}\\\bar{\beta}_+\oplus \bar{Z}\\\bar{\alpha}_+\end{matrix}$& $\begin{matrix} \bar{\beta}\\\bar{\alpha}\oplus\bar{Y}\\ \bar{\beta}_+\end{matrix}$ & $\begin{matrix}\\ \bar{\beta}\oplus\bar{Z}\\\bar{\alpha}\end{matrix}$& $\begin{matrix} \alpha\\\beta_+\oplus Y\\ \alpha_+\end{matrix}$& $\begin{matrix} \\\alpha\oplus Z\\ \beta_+\end{matrix}$ & $\begin{matrix}\gamma \\ Y\\\alpha\end{matrix}$ \\\hline
\end{tabular}

\begin{tabular}{|c||c|c|c||c|c|c|}\hline
{\rm (5)}& $\begin{matrix} \bar{\alpha}\\\bar{\beta}_+\oplus \bar{Z}\\\bar{\alpha}_+\end{matrix}$ & $ \begin{matrix} \\\bar{\alpha}\oplus\bar{Y}\\ \bar{\beta}_+\end{matrix}$ &$\begin{matrix}\\ \bar{Z}\\\bar{\alpha}\end{matrix}$&$\begin{matrix} \alpha\\\beta_+\oplus Y\\\alpha_+\end{matrix}$ & $ \begin{matrix} \\\alpha\oplus Z\\ \beta_+\end{matrix}$ &$\begin{matrix}\\ Y\\\alpha\end{matrix}$\\\hline\hline
{\rm (6)} & $\begin{matrix} \bar{\alpha}\\\bar{\beta}_+\oplus \bar{Z}\\\bar{\alpha}_+\end{matrix}$ & $ \begin{matrix} \bar{\beta}\\\bar{\alpha}\oplus\bar{Y}\\ \bar{\beta}_+\end{matrix}$ &$\begin{matrix}\bar{\gamma}\\\bar{\beta}\oplus\bar{Z}\\\bar{\alpha}\end{matrix}$&
$\begin{matrix} \alpha\\\beta_+\oplus Y\\\alpha_+\end{matrix}$ & $ \begin{matrix} \beta\\\alpha\oplus Z\\ \beta_+\end{matrix}$ &$\begin{matrix}\gamma\\\beta\oplus Y\\\alpha\end{matrix}$\\\hline
\end{tabular}
\end{center}

Here $\bar{Y}$ and $\bar{Z}$ are good semisimple $\bar{B}$-modules, and $Y$ and $Z$ are good semisimple $B$-modules. Let $d:=\partial(\alpha)$.
If $D^\mu$ is a composition factor of $\bar{Y}$ and $D^\rho$ is a composition factor of $\bar{Z}$, then
$\partial(\mu)\in\{d-1,d+1\}$ and  $\partial(\rho)\in \{d,d+2\}$.
Moreover, $Y\cong\bar{Y}\uparrow^B$ and
$Z\cong\bar{Z}\uparrow^B$.
If $D^\lambda$ is a composition factor of $Y$ and $D^\nu$ is a composition factor of $Z$, then
$\partial(\lambda)\in\{d-1,d+1\}$ and  $\partial(\nu)\in \{d,d+2\}$.

The partition $\bar{\alpha}_+$ exists if and only if $\beta_+$ does; if so, then both partitions are good 
and
$\bar{\alpha}_+=\Phi(\beta_+)$. The partition $\bar{\beta}_+$ exists if and only if $\alpha_+$ does; if so, then both partitions are good and
$\bar{\beta}_+=\Phi(\alpha_+)$. In the cases {\rm (4), (5)} and {\rm (6)}, one also has $[\bar{Y}:D^{\bar{\alpha}_+}]=0$.
\end{thm}

\begin{noth}\label{noth ingredients}{\bf Strategy of proof.}\,
Before proving Theorem~\ref{thm Loewy except}, we shall collect a number of important properties of
the blocks $B$ and $\bar{B}$ and their modules that we shall use extensively. 

\smallskip

(a)\, By \cite[Theorem I]{Scopes1995}, every principal indecomposable $B$-module (respectively, $\bar{B}$-module)
has Loewy length 5 and is \textit{stable}, that is, its Loewy and socle series coincide. Moreover, all Specht modules in $B$ and $\bar{B}$, respectively,
are multiplicity-free. The $\Ext^1$-space between any two simple $B$-modules (respectively, simple $\bar{B}$-modules) is at most one-dimensional,
and there are no self-extensions of any simple $B$-module (respectively, simple $\bar{B}$-module).

As well, by \cite[Remark 4.4]{Scopes1995}, the projective cover $P^\alpha$ of $D^\alpha$ has a Specht filtration with quotients, from top 
to bottom, isomorphic to $S^\alpha$, $S^\beta$ and $S^\gamma$. Analogously, the projective cover $P^{\bar{\alpha}}$ of $D^{\bar{\alpha}}$
has a Specht filtration with quotients, from top 
to bottom, isomorphic to $S^{\bar{\alpha}}$, $S^{\bar{\beta}}$ and $S^{\bar{\gamma}}$. 

\smallskip

(b)\, By \cite[Theorem~6.1]{ChuangTan2001}, the Ext-quiver of $B$ is bipartite; more precisely, if $D^\lambda$ and $D^\mu$ are
simple $B$-modules with $\Ext^1_B(D^\lambda,D^\mu)\neq \{0\}$, then $\partial(\lambda)\in\{\partial(\mu)-1,\partial(\mu)+1\}$. The analogous
statement holds for $\bar{B}$.

\smallskip

(c)\, By \cite[Proposition 6.2]{ChuangTan2001}, every Specht module $S^\lambda$ in $B$ (respectively, in $\bar{B}$) has
Loewy length at most 3, and has Loewy length 3 if and only if $\lambda$ is both $p$-regular and $p$-restricted.

\smallskip

(d)\, By \cite[Lemma 3.3]{Scopes1995}, one has
$$(S^{\bar{\alpha}})\uparrow^B\sim S^\alpha\oplus S^\beta\,, \quad (S^{\bar{\beta}})\uparrow^B\sim S^\alpha\oplus S^\gamma\;, \quad (S^{\bar{\gamma}})\uparrow^B\sim S^\beta\oplus S^\gamma\,.$$

\smallskip

(e)\, Our general strategy towards proving the assertions of Theorem~\ref{thm Loewy except} will be as follows: with (a)-(c), we shall already
deduce the Loewy lengths and most of the Loewy structure of the Specht modules in question. Part (d) will then provide us with 
systems of linear equations from which we shall 
obtain the claimed connections between the composition factors of the exceptional Specht $\bar{B}$-modules and
those of the exceptional Specht $B$-modules.
\end{noth}

\begin{proof}(of Theorem~\ref{thm Loewy except})
We prove the assertions case by case, starting with case (6), which is in some sense the most general one. Throughout this proof, let
$d:=\partial(\alpha)=\partial(\bar{\beta})$.
Recall from (\ref{eqn Loewy}) in \ref{noth block nota} our notation for Loewy structures.

\medskip

\underline{Case (6):} By Lemma~\ref{lemma l and r}, we know that all exceptional partitions of $B$ and $\bar{B}$, respectively,
are both $p$-regular and $p$-restricted. By \ref{noth ingredients}(c), the Specht modules
labelled by the exceptional partitions must, thus, have Loewy length 3. We examine $S^{\bar{\alpha}}$, $S^{\bar{\beta}}$ and
$S^{\bar{\gamma}}$ first. 
By \ref{noth part}(a) and \ref{noth partial}, we know that $\Soc(S^{\bar{\alpha}})\cong D^{\bar{\alpha}_+}$,
 $\Soc(S^{\bar{\beta}})\cong D^{\bar{\beta}_+}$ and $\Soc(S^{\bar{\gamma}})\cong D^{\bar{\gamma}_+}$. Hence
$$S^{\bar{\alpha}}\ \approx \ \begin{matrix}D^{\bar{\alpha}}\\ H^{\bar{\alpha}}\\ D^{\bar{\alpha}_+}\end{matrix}\,,\quad S^{\bar{\beta}}\ \approx \  \begin{matrix}D^{\bar{\beta}}\\ H^{\bar{\beta}}\\ D^{\bar{\beta}_+}\end{matrix}\,,\quad S^{\bar{\gamma}}\ \approx \ \begin{matrix}D^{\bar{\gamma}}\\ H^{\bar{\gamma}}\\ D^{\bar{\gamma}_+}\end{matrix}\,,$$
for non-zero semisimple $\bar{B}$-modules $H^{\bar{\alpha}}, H^{\bar{\beta}}$ and $H^{\bar{\gamma}}$.
By \ref{noth ingredients}(a), the projective cover $P^{\bar{\alpha}}$ of $D^{\bar{\alpha}}$ has a Specht filtration with quotients, from top to bottom, isomorphic to $S^{\bar{\alpha}}, S^{\bar{\beta}}$ and $S^{\bar{\gamma}}$. Hence, in particular, 
$D^{\bar{\gamma}_+}\cong \Soc(S^{\bar{\gamma}})\cong \Soc(P^{\bar{\alpha}})\cong D^{\bar{\alpha}}$ and
$\bar{\gamma}_+=\bar{\alpha}$. Since $P^{\bar{\alpha}}$ is stable, we further deduce that $D^{\bar{\gamma}}$ is isomorphic to a 
submodule of $\Rad^2(P^{\bar{\alpha}})/\Rad^3(P^{\bar{\alpha}})$, and $H^{\bar{\gamma}}$ is isomorphic to a submodule of $\Rad^3(P^{\bar{\alpha}})/\Rad^4(P^{\bar{\alpha}})$.
Analogously, we conclude that $H^{\bar{\alpha}}$ is isomorphic to a submodule of the second Loewy layer
and $D^{\bar{\alpha}_+}$ is isomorphic to a submodule of the third Loewy layer of $P^{\bar{\alpha}}$.
Next, by \cite[Theorem 6.1]{ChuangTan2001} and Lemma~\ref{lemma Ext alpha}, we know that $\dim(\Ext^1_{\bar{B}}(D^{\bar{\alpha}},D^{\bar{\beta}}))=1$ and $[S^{\bar{\beta}}:D^{\bar{\alpha}}]=1$. Since $S^{\bar{\beta}}$ is multiplicity-free, this implies
$H^{\bar{\beta}}\cong D^{\bar{\alpha}}\oplus \bar{Y}$, for some good semisimple $\bar{B}$-module $\bar{Y}$. Again using the fact that the quiver of $\bar{B}$ is bipartite, we so far get
\begin{equation}\label{eqn Palpha}
P^{\bar{\alpha}}\; \ \approx\; \  \begin{matrix}D^{\bar{\alpha}}\\H^{\bar{\alpha}}\oplus D^{\bar{\beta}}\\ D^{\bar{\alpha}_+}\oplus H^{\bar{\beta}}\oplus D^{\bar{\gamma}}\\ D^{\bar{\beta}_+}\oplus H^{\bar{\gamma}}\\ D^{\bar{\gamma}_+}\end{matrix}\; \ \approx\; \  \begin{matrix}D^{\bar{\alpha}}\\H^{\bar{\alpha}}\oplus D^{\bar{\beta}}\\ D^{\bar{\alpha}_+}\oplus D^{\bar{\alpha}}\oplus \bar{Y}\oplus D^{\bar{\gamma}}\\ D^{\bar{\beta}_+}\oplus H^{\bar{\gamma}}\\ D^{\bar{\alpha}}\end{matrix}\,.
\end{equation}
By \cite[Lemma 4.3]{Scopes1995}, $[S^{\bar{\gamma}}:D^{\bar{\beta}}]=1=[S^{\bar{\gamma}}:D^{\bar{\alpha}}]$, thus
$H^{\bar{\gamma}}\cong D^{\bar{\beta}}\oplus \bar{Z}$, for some good semisimple $\bar{B}$-module $\bar{Z}$. The assertion concerning the $\partial$-values of
the composition factors of $\bar{Y}$ and $\bar{Z}$ follows from \cite[Theorem 6.1(3)]{ChuangTan2001}.
Since $P^{\bar{\alpha}}$ is stable, this now gives $H^{\bar{\alpha}}\cong D^{\bar{\beta}_+}\oplus \bar{Z}$. Since all Specht $\bar{B}$-modules are multiplicity-free, we also see that $\bar{\alpha}_+$ and $\bar{\beta}_+$ are good. 

It remains to show that $[\bar{Y}:D^{\bar{\alpha}_+}]=0$. Assume not, so that $[S^{\bar{\beta}}:D^{\bar{\alpha}_+}]=1$ and $\bar{\alpha}_+\rhd \bar{\beta}$.
Furthermore, by (\ref{eqn Palpha}), we get $\Ext^1_{\bar{B}}(D^{\bar{\beta}},D^{\bar{\alpha}_+})\neq \{0\}$ 
implying $\bar{\alpha}_+\rhd \bar{\beta}\rhd \bar{\alpha}$, by \cite[Theorem 6.1(3)]{ChuangTan2001}, a contradiction, since $\bar{\alpha}\rhd \bar{\beta}$. This completes the proof concerning the Loewy structure of $S^{\bar{\alpha}}$, $S^{\bar{\beta}}$ and
$S^{\bar{\gamma}}$, and gives
\begin{equation}\label{eqn Salphabar}
S^{\bar{\alpha}}\ \approx \ \begin{matrix}\alpha\\ \bar{\beta}_+\oplus \bar{Z}\\\bar{\alpha}_+\end{matrix}\,,\ \quad S^{\bar{\beta}}\ \approx \ \begin{matrix}\bar{\beta}\\\bar{\alpha}\oplus \bar{Y}\\\bar{\beta}_+\end{matrix}\,,\ \quad S^{\bar{\gamma}}\ \approx \  \begin{matrix}\bar{\gamma}\\ \bar{\beta}\oplus \bar{Z}\\\bar{\alpha}\end{matrix}\,.
\end{equation}

\smallskip

Next we consider $S^\alpha$, $S^{\beta}$ and
$S^{\gamma}$, still in the case (6). By \ref{noth ingredients}(a), the projective cover $P^\alpha$ of $D^\alpha$ has a 
Specht filtration with quotients, from top to bottom, isomorphic to $S^\alpha$, $S^\beta$ and $S^\gamma$. The above arguments now
work completely analogously and give
\begin{equation}\label{eqn Salpha}
S^\alpha\ \approx \ \begin{matrix}\alpha\\ \beta_+\oplus R\\\alpha_+\end{matrix}\,,\ \quad S^\beta\ \approx \ \begin{matrix}\beta\\\alpha\oplus T\\\beta_+\end{matrix}\,,\quad \ S^\gamma \ \approx \  \begin{matrix}\gamma\\ \beta\oplus R\\\alpha\end{matrix}\,,
\end{equation}
for good semisimple $B$-modules $R$ and $T$ that are disjoint. Moreover, $\alpha_+$, $\beta_+$ are good, and 
$[T:D^{\alpha_+}]=0$. 
Every composition factor of $R$ is labelled by a $p$-regular partition with $\partial$-value $d-1$ or $d+1$, every composition 
factor of $T$ is labelled by a $p$-regular partition with $\partial$-value $d$ or $d+2$.
To complete the proof of case (6), we 
need to show that
\begin{equation}\label{eqn R=Y}
R\cong Y\;, T\cong Z\;, (D^{\bar{\alpha}_+})\uparrow^B\cong D^{\beta_+}\;, (D^{\bar{\beta}_+})\uparrow^B\cong D^{\alpha_+}\,.
\end{equation}
To this end, we
exploit the partial Scopes equivalence between $B$ and $\bar{B}$ given by block restriction and block induction, as explained in \ref{noth partial Scopes}. We set $Y:=\bar{Y}\uparrow^B$ and $Z:=\bar{Z}\uparrow^B$.  By \ref{noth partial Scopes}, we know that
$D^{\bar{\beta}}\uparrow^B\cong D^{\gamma}$ and $D^{\bar{\gamma}}\uparrow^B\cong D^\beta$. Therefore, with
\ref{noth ingredients}(d), we obtain the following
\begin{align}\label{eqn upalpha}
S^\alpha\oplus S^\beta &\sim (S^{\bar{\alpha}})\uparrow^B \sim (D^{\bar{\alpha}})\uparrow^B\oplus (D^{\bar{\beta}_+})\uparrow^B\oplus Z\oplus (D^{\bar{\alpha}_+})\uparrow^B\\ \label{eqn upbeta}
 S^\alpha\oplus S^\gamma&\sim  (S^{\bar{\beta}})\uparrow^B\sim D^\gamma\oplus  (D^{\bar{\alpha}})\uparrow^B\oplus (D^{\bar{\beta}_+})\uparrow^B\oplus Y\\ \label{eqn upgamma}
S^\beta\oplus S^\gamma &\sim (S^{\bar{\gamma}})\uparrow^B \sim D^\beta\oplus D^\gamma\oplus Z\oplus (D^{\bar{\alpha}})\uparrow^B\,.
\end{align}
To exploit these identities, recall from Lemma~\ref{lemma l and r} that 
$\partial(\bar{\alpha})=d+1=\partial(\bar{\gamma})=\partial(\beta)$ and $\partial(\bar{\beta})=d=\partial(\alpha)=\partial(\gamma)$. By \ref{noth partial}, we have $\partial(\theta)=\partial(\theta_+)$, for every $p$-restricted partition $\theta$ belonging to $B$ or $\bar{B}$. 
By Lemma~\ref{lemma colour good}, we also know that $\partial(\mu)=\partial(\Phi^{-1}(\mu))$, for every good $p$-regular partition of $\bar{B}$.

Subtracting (\ref{eqn upalpha}) from (\ref{eqn upbeta}) (in the Grothendieck group) and using (\ref{eqn Salpha}), we get
$$T\oplus Y\oplus D^{\beta_+}\sim R\oplus Z\oplus (D^{\bar{\alpha}_+})\uparrow^B\,.$$
As we have proved above, $\bar{Z}\oplus D^{\bar{\alpha}_+}$ and $\bar{Y}$ are disjoint and good. Thus the same holds for $Z\oplus (D^{\bar{\alpha}_+})\uparrow^B$ and $Y$.
As well, $R$ and $T\oplus D^{\beta_+}$ are disjoint. Therefore, we must have $T\oplus D^{\beta_+}\sim Z\oplus (D^{\bar{\alpha}_+})\uparrow^B$
and $R\cong Y$. Comparing the $\partial$-values, we deduce $D^{\beta_+}\cong (D^{\bar{\alpha}_+})\uparrow^B$
and $T\cong Z$. 
It remains to show that $(D^{\bar{\beta}_+})\uparrow^B\cong D^{\alpha_+}$. Subtracting (\ref{eqn upbeta}) from
(\ref{eqn upgamma}) (in the Grothendieck group) and using what we have just proved about $S^\beta$, we get
$S^\alpha\sim D^\alpha\oplus D^{\beta_+}\oplus R\oplus (D^{\bar{\beta}_+})\uparrow^B$. On the other hand, by (\ref{eqn Salpha}),
$S^\alpha\sim D^\alpha\oplus D^{\beta_+}\oplus R\oplus D^{\alpha_+}$, hence $(D^{\bar{\beta}_+})\uparrow^B\cong D^{\alpha_+}$
and $\Phi(\alpha_+)=\bar{\beta}_+$.

This completes the proof of the assertion of the theorem in the case (6).

\bigskip

\underline{Case (2):} By Lemma~\ref{lemma l and r}, the partitions $\bar{\alpha}$ and $\beta$ are $p$-regular and not $p$-restricted.
The remaining exceptional partitions are both $p$-regular and $p$-restricted. Using \ref{noth ingredients} and
arguing as in the proof of case (6) above, we this time deduce
\begin{equation}\label{eqn Specht case (2)}
 S^{\bar{\alpha}}\ \approx \ \begin{matrix} \bar{\alpha}\\\bar{\beta}_+\oplus \bar{Z}\\ \\\end{matrix}\;, \ \quad S^{\bar{\beta}}\ \approx \  \begin{matrix} \bar{\beta}\\\bar{\alpha}\oplus\bar{Y}\\ \bar{\beta}_+\end{matrix}\;, \  \quad S^{\bar{\gamma}}\ \approx \  \begin{matrix}\bar{\gamma}\\\bar{\beta}\oplus\bar{Z}\\\bar{\alpha}\end{matrix}\,
\end{equation}
as well as 
\begin{equation}\label{eqn Specht B case (2)}
 S^{\alpha}\ \approx \  \begin{matrix} \alpha\\ R\\ \alpha_+\end{matrix}\;, \ \quad S^{\beta}\ \approx \  \begin{matrix} \beta\\\alpha\oplus T\\ \\\end{matrix}\;, \ \quad S^{\gamma}\ \approx \ \begin{matrix}\gamma\\\beta\oplus R\\\alpha\end{matrix}\,.
\end{equation}
Here $\bar{Y}$ and $\bar{Z}$ are good, semisimple and disjoint. Moreover, $\bar{\beta}_+$ is good. 
The assertion concerning $\partial$-values follows from \cite[Theorem 6.1(3)]{ChuangTan2001}.
As well, $R$ and $T$ are good, semisimple and disjoint, and $\alpha_+$ is good. Every composition factor of $R$ is labelled by a $p$-regular partition with $\partial$-value $d-1$ or $d+1$, every composition 
factor of $T$ is labelled by a $p$-regular partition with $\partial$-value $d$ or $d+2$.

In the following, set $Y:=\bar{Y}\uparrow^B$, $Z:=\bar{Z}\uparrow^B$ and
$D^\lambda:=(D^{\bar{\beta}_+})\uparrow^B$. 
Again we have $(D^{\bar{\beta}})\uparrow^B\cong D^\gamma$ and $(D^{\bar{\gamma}})\uparrow^B\cong D^\beta$
It remains to show that $Y\sim R$, $Z\sim T$ and $\Phi(\alpha_+)=\bar{\beta}_+$.
From \ref{noth ingredients}(c) and (\ref{eqn Specht B case (2)}) we get

\begin{align}\label{eqn upalpha 2}
S^\alpha\oplus S^\beta&\sim (S^{\bar{\alpha}})\uparrow^B\sim (D^{\bar{\alpha}})\uparrow^B\oplus (D^{\bar{\beta}_+})\uparrow^B\oplus Z\,\\ \label{eqn upbeta 2}
S^\alpha\oplus S^\gamma&\sim (S^{\bar{\beta}})\uparrow^B\sim D^\gamma\oplus (D^{\bar{\alpha}})\uparrow^B\oplus Y\oplus (D^{\bar{\beta}_+})\uparrow^B\, \\ \label{eqn upgamma 2}
S^\beta\oplus S^\gamma& \sim (S^{\bar{\gamma}})\uparrow^B\sim D^\beta\oplus D^\gamma\oplus Z\oplus  (D^{\bar{\alpha}})\uparrow^B\,.
\end{align}
Subtracting (\ref{eqn upalpha 2}) from (\ref{eqn upbeta 2}) in the Grothendieck group and using (\ref{eqn Specht case (2)}) and 
(\ref{eqn Specht B case (2)}), we see that $Y\oplus T\sim R\oplus Z$. Since $\bar{Y}$ and $\bar{Z}$ are disjoint, so are $Y$ and $Z$.
Since also $R$ and $T$ are disjoint, we must have $T\sim Z$ and $Y\sim R$. 
Lastly, we subtract (\ref{eqn upbeta 2}) from (\ref{eqn upgamma 2}). Together with what we have just shown and (\ref{eqn Salpha})
this gives $D^{\alpha_+}\oplus Y\sim D^\beta\oplus Z\sim (D^{\bar{\beta}_+})\uparrow^B\oplus Y$, 
thus $D^{\alpha_+}\cong (D^{\bar{\beta}_+})\uparrow^B$ and $\Phi(\alpha_+)=\bar{\beta}_+$.

\medskip

\underline{Case (4):} By Lemma~\ref{lemma l and r}, the partitions $\beta$ and $\bar{\gamma}$ are $p$-restricted, but not $p$-regular.
The remaining exceptional partitions are both $p$-regular and $p$-restricted. From \ref{noth ingredients}(a)-(c) we get
\begin{equation}\label{eqn Specht case (4)}
 S^{\bar{\alpha}}\ \approx \ \begin{matrix} \bar{\alpha}\\\bar{\beta}_+\oplus \bar{Z}\\ \bar{\alpha}_+\end{matrix}\;, \ \quad S^{\bar{\beta}}\ \approx \ \begin{matrix} \bar{\beta}\\\bar{\alpha}\oplus\bar{Y}\\ \bar{\beta}_+\end{matrix}\;,\  \quad S^{\bar{\gamma}}\ \approx \  \begin{matrix}\\\bar{\beta}\oplus\bar{Z}\\\bar{\alpha}\end{matrix}\,
\end{equation}
and 
\begin{equation}\label{eqn Specht B case (4)}
 S^{\alpha}\ \approx \ \begin{matrix} \alpha\\ R\oplus \beta_+\\ \alpha_+\end{matrix}\;, \ \quad S^{\beta}\ \approx \ \begin{matrix} \\\alpha\oplus T\\ \beta_+\end{matrix}\;, \ \quad S^{\gamma}\ \approx \  \begin{matrix}\gamma\\ R\\\alpha\end{matrix}\,.
\end{equation}
Here $\bar{Y}$ and $\bar{Z}$, $R$ and $T$ are good and semisimple. Moreover, $\bar{\alpha}_+$, $\bar{\beta}_+$,  $\alpha_+$ and $\beta_+$
are good. As in the proof of case (6) above, we see that $[\bar{Y}:D^{\bar{\alpha}_+}]=0$.
The assertion concerning the $\partial$-values of the composition factors of $\bar{Y}$ and $\bar{Z}$ follows from \cite[Theorem 6.1(3)]{ChuangTan2001}.
Since $S^\gamma$ has Loewy length $3$, we 
further have $R\neq \{0\}$. 
Let $\lambda:=\Phi^{-1}(\bar{\beta}_+)$, $\mu:=\Phi^{-1}(\bar{\alpha}_+)$, $Y:=\bar{Y}\uparrow^B$ and $Z:=\bar{Z}\uparrow^B$.
Then from \ref{noth ingredients}(d) we get
\begin{align}\label{eqn upalpha 4}
S^\alpha\oplus S^\beta&\sim (S^{\bar{\alpha}})\uparrow^B\sim (D^{\bar{\alpha}})\uparrow^B\oplus D^\lambda\oplus D^\mu\oplus Z\,\\ \label{eqn upbeta 4}
S^\alpha\oplus S^\gamma&\sim (S^{\bar{\beta}})\uparrow^B\sim D^\gamma\oplus (D^{\bar{\alpha}})\uparrow^B\oplus Y\oplus D^\lambda\, \\ \label{eqn upgamma 4}
S^\beta\oplus S^\gamma& \sim (S^{\bar{\gamma}})\uparrow^B\sim D^\gamma\oplus Z\oplus  (D^{\bar{\alpha}})\uparrow^B\,.
\end{align}
We subtract (\ref{eqn upalpha 4}) from (\ref{eqn upbeta 4}) and use (\ref{eqn Specht case (4)}) and (\ref{eqn Specht B case (4)}) to 
get $Y\oplus T\oplus D^{\beta_+}\sim R\oplus Z\oplus D^\mu$. By Lemma~\ref{lemma l and r} and Lemma~\ref{lemma colour good}, we 
have $\partial(\beta_+)=d+1=\partial(\mu)$. 
By \cite[Theorem 6.1(3)]{ChuangTan2001} and Lemma~\ref{lemma l and r} again,
every composition factor of $R$ has a labelling partition with $\partial$-value $d-1$ or $d+1$; 
every composition factor of $T$ has a labelling partition with $\partial$-value $d$ or $d+2$. This implies $Z\cong T$ and $Y\oplus D^{\beta_+}\sim R\oplus D^\mu$. Next, consider the difference (\ref{eqn upgamma 4})-(\ref{eqn upbeta 4}). Together with (\ref{eqn Specht B case (4)})
this gives $T\oplus Y\oplus D^\lambda\sim Z\oplus R\oplus D^{\alpha_+}$, hence $Y\oplus D^\lambda\sim R\oplus D^{\alpha_+}$. By comparing 
$\partial$-values, we deduce $Y\cong R$ and $\lambda=\alpha$, and then also $\beta_+=\mu$. This completes the proof
in the case (4).

\medskip

\underline{Case (1):} In this case, by Lemma~\ref{lemma l and r}, the partitions $\bar{\alpha}$, $\bar{\beta}$, $\alpha$ and $\beta$ are
$p$-regular and not $p$-restricted, while $\bar{\gamma}$ and $\gamma$ are both $p$-regular and $p$-restricted. 
With \ref{noth ingredients}(a)-(c) we deduce
\begin{equation}\label{eqn Specht case (1)}
 S^{\bar{\alpha}}\ \approx \  \begin{matrix} \bar{\alpha}\\\bar{Z}\\ \\\end{matrix}\;, \ \quad S^{\bar{\beta}}\ \approx \  \begin{matrix} \bar{\beta}\\\bar{\alpha}\oplus\bar{Y}\\ \\\end{matrix}\;, \quad S^{\bar{\gamma}}\ \approx \  \begin{matrix}\bar{\gamma}\\\bar{\beta}\oplus\bar{Z}\\\bar{\alpha}\end{matrix}\,
\end{equation}
and 
\begin{equation}\label{eqn Specht B case (1)}
 S^{\alpha}\ \approx \  \begin{matrix} \alpha\\ R\\ \\\end{matrix}\;, \ \quad S^{\beta}\ \approx \  \begin{matrix} \beta\\\alpha\oplus T\\ \\\end{matrix}\;, \  \quad S^{\gamma}\ \approx \ \begin{matrix}\gamma\\ \beta\oplus R\\\alpha\end{matrix}\,.
\end{equation}
Here $\bar{Y}$, $\bar{Z}$, $R$ and $T$ are semisimple and good. The assertion concerning the $\partial$-values of the composition factors of $\bar{Y}$ and
$\bar{Z}$ follows from \cite[Theorem 6.1(3)]{ChuangTan2001} and the Loewy structure of $P^{\bar{\alpha}}$.
From  \cite[Theorem 6.1(3)]{ChuangTan2001}, Lemma~\ref{lemma l and r} and the Loewy structure of $P^{\alpha}$ we further deduce
that every composition factor of $R$ has a labelling partition with $\partial$-value $d-1$ or $d+1$; 
every composition factor of $T$ has a labelling partition with $\partial$-value $d$ or $d+2$.
Let $Y:=\bar{Y}\uparrow^B$ and $Z:=\bar{Z}\uparrow^B$. With \ref{noth ingredients}(d) we this time have
\begin{align}\label{eqn upalpha 1}
S^\alpha\oplus S^\beta&\sim (S^{\bar{\alpha}})\uparrow^B\sim (D^{\bar{\alpha}})\uparrow^B\oplus  Z\,\\ \label{eqn upbeta 1}
S^\alpha\oplus S^\gamma&\sim (S^{\bar{\beta}})\uparrow^B\sim D^\gamma\oplus (D^{\bar{\alpha}})\uparrow^B\oplus Y\, \\ \label{eqn upgamma 1}
S^\beta\oplus S^\gamma& \sim (S^{\bar{\gamma}})\uparrow^B\sim D^\beta\oplus D^\gamma\oplus Z\oplus  (D^{\bar{\alpha}})\uparrow^B\,.
\end{align}
Considering the difference(\ref{eqn upbeta 1})-(\ref{eqn upalpha 1}) and (\ref{eqn Specht B case (1)}) we see that $Y\oplus T\sim R\oplus Z$. 
Comparing the $\partial$-values of the composition factors of $Y$, $Z$, $R$ and $T$, this forces $T\cong Z$ and $Y\cong R$.

\medskip

\underline{Case (3):} By Lemma~\ref{lemma l and r}, we know that $\bar{\alpha}$ is $p$-regular and not $p$-restricted, $\bar{\beta}, \alpha$ and
$\gamma$
are both $p$-regular and $p$-restricted,  $\bar{\gamma}$ is $p$-restricted and not $p$-regular, and $\gamma$ is both $p$-regular and $p$-restricted.
Recall from Remark~\ref{rem l and r} that $S^\beta\cong D^\alpha$. Together with \ref{noth ingredients}(a)-(c) we get
\begin{equation}\label{eqn Specht case (3)}
 S^{\bar{\alpha}}\ \approx \  \begin{matrix} \bar{\alpha}\\\bar{\beta}_+\oplus \bar{Z}\\ \\ \end{matrix}\;, \ \quad S^{\bar{\beta}}\ \approx \ \begin{matrix} \bar{\beta}\\\bar{\alpha}\oplus\bar{Y}\\ \bar{\beta}_+\end{matrix}\;, \ \quad S^{\bar{\gamma}}\ \approx \  \begin{matrix}\\\bar{\beta}\oplus\bar{Z}\\\bar{\alpha}\end{matrix}\,
\end{equation}
and 
\begin{equation}\label{eqn Specht B case (3)}
 S^{\alpha}\ \approx \  \begin{matrix} \alpha\\ R\\ \alpha_+\end{matrix}\;, \ \quad S^{\beta}\ \approx \  \begin{matrix} \\\alpha\\ \\\end{matrix}\;, \  \quad S^{\gamma}\ \approx \ \begin{matrix}\gamma\\  R\\\alpha\end{matrix}\,.
\end{equation}
Here $\bar{Y}$, $\bar{Z}$ and $R$ are good and semisimple.
As well,  $\bar{\beta}_+$ and $\alpha_+$ are good. 
The assertion concerning the $\partial$-values of the composition factors of $\bar{Y}$ and $\bar{Z}$
 follows from \cite[Theorem 6.1(3)]{ChuangTan2001} and the Loewy structure of $P^{\bar{\alpha}}$.
From  \cite[Theorem 6.1(3)]{ChuangTan2001}, Lemma~\ref{lemma l and r} and the Loewy structure of $P^{\alpha}$ we further deduce
that every composition factor of $R$ has a labelling partition with $\partial$-value $d-1$ or $d+1$.
Observe also that $R\neq \{0\}$, since $S^\alpha$ has Loewy length $3$, by \ref{noth ingredients}(c).
Let $Y:=\bar{Y}\uparrow^B$ and $Z:=\bar{Z}\uparrow^B$, and let $\lambda:=\Phi^{-1}(\bar{\beta}_+)$. Then \ref{noth ingredients}(d) gives
\begin{align}\label{eqn upalpha 3}
S^\alpha\oplus S^\beta&\sim (S^{\bar{\alpha}})\uparrow^B\sim (D^{\bar{\alpha}})\uparrow^B\oplus  Z\oplus D^\lambda\,\\ \label{eqn upbeta 3}
S^\alpha\oplus S^\gamma&\sim (S^{\bar{\beta}})\uparrow^B\sim D^\gamma\oplus (D^{\bar{\alpha}})\uparrow^B\oplus Y\oplus D^\lambda\, \\ \label{eqn upgamma 3}
S^\beta\oplus S^\gamma& \sim (S^{\bar{\gamma}})\uparrow^B\sim D^\gamma\oplus Z\oplus  (D^{\bar{\alpha}})\uparrow^B\,.
\end{align}
Considering the difference (\ref{eqn upbeta 3})-(\ref{eqn upalpha 3}) and (\ref{eqn Specht B case (3)}), we see that $Y\sim Z\oplus R$.
Comparing $\partial$-values, we further deduce that $Z=\{0\}$ and $Y\cong R$; in particular, also $\bar{Z}=\{0\}$. Next we consider the difference (\ref{eqn upgamma 3})-(\ref{eqn upalpha 3}) and (\ref{eqn Specht B case (3)}), which yields $\lambda=\bar{\alpha}_+$.

\medskip

\underline{Case (5):} By Lemma~\ref{lemma l and r}, the partitions $\bar{\alpha}$ and $\alpha$ are $p$-regular and $p$-restricted, while
$\bar{\beta}$, $\bar{\gamma}$, $\beta$ and $\gamma$ are $p$-restricted and not $p$-regular. From
\ref{noth ingredients}(a)-(c) we get
\begin{equation}\label{eqn Specht case (5)}
 S^{\bar{\alpha}}\ \approx \  \begin{matrix} \bar{\alpha}\\\bar{\beta}_+\oplus \bar{Z}\\ \bar{\alpha}_+\end{matrix}\;, \ \quad S^{\bar{\beta}}\ \approx \  \begin{matrix} \\\bar{\alpha}\oplus\bar{Y}\\ \bar{\beta}_+\end{matrix}\;, \ \quad S^{\bar{\gamma}}\ \approx \  \begin{matrix}\\ \bar{Z}\\\bar{\alpha}\end{matrix}\,
\end{equation}
and 
\begin{equation}\label{eqn Specht B case (5)}
 S^{\alpha}\ \approx \  \begin{matrix} \alpha\\ R\oplus \beta_+\\ \alpha_+\end{matrix}\;, \  \quad S^{\beta}\ \approx \  \begin{matrix} \\\alpha\oplus T\\ \beta_+\end{matrix}\;, \ \quad S^{\gamma}\ \approx \  \begin{matrix}\\ R\\\alpha\end{matrix}\,.
\end{equation}
Here $\bar{Y}$ $\bar{Z}$, $R$ and $T$ are good and semisimple and disjoint. Moreover, $\bar{\alpha}_+$, $\bar{\beta}_+$, $\alpha_+$ and $\beta_+$ are good. 
The assertion concerning the $\partial$-values of the composition factors of $\bar{Y}$ and $\bar{Z}$ follows from \cite[Theorem 6.1(3)]{ChuangTan2001} and the Loewy structure of $P^{\bar{\alpha}}$.
From \cite[Theorem 6.1(3)]{ChuangTan2001}, Lemma~\ref{lemma l and r} and the Loewy structure of $P^{\alpha}$ we further deduce
that every composition factor of $R$ has a labelling partition with $\partial$-value $d-1$ or $d+1$; 
every composition factor of $T$ has a labelling partition with $\partial$-value $d$ or $d+2$.
Let $Y:=\bar{Y}\uparrow^B$ and $Z:=\bar{Z}\uparrow^B$. Let further $\lambda:=\Phi^{-1}(\bar{\beta}_+)$ and
$\mu:=\Phi^{-1}(\bar{\alpha}_+)$. From \ref{noth ingredients}(d) we obtain
\begin{align}\label{eqn upalpha 5}
S^\alpha\oplus S^\beta&\sim (S^{\bar{\alpha}})\uparrow^B\sim (D^{\bar{\alpha}})\uparrow^B\oplus  Z\oplus D^\lambda\oplus D^\mu\,\\ \label{eqn upbeta 5}
S^\alpha\oplus S^\gamma&\sim (S^{\bar{\beta}})\uparrow^B\sim (D^{\bar{\alpha}})\uparrow^B\oplus Y\oplus D^\lambda\, \\ \label{eqn upgamma 5}
S^\beta\oplus S^\gamma& \sim (S^{\bar{\gamma}})\uparrow^B\sim Z\oplus  (D^{\bar{\alpha}})\uparrow^B\,.
\end{align}
We consider the difference (\ref{eqn upbeta 5})-(\ref{eqn upalpha 5}) and (\ref{eqn Specht B case (5)}) to get
$Y\oplus T\oplus \beta_+\sim R\oplus Z\oplus D^\mu$. So, comparing $\partial$-values, we get $T\cong Z$ and
$Y\oplus D^{\beta_+}\sim R\oplus D^\mu$. Next we consider the difference (\ref{eqn upalpha 5})-(\ref{eqn upgamma 5})
and (\ref{eqn Specht B case (5)}) to get $D^\lambda\oplus D^\mu\sim D^{\beta_+}\oplus D^{\alpha_+}$.
Again comparing $\partial$-values, this gives $\alpha_+=\lambda$, $\beta_+=\mu$ and then also $Y\cong R$.

To show that $[\bar{Y}:D^{\bar{\alpha}_+}]=0$, we cannot argue as in cases (4) and (6), since $\bar{\beta}$ is not $p$-regular.
However, we now see that if $[\bar{Y}:D^{\bar{\alpha}_+}]>0$, then we would also have $[Y:D^{\beta_+}]>0$ and $[S^\alpha:D^{\beta_+}]>1$, which
is impossible, since every Specht module in $B$ is multiplicity-free.

This completes the proof of the theorem.
\end{proof}

%%%%%%%%%%%%%%%%%%%%%%%%%%%%%%%%%%%%%%%%%%%%%%%%%%%%%%%%%%%%%%%%%%

\section{Proofs of Theorem~\ref{thm main1} and Theorem~\ref{thm main2}}\label{sec main1}

From now on, let $p\geq 5$, for the remainder of this section.
Our aim is to prove Theorem~\ref{thm main1} and Theorem~\ref{thm main2}. To this end we shall start by applying the results from
Section~\ref{sec weight 2}  to the case of blocks whose $p$-cores are hook partitions, that is, are of
the form $(k,1^l)$, for some $k,l\in\NN_0$. Note that a hook partition $(k,1^l)$ is a $p$-core if and only if either $0\leq k+l\leq p-1$, or
$p+1\leq k+l\leq 2p-1$, $0\leq k<p+1$ and $0\leq l <p$.
To simplify the notation, we shall denote the block $B_{(k,1^l)}(2,p)$ by $B_{k,l}$.

While Theorem~\ref{thm main1} is stated in terms of undirected graphs, we restate and prove a more detailed version here:

\begin{thm}\label{thm main1 details}
Let $p\geq 5$, and let $B_{k,l}$ be a block of $F\mathfrak{S}_n$ of $p$-weight $2$ and $p$-core $(k,1^l)$, for
some $k,l\in\NN_0$. With the graphs defined in Appendix~\ref{sec quiv}, the Ext-quiver of $B_{k,l}$ 
equals

\begin{enumerate}
\item[{\rm (a)}] $Q_{0,0}(p)$, if $k=l=0$,
\item[{\rm (b)}] $Q_{k,l}(p)$, if $1\leq k+l\leq p-1$,
\item[{\rm (c)}] $Q_{k-1,l-1}(p)$, if $p+1\leq k+l\leq 2p-1$,
\end{enumerate}

where the vertices in row $i\geq 0$, from top to bottom, are labelled by the $p$-regular partitions of $B_{k,l}$ with $\partial$-value $i$
and the total ordering on the vertices is the lexicographic ordering on partitions.
\end{thm}

In preparation of the proof of Theorem~\ref{thm main1 details}, we next collect a number of properties of 
$(2:1)$-pairs of weight-2 blocks labelled by hook partitions.

\begin{prop}\label{prop (2:1) pairs hooks}
Let $k,l\in \NN_0$ be such that $(k,1^l)$ is a $p$-core partition. Moreover, let $n:=k+l+2p$, and let $B$ be the block
$B_{k,l}$ of $F\mathfrak{S}_n$.

\begin{enumerate}

\item[{\rm (a)}] Let $k>1$. If $1\leq k+l\leq p-1$ or $p+2\leq k+l\leq 2p-1$, then $B$ forms a $(2:1)$-pair with the block $B_{k-1,l}$
of $F\mathfrak{S}_{n-1}$.

\item[{\rm (b)}] Let $l\geq 1$. If $1\leq k+l\leq p-1$ or $p+2\leq k+l\leq 2p-1$, then $B$ forms a $(2:1)$-pair with the block
$B_{k,l-1}$ of $F\mathfrak{S}_{n-1}$.

\item[{\rm (c)}] Suppose that also $(k',1^{l'})$ is a $p$-core partition, for $k',l'\in\NN_0$ with $k'+l'=k+l-2$. Then $B$ 
forms a $(2:2)$-pair with the block $B_{k',l'}$ of $F\mathfrak{S}_{n-2}$ if and only if $k+l=p+1$, $k=k'+1$ and $l=l'+1$.
\end{enumerate}
\end{prop}

\begin{proof} 
In order to prove the first half of (a) and the first half of (b), suppose that $1\leq k+l\leq p-1$. We display $(k,1^l)$ on an abacus with $l+1+2p$ beads. Then the last row of the abacus is

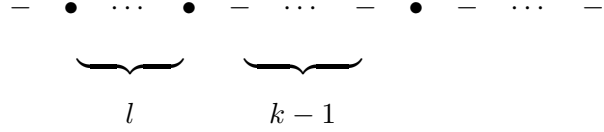
\begin{center}
\begin{tabular}{ccccccccccc}
$-$ & $\bullet$ & $\cdots$ & $\bullet$ & $-$ &$\cdots$ & $-$ &$\bullet$ & $-$ & $\cdots$ & $-$\\
&\multicolumn{3}{c}{\upbracefill}&\multicolumn{3}{c}{\upbracefill}&&&&\\
&\multicolumn{3}{c}{$l$}&\multicolumn{3}{c}{$k-1$}&&&&
\end{tabular}
\end{center}

\smallskip

Swapping runners $l+1+k$ and $l+k$, the assertion of (a) follows in this case. Swapping runners $1$ and $0$, the assertion of (b)
follows in this case.

\medskip

For the second half of (a) and (b), suppose now that $p+1\leq k+l\leq 2p-1$. We also display $(k,1^l)$ on an abacus with $l+1+2p$ beads. Note that
we must have $2\leq k\leq p$ and $1<l\leq p-1$, since $k+l\geq p+2$ and $(k,1^l)$ is a $p$-core.
We consider the last two rows of this abacus:

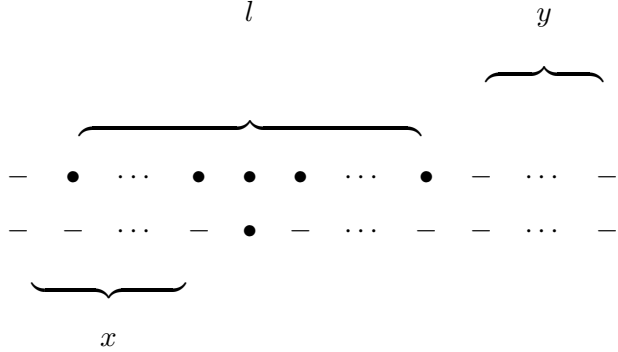
\begin{center}
\begin{tabular}{ccccccccccc}
  & \multicolumn{7}{c}{$l$}&\multicolumn{3}{c}{$y$}\\
 & \multicolumn{7}{c}{}&\multicolumn{3}{c}{\downbracefill}\\
  \multicolumn{9}{c}{\downbracefill}&&\\
 $-$ & $\bullet$ & $\cdots$ & $\bullet$ & $\bullet$ & $\bullet$ & $\cdots$ & $\bullet$ & $-$ & $\cdots$ & $-$\\
 $-$ & $-$ & $\cdots$ & $-$  & $\bullet$ & $-$ & $\cdots$  & $-$ & $-$ & $\cdots$ & $-$\\
 \multicolumn{4}{c}{\upbracefill}&&&&&&&\\
 \multicolumn{4}{c}{$x$}&&&&&&&
\end{tabular}
\end{center}

\smallskip

Here $x+y+1=k$ and $x\geq 2$. So assertion (a) follows by swapping runners $x$ and $x+1$, and assertion (b) follows by swapping runners
$0$ and $1$.

\medskip

As for assertion (c), first note that $B_{k,l}$ can only form a $(2:2)$-pair with the block $B_{k-1,l-1}$, since
the core $(k',1^{l'})$ is obtained from $(k,1^l)$ by removing two nodes of the same $p$-residue.

Next we show that, whenever $k+l=p+1$, the block $B$ forms a $(2:2)$-pair with $B_{k-1,l-1}$.
So let $k+l=p+1$. Note that then we necessarily have $k>1$ and $l\geq 1$. The last two rows of the abacus display 
of $(k,1^l)$ with $l+1+2p$ beads are

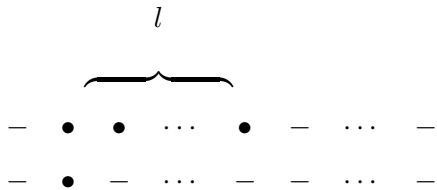
\begin{center}
\begin{tabular}{cccccccc}
&\multicolumn{4}{c}{$l$}&&&\\
&\multicolumn{4}{c}{\downbracefill}&&&\\
$-$&$\bullet$&$\bullet$& $\cdots$&$\bullet$&$-$&$\cdots$&$-$\\
$-$&$\bullet$&$-$&$\cdots$&$-$&$-$&$\cdots$&$-$
\end{tabular}
\end{center}

\smallskip

We swap the first two runners and get

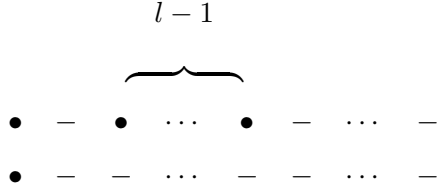
\begin{center}
\begin{tabular}{cccccccc}
&&\multicolumn{3}{c}{$l-1$}&&&\\
&&\multicolumn{3}{c}{\downbracefill}&&&\\
$\bullet$&$-$&$\bullet$& $\cdots$&$\bullet$&$-$&$\cdots$&$-$\\
$\bullet$&$-$&$-$&$\cdots$&$-$&$-$&$\cdots$&$-$
\end{tabular}
\end{center}

\smallskip

This is the abacus display of the partition $(p-(l-1)-1,1^{l-1})=(k-1,1^{l-1})$.

To complete the proof of (c), it remains to verify that, whenever $2\leq k+l\leq p-1$ or $p+3\leq k+l\leq 2p-1$, the block $B$ cannot
form a $(2:2)$-pair with $B_{k-1,l-1}$, provided the partition $(k-1,1^{l-1})$ exists and is a $p$-core. So suppose it does, and suppose that
$B$ and $B_{k-1,l-1}$ form a $(2:2)$-pair. Then $k\geq 2$, $l\geq 1$, and the two removable nodes of $(k,1^l)$ must have the same
$p$-residue. Hence $k-1\equiv p-l\pmod{p}$ and $k+l\equiv 1\pmod{p}$, a contradiction. 
\end{proof}

Note that, for $k\geq 3$, there cannot be any $(2:k)$-pair of blocks
of $p$-weight $2$ whose $p$-cores are hook partitions.
Thus, as an immediate consequence of Proposition~\ref{prop (2:1) pairs hooks}(c), we have

\begin{cor}\label{cor Scopes hooks}
There are precisely $(p-1)^2+1$ Scopes classes of $p$-blocks of symmetric groups of $p$-weight $2$ whose $p$-cores
are hook partitions. Representatives of these are given by the blocks $B_{k,l}$, where 
\begin{itemize}
\item[\rm{(i)}] either $0\leq k+l\leq p-1$, 
\item[\rm{ (ii)}] or $p+2\leq k+l\leq 2p-1$, $0\leq k<p+1$ and $0\leq l <p$.
\end{itemize}
\end{cor}

The next corollary is a consequence of Lemma~\ref{lemma l and r}.

\begin{cor}\label{cor l and r hooks}
Let $k,l\in \NN_0$ be such that $(k,1^l)$ is a $p$-core partition. Moreover, let $n:=k+l+2p$, and let $B$ be the block
$B_{k,l}$ of $F\mathfrak{S}_n$. 
Display all partitions under considerations on an $[m_1,\ldots,m_p]$-abacus with $l+1+2p$ beads.
With the notation as in Lemma~\ref{lemma l and r}, one has the following:

\begin{enumerate}

\item[{\rm (a)}] Suppose that $1\leq k+l\leq p-1$ and $k>1$, so that $B$ forms a $(2:1)$-pair with the block $\bar{B}:=B_{k-1,l}$ of $F\mathfrak{S}_{n-1}$. Then
$l_1=l$, $l_2=0$, $r_2=0$ and $r_1=p-k-l-1$; in particular, $d=p-k-1$ and $l_1+r_1=p-k-1<p-2$. Moreover,
one has

\begin{center}
\begin{tabular}{|c||c|c|c|c|c|c|}\cline{2-7}
\multicolumn{1}{c||}{}& $\bar{\alpha}$&$\bar{\beta}$&$\bar{\gamma}$&$\alpha$&$\beta$&$\gamma$\\\hline\hline
$p$-regular&$\checkmark$&$\checkmark$&$\checkmark$&$\checkmark$&$\checkmark$&$\checkmark$\\\hline
$p$-restricted&$-$&$k<p-1$&$\checkmark$&$k<p-1$&$-$&$\checkmark$\\\hline
\end{tabular}
\end{center}

\item[{\rm (b)}] Suppose that $1\leq k+l\leq p-1$, $k\geq 1$ and $l\geq 1$, so that $B$ forms a $(2:1)$-pair with the block $\bar{B}:=B_{k,l-1}$
of $F\mathfrak{S}_{n-1}$. Then $l_1=l_2=0$, $r_1=p-2$, and $r_2=l$; in particular, $d=p-2-l$, $l_1+r_1=p-2$ and $l_2+r_2=l$. Moreover,

\begin{center}
\begin{tabular}{|c||c|c|c|c|c|c|}\cline{2-7}
\multicolumn{1}{c||}{}& $\bar{\alpha}$&$\bar{\beta}$&$\bar{\gamma}$&$\alpha$&$\beta$&$\gamma$\\\hline\hline
$p$-regular&$\checkmark$&$l<p-2$&$-$&$\checkmark$&$-$&$l<p-2$\\\hline
$p$-restricted&$\checkmark$&$\checkmark$&$\checkmark$&$\checkmark$&$\checkmark$&$\checkmark$\\\hline
\end{tabular}
\end{center}

\item[{\rm (c)}] Suppose that $p+2\leq k+l\leq 2p-1$ and $k>1$, so that $B$ forms a $(2:1)$-pair with the block $\bar{B}:=B_{k-1,l}$ of $F\mathfrak{S}_{n-1}$. Then
$l_1=l$, $l_2=0$, $r_2=0$ and $r_1=p-k$; in particular, $d=p-k$ and $l_1+r_1=p-k$ and $l_2+r_2=0$. Moreover,

\begin{center}
\begin{tabular}{|c||c|c|c|c|c|c|}\cline{2-7}
\multicolumn{1}{c||}{}& $\bar{\alpha}$&$\bar{\beta}$&$\bar{\gamma}$&$\alpha$&$\beta$&$\gamma$\\\hline\hline
$p$-regular&$\checkmark$&$\checkmark$&$\checkmark$&$\checkmark$&$\checkmark$&$\checkmark$\\\hline
$p$-restricted&$-$&$2<k<p$&$\checkmark$&$2<k<p$&$-$&$\checkmark$\\\hline
\end{tabular}
\end{center}

\item[{\rm (d)}] Suppose that $p+2\leq k+l\leq 2p-1$, $k\geq 1$ and $l\geq 1$, so that $B$ forms a $(2:1)$-pair with the block $\bar{B}:=B_{k,l-1}$
of $F\mathfrak{S}_{n-1}$. Then $l_1=l_2=0$, $r_1=p-2$, and $r_2=l-1$; in particular, $d=p-1-l$, $l_1+r_1=p-2$ and $l_2+r_2=l-1$. Moreover,

\begin{center}
\begin{tabular}{|c||c|c|c|c|c|c|}\cline{2-7}
\multicolumn{1}{c||}{}& $\bar{\alpha}$&$\bar{\beta}$&$\bar{\gamma}$&$\alpha$&$\beta$&$\gamma$\\\hline\hline
$p$-regular&$\checkmark$&$l<p-1$&$-$&$\checkmark$&$-$&$l<p-1$\\\hline
$p$-restricted&$\checkmark$&$\checkmark$&$\checkmark$&$\checkmark$&$\checkmark$&$\checkmark$\\\hline
\end{tabular}
\end{center}
\end{enumerate}
\end{cor}

\begin{proof}
Under the assumptions of (a), we are in case (2) of Remark~\ref{rem l and r} if $k<p-1$, and in case (1) if $k=p-1$ (and $l=0$). Under
the assumptions of (b), we are in case (4) of Remark~\ref{rem l and r} if $l<p-2$, and in case (5) if $l=p-2$ (and $k=1$).
Under the assumptions of (c), we are in case (1) if $k=p$, and in case (2) if $2<k<p$; note that $k=2$ is not possible. Lastly, under 
the assumptions of (d), we are in case (4) if $l<p-1$, and in case (5) if $l=p-1$. Hence, the assertions follow form Lemma~\ref{lemma l and r}.
\end{proof}

\begin{lemma}\label{lemma alpha max}
Let $k,l\in \NN_0$ be such that $(k,1^l)$ is a $p$-core partition. Moreover, let $n:=k+l+2p$, and let $B$ be the block
$B_{k,l}$ of $F\mathfrak{S}_n$. 

\begin{enumerate}
\item[{\rm (a)}] Suppose that $k>1$, so that $B$ forms a $(2:1)$-pair with the block $\bar{B}:=B_{k-1,l}$ of $F\mathfrak{S}_{n-1}$.
Then $\bar{\alpha}:=\bar{\alpha}(B,\bar{B})$ is the lexicographically largest partition of $\bar{B}$ with $\partial$-value $\partial(\bar{\alpha})$.

\item[{\rm (b)}] Suppose that $k\geq 1$ and $l\geq 1$, so that $B$ forms a $(2:1)$-pair with the block $\bar{B}:=B_{k,l-1}$
of $F\mathfrak{S}_{n-1}$. Then $\bar{\alpha}:=\bar{\alpha}(B,\bar{B})$ is the lexicographically smallest $p$-regular partition of $\bar{B}$ with
$\partial$-value $\partial(\bar{\alpha})$.
\end{enumerate}
\end{lemma}

\begin{proof}
(a)\, Assume that there is a partition $\mu$ of $n-1$ with $p$-core $(k-1,1^l)$, $\partial(\mu)=\partial(\bar{\alpha})$ and
$\bar{\alpha}<\mu$. Then we may choose $\mu$ to be the lexicographically smallest such partition.  By \cite[Remark 2.1]{ChuangTan2001},
this forces $\mu=(\bar{\alpha})_+$ and $\bar{\alpha}$ to be $p$-restricted. But this contradicts Corollary~\ref{cor l and r hooks}(a) and (c).

\smallskip

(b)\, Assume that there is a $p$-regular partition $\mu$ of $n-1$ with $p$-core $(k,1^{l-1})$, $\partial(\bar{\alpha})=\partial(\mu)$ and
$\mu<\bar{\alpha}$. Then we may choose $\mu$ to be the lexicographically largest such partition. Then \cite[Remark 2.1]{ChuangTan2001}
gives $\mu=(\bar{\alpha})_-$. But, by \ref{noth partial} and Theorem~\ref{thm Loewy except}, we know that $\mu=(\bar{\alpha})_-=\bar{\gamma}$ and $\bar{\gamma}$ is not $p$-regular, by
Corollary~\ref{cor l and r hooks} (b) and (d), a contradiction.
\end{proof}

\begin{noth}\label{noth strategy 1}{\bf Strategy of proof.}\,
We are now in the position to prove Theorem~\ref{thm main1 details}. Our most important ingredients will be the results
concerning the Ext-quivers of the principal blocks of $F\mathfrak{S}_{2p}$ and $F\mathfrak{S}_{2p+1}$ in Appendix~\ref{sec B0} 
together with Theorem~\ref{thm Loewy except}. The notation used throughout the proof will thus be chosen in accordance with
that fixed in Theorem~\ref{thm Loewy except}. As well, whenever $(B,\bar{B})$ is a $(2:1)$-pair of blocks of weight 2 and $\mu$ is
a good partition of $\bar{B}$, we denote the corresponding good partition $\Phi^{-1}(\mu)$ by $\hat{\mu}$, as in \ref{noth (2:1)}.

\smallskip

In the proof of Theorem~\ref{thm main1 details} below, we shall argue by induction on $k+l$, and give full details in the case where $k+l\leq p-1$. In doing so, we
shall consider the following subcases: 

\begin{itemize}
\item $k=l=0$, or $k=1$ and $l=0$, when $B_{k,l}$ is the principal
block of $F\mathfrak{S}_{2p}$ or the principal block of $F\mathfrak{S}_{2p+1}$; these are covered by Theorem~\ref{thm quiver B0};
\item $1<k\leq p-3$, when we are in case (2) of 
Theorem~\ref{thm Loewy except};
\item $k=1$ and $1\leq l<p-3$, when we are case (4) of Theorem~\ref{thm Loewy except};
\item $k=p-1$ and $l=0$, when we are in case (1) of Theorem~\ref{thm Loewy except};
\item $k=1$ and $p-3\leq l\leq p-2$, when we are in case (4) and (5), respectively, of Theorem~\ref{thm Loewy except}.
\end{itemize}

The assertion in the case where $p+1\leq k+l$, can then be obtained completely analogously, so that we shall leave the details 
to the reader.  Moreover, we should like to emphasize that, for all $p$-cores $(k,1^l)$ with $k\geq 1$,  the blocks
$B_{k,l}$ and $B_{l+1,k-1}$ are isomorphic via tensoring with the sign representation, since $(l+1,1^{k-1})=(k,1^l)'$.
In particular, the Ext-quivers of $B_{k,l}$ and $B_{l+1,k-1}$ are isomorphic as undirected graphs. Hence, if we were only interested 
in the structure of the Ext-quivers as undirected graphs, we would only need to examine half of the blocks. Since, however, we 
also want to give more detailed information on the lexicographic ordering and the $\partial$-values of the 
$p$-regular partitions in the blocks occurring in Theorem~\ref{thm main1 details}, we shall treat all
blocks via our inductive arguments.
\end{noth}

\begin{proof}(of Theorem~\ref{thm main1 details})
In the following, set $B:=B_{k,l}$ and $\kappa:=\kappa_B:=(k,1^l)$. We argue by induction on $k+l$, and suppose first that $k+l\leq p-1$. If $k=l=0$, then 
$B$ is the principal block of $F\mathfrak{S}_{2p}$, which has Ext-quiver $Q_{0,0}(p)$, by Theorem~\ref{thm quiver B0}. If $k+l=1$, then $B$ is the principal
block of $F\mathfrak{S}_{2p+1}$, which has Ext-quiver $Q_{1,0}(p)$, by Theorem~\ref{thm quiver B0}. Thus, from now on we may suppose that $k+l>1$. 
Suppose, moreover, that $1<k\leq p-3$. Then $B$ forms a $(2:1)$-pair with the block $\bar{B}:=B_{k-1,l}$, by Proposition~\ref{prop (2:1) pairs hooks}.
By induction, $\bar{B}$ has Ext-quiver $Q_{k-1,l}(p)$. 
By Corollary~\ref{cor l and r hooks} and Lemma~\ref{lemma alpha max}, we know that $\bar{\alpha}$ is the largest partition of $n-1$ with $p$-core
$\kappa_{\bar{B}}$ and $\partial$-value $p-k$.
By \ref{noth partial Scopes}, it now suffices to consider the following part of the quiver of $\bar{B}$, whose vertices lie
on those rows with $\partial$-values $p-k-2,p-k-1,p-k$ and $p-k+1$.

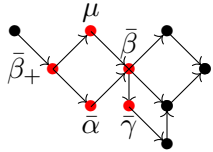
\begin{center}
\begin{tikzpicture}
\coordinate[label=left:$\bar{\beta}_+$] (B+) at (0.5,-0.5);
\coordinate[label=above:$\mu$] (M) at (1,0);
\coordinate[label=above:$\bar{\beta}$] (B) at (1.5,-0.5);
\coordinate[label=below:$\bar{\alpha}$] (A) at (1,-1);
\coordinate[label=below:$\bar{\gamma}$] (G) at (1.5,-1);
\draw[con] (0,0) node{$\bullet$} -- (B+) node{\color{red} $\bullet$};
\draw[con] (B+) -- (M) node{\color{red} $\bullet$};
\draw[con] (B) node{\color{red} $\bullet$} -- (G) node{\color{red} $\bullet$};
\draw[con] (M)  -- (B) ;
\draw[con] (B+) -- (A) node{\color{red} $\bullet$};
\draw[con] (A)  -- (B) ;
\draw[con] (B) -- (2,0) node{$\bullet$};
\draw[con] (B) -- (2,-1) node{$\bullet$};
\draw[con] (G) -- (2,-1.5) node{$\bullet$};
\draw[con]  (2,-1.5) -- (2,-1) node{$\bullet$};
\draw[con]  (2,-1) -- (2.5,-0.5) node{$\bullet$};
\draw[con]  (2,0) -- (2.5,-0.5) node{$\bullet$};
\end{tikzpicture}
\end{center}

We identify the labels of the red vertices. 
To do so, we use Corollary~\ref{cor l and r hooks} and Theorem~\ref{thm Loewy except}.
First, by Corollary~\ref{cor l and r hooks}, $\bar{\gamma}$ and $\bar{\beta}$ are
$p$-regular. By Theorem~\ref{thm Loewy except}, we then know that $\bar{\alpha}=\bar{\gamma}_+$, so that $\bar{\gamma}$ is the vertex to the right of $\bar{\alpha}$.
Moreover, by Theorem~\ref{thm Loewy except}, $\bar{\gamma}$ is connected to $\bar{\beta}$, and $\partial(\bar{\beta})=\partial(\bar{\gamma})-1$, by Corollary~\ref{cor l and r hooks}.
Since, by Lemma~\ref{lemma Ext alpha}, $\bar{\beta}$ is also connected to $\bar{\alpha}$, this identifies the positions of $\bar{\beta}$
and $\bar{\beta}_+$. By induction, we also know that $\mu>\bar{\beta}$. Hence, by \cite[Theorem 6.1]{ChuangTan2001},
 we must have $[S^{\bar{\beta}}:D^\mu]\neq 0$, and then $\bar{Y}\cong D^\mu$, in the notation of Theorem~\ref{thm Loewy except}.
  This then also forces $\bar{Z}=\{0\}$,
 so that $\bar{\beta}$ is the only common neighbour of $\bar{\alpha}$ and $\bar{\gamma}$.
 So, by Corollary~\ref{cor l and r hooks}, Theorem~\ref{thm Loewy except} and \ref{noth partial Scopes}, we deduce the
 following information on the corresponding part of the quiver of $B$, whose vertices again lie in rows with $\partial$-values
  $p-k-2,p-k-1,p-k$ and $p-k+1$:

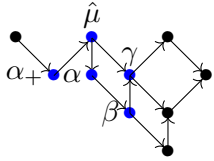
\begin{center}
 \begin{tikzpicture}
\coordinate[label=left:$\alpha_+$] (A+) at (0.5,-0.5);
\coordinate[label=above:$\hat{\mu}$] (MH) at (1,0);
\coordinate[label=left:$\alpha$] (A) at (1,-0.5);
\coordinate[label=left:$\beta$] (B) at (1.5,-1);
\coordinate[label=above:$\gamma$] (G) at (1.5,-0.5);
\draw[con] (0,0) node{$\bullet$} -- (A+) node{\color{blue} $\bullet$};
\draw[con] (A+) -- (MH) node{\color{blue} $\bullet$} ;
\draw[con] (MH)  -- (G) node{\color{blue} $\bullet$};
\draw[con] (MH)  -- (A) node{\color{blue} $\bullet$};
\draw[con] (A) -- (B) node{\color{blue} $\bullet$};
\draw[con] (B) -- (G);
\draw[con] (B) -- (2,-1.5) node{$\bullet$};
\draw[con] (G) -- (2,-1) node{$\bullet$};
\draw[con] (G) -- (2,0) node{$\bullet$};
\draw[con] (2,-1) -- (2.5,-0.5) node{$\bullet$};
\draw[con] (2,0) -- (2.5,-0.5);
\draw[con]  (2,-1.5) -- (2,-1);
\end{tikzpicture}

\end{center}

Here $D^{\hat{\mu}}\cong Y$. Since, by Corollary~\ref{cor l and r hooks}, $\beta$ is not $p$-restricted, $\beta_+$ does not exist, so that
$\alpha$ is only connected to $\hat{\mu}$ and $\beta$, by Lemma~\ref{lemma Ext alpha} and Theorem~\ref{thm Loewy except}; 
in particular, we have $\hat{\mu}>\alpha$ as well as $\beta>\gamma$.
This shows that $B$ has quiver $Q_{k,l}(p)$.

\smallskip

Next suppose that $k=1$ and $1\leq l<p-3$. Then, by Proposition~\ref{prop (2:1) pairs hooks}, $B$ forms a $(2:1)$-pair with the 
block $\bar{B}:=B_{k,l-1}$. By induction, $\bar{B}$ has quiver $Q_{k,l-1}(p)$. By Corollary~\ref{cor l and r hooks} and Lemma~\ref{lemma alpha max},
$\bar{\alpha}$ is the smallest $p$-regular partition of $n-1$ with $p$-core $\kappa_{\bar{B}}$ and $\partial$-value $p-1-l$. So, by 
\ref{noth partial Scopes}, we only need to consider the following part of the quiver of $\bar{B}$, whose vertices 
lie in rows with $\partial$-values $p-l-3,p-l-2,p-l-1$ and $p-l$:

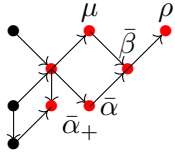
\begin{center}
\begin{tikzpicture}
\coordinate[label=above:$\mu$] (M) at (1,0);
\coordinate[label=above:$\rho$] (R) at (2,0);
\coordinate[label=above:$\bar{\beta}$] (B) at (1.5,-0.5);
\coordinate[label= right:$\bar{\alpha}$] (A) at (1,-1);
\coordinate[label=below right:$\bar{\alpha}_+$] (A+) at (0.5,-1);
\draw[con] (0,0) node{$\bullet$} -- (B+) node{\color{red} $\bullet$};
\draw[con] (B+) -- (M) node{\color{red} $\bullet$};
\draw[con] (M) -- (B) node{\color{red} $\bullet$};
\draw[con] (B) -- (R) node{\color{red} $\bullet$};
\draw[con] (B+) -- (A+) node{\color{red} $\bullet$};
\draw[con] (B+) -- (A) node{\color{red} $\bullet$};
\draw[con] (A) -- (B);
\draw[con] (0,-1.5) node{$\bullet$} -- (A+);
\draw[con] (0,-1) node{$\bullet$} -- (0,-1.5);
\draw[con] (0,-1) -- (B+);
\end{tikzpicture}
\end{center}

Again we determine the labels of the red vertices. By Corollary~\ref{cor l and r hooks}, $\bar{\gamma}$ is $p$-singular, while $\bar{\beta}$ is $p$-regular and $p$-restricted, so that $\bar{\beta}_+$ exists. By Theorem~\ref{thm Loewy except} and Lemma~\ref{lemma Ext alpha}, 
$\bar{\alpha}$ is connected to
$\bar{\beta}$ and $\bar{\beta}_+$. This identifies these two vertices. Moreover, by induction, we know that $\mu>\bar{\beta}$, so that 
$D^\mu\cong \bar{Y}$ and $\bar{Z}=\{0\}$. So, by Corollary~\ref{cor l and r hooks}, Theorem~\ref{thm Loewy except} and \ref{noth partial Scopes}, we deduce the
 following information on the corresponding part of the quiver of $B$, whose vertices again lie in rows with $\partial$-values
  $p-l-3,p-l-2,p-l-1$ and $p-l$:

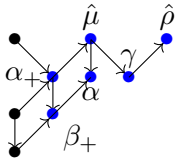
\begin{center}
\begin{tikzpicture}
\coordinate[label=left:$\alpha_+$] (A+) at (0.5,-0.5);
\coordinate[label=above:$\hat{\mu}$] (M) at (1,0);
\coordinate[label=above:$\hat{\rho}$] (R) at (2,0);
\coordinate[label=above:$\gamma$] (G) at (1.5,-0.5);
\coordinate[label=below:$\alpha$] (A) at (1,-0.5);
\coordinate[label=below right:$\beta_+$] (B+) at (0.5,-1);
\draw[con] (0,0) node{$\bullet$} -- (A+) node{\color{blue} $\bullet$};
\draw[con] (A+) -- (M) node{\color{blue} $\bullet$};
\draw[con] (M) -- (G) node{\color{blue} $\bullet$};
\draw[con] (G) -- (R) node{\color{blue} $\bullet$};
\draw[con] (M) -- (A) node{\color{blue} $\bullet$};
\draw[con] (A+) -- (B+) node{\color{blue} $\bullet$};
\draw[con] (B+) -- (A);
\draw[con] (0,-1.5) node{$\bullet$} -- (B+);
\draw[con] (0,-1) node{$\bullet$} -- (0,-1.5);
\draw[con] (0,-1) -- (A+);
\end{tikzpicture}
\end{center}

Here $D^{\hat{\mu}}\cong Y$. Since, by Corollary~\ref{cor l and r hooks}, $\beta$ is $p$-singular, Lemma~\ref{lemma Ext alpha} and
Theorem~\ref{thm Loewy except}
imply that $\beta_+$ and $\hat{\mu}$ are the only neighbours of $\alpha$. 
Since $[S^\alpha:D^{\hat{\mu}}]\neq 0\neq [S^\alpha:D^{\beta_+}]$, we must have $\beta_+>\alpha$ and $\hat{\mu}>\alpha$.
Moreover, since $[S^\rho:D^{\bar{\beta}}]\neq 0$, also $[S^{\hat{\rho}}:D^\gamma]\neq 0$ and $\gamma>\hat{\rho}$.
Consequently, $B$ has quiver $Q_{k,l}(p)$.

\smallskip

Next consider the case where $k=p-2$ and $0\leq l\leq 1$. Then, by Proposition~\ref{prop (2:1) pairs hooks}, $B$ 
forms a $(2:1)$-pair with the block $\bar{B}:=B_{p-3,l}$, which has quiver $Q_{p-3,l}(p)$, by induction. Again, by 
Corollary~\ref{cor l and r hooks} and Lemma~\ref{lemma alpha max},
$\bar{\alpha}$ is the largest $p$-regular partition of $n-1$ with $p$-core $\kappa_{\bar{B}}$ and $\partial$-value $2$. By \ref{noth partial Scopes},
we need to consider the following part of the quiver of $\bar{B}$, whose vertices lie in rows with $\partial$-values $0,1,2$ and $3$:

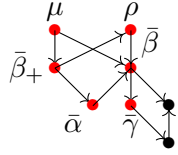
\begin{center}
\begin{tikzpicture}
\coordinate[label=left:$\bar{\beta}_+$] (B+) at (0,-0.5);
\coordinate[label=above:$\mu$] (M) at (0,0);
\coordinate[label=above:$\rho$] (R) at (1,0);
\coordinate[label=above right:$\bar{\beta}$] (B) at (1,-0.5);
\coordinate[label=below:$\bar{\gamma}$] (G) at (1,-1);
\coordinate[label=below left:$\bar{\alpha}$] (A) at (0.5,-1);
\draw[con] (M) node{\color{red} $\bullet$} -- (B) node{\color{red} $\bullet$};
\draw[con] (M)  -- (B+) node{\color{red} $\bullet$};
\draw[con] (B+)  -- (R) node{\color{red} $\bullet$};
\draw[con] (B+)  -- (A) node{\color{red} $\bullet$};
\draw[con] (A) -- (B);
\draw[con] (B) -- (G) node{\color{red} $\bullet$};
\draw[con] (R) -- (B);
\draw[con] (G) -- (1.5,-1.5) node{$\bullet$};
\draw[con] (B) -- (1.5,-1) node{$\bullet$};
\draw[con] (1.5,-1.5) -- (1.5,-1);
\end{tikzpicture}
\end{center}

By Corollary~\ref{cor l and r hooks}, $\bar{\beta}$ is $p$-regular and $p$-restricted, so that $\bar{\beta}_+$ exists. As well, $\bar{\gamma}$ is
$p$-regular and $p$-restricted, and $\bar{\gamma}_+=\bar{\alpha}$, by Theorem~\ref{thm Loewy except}.  So, by Theorem~\ref{thm Loewy except} and Lemma~\ref{lemma Ext alpha}, 
$\bar{\alpha}$ is only connected to $\bar{\beta}$ and $\bar{\beta}_+$, which identifies 
the positions of $\bar{\beta}$, $\bar{\beta}_+$ and $\bar{\gamma}$. By induction, we further have $\mu>\bar{\beta}$ and $\rho>\bar{\beta}$.
From Theorem~\ref{thm Loewy except} and \cite[Theorem 6.1]{ChuangTan2001}, we thus deduce that $\bar{Y}\cong D^\mu\oplus D^\rho$ and $\bar{Z}=\{0\}$.
Thus, by Theorem~\ref{thm Loewy except} and \ref{noth partial Scopes} we obtain the corresponding part of the quiver of $B$, where again the vertices drawn lie in
rows with $\partial$-values $0,1,2$ and $3$:

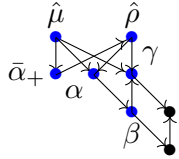
\begin{center}
\begin{tikzpicture}
\coordinate[label=left:$\bar{\alpha}_+$] (A+) at (0,-0.5);
\coordinate[label=above:$\hat{\mu}$] (M) at (0,0);
\coordinate[label=above:$\hat{\rho}$] (R) at (1,0);
\coordinate[label=above right:$\gamma$] (G) at (1,-0.5);
\coordinate[label=below:$\beta$] (B) at (1,-1);
\coordinate[label=below left:$\alpha$] (A) at (0.5,-0.5);
\draw[con] (M) node{\color{blue} $\bullet$} -- (G) node{\color{blue} $\bullet$};
\draw[con] (M)  -- (A+) node{\color{blue} $\bullet$};
\draw[con] (M)  -- (A) node{\color{blue} $\bullet$};
\draw[con] (A+)  -- (R) node{\color{blue} $\bullet$};
\draw[con] (A) -- (B)node{\color{blue} $\bullet$};
\draw[con] (B) -- (G) node{\color{blue} $\bullet$};
\draw[con] (R) -- (G);
\draw[con] (R) -- (A);
\draw[con] (B) -- (1.5,-1.5) node{$\bullet$};
\draw[con] (G) -- (1.5,-1) node{$\bullet$};
\draw[con] (1.5,-1.5) -- (1.5,-1);
\end{tikzpicture}
\end{center}

Here, by Theorem~\ref{thm Loewy except}, we have $Y\cong D^{\hat{\mu}}\oplus D^{\hat{\rho}}$. Since, by Corollary~\ref{cor l and r hooks}, $\beta$ is not $p$-restricted, 
 $\beta_+$ does not exist. Hence, by Lemma~\ref{lemma Ext alpha} and Theorem~\ref{thm Loewy except}, $\alpha$ is only connected to $\beta$, $\hat{\mu}$ and
 $\hat{\rho}$. The information on the lexicographic order is again obtained from Theorem~\ref{thm Loewy except} and \cite[Theorem 6.1]{ChuangTan2001}. This shows that
 $B$ has quiver $Q_{k,l}(p)$.
 
 \smallskip

 If $k=p-1$, then $l=0$. By Proposition~\ref{prop (2:1) pairs hooks}, $B$ forms a $(2:1)$-pair with
 the block $\bar{B}:=B_{p-2,0}$. The latter has quiver $Q_{p-2,0}(p)$, by induction.
 By Corollary~\ref{cor l and r hooks} and Lemma~\ref{lemma alpha max},
$\bar{\alpha}$ is the largest $p$-regular partition of $n-1$ with $p$-core $\kappa_{\bar{B}}$ and $\partial$-value $1$. By 
\ref{noth partial Scopes},
we consider the following part of the quiver of $\bar{B}$, whose vertices lie in rows with $\partial$-values $0,1,2$ and $3$:

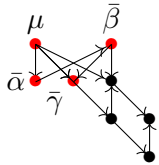
\begin{center}
\begin{tikzpicture}
\coordinate[label=left:$\bar{\alpha}$] (A) at (0,-0.5);
\coordinate[label=above:$\mu$] (M) at (0,0);
\coordinate[label=above:$\bar{\beta}$] (B) at (1,0);
\coordinate[label=below left:$\bar{\gamma}$] (G) at (0.5,-0.5);
\draw[con] (M) node{\color{red} $\bullet$} -- (G) node{\color{red} $\bullet$};
\draw[con] (M)  -- (A) node{\color{red} $\bullet$};
\draw[con] (M)  -- (G);
\draw[con] (A)  -- (B) node{\color{red} $\bullet$};
\draw[con] (B) -- (G);
\draw[con] (B) -- (1,-0.5) node{$\bullet$};
\draw[con] (G) -- (1,-1) node{$\bullet$};
\draw[con] (1,-1) -- (1,-0.5);
\draw[con] (1,-0.5) -- (1.5,-1)node{$\bullet$};
\draw[con] (1,-1) -- (1.5,-1.5)node{$\bullet$};
\draw[con] (1.5,-1.5) -- (1.5,-1);
\draw[con] (M) -- (1,-0.5);
\end{tikzpicture}
\end{center}

By Corollary~\ref{cor l and r hooks}, $\bar{\gamma}$ is both $p$-regular and $p$-restricted, so that $\bar{\gamma}_+$ exists. 
By Theorem~\ref{thm Loewy except}, we get $\bar{\alpha}=\bar{\gamma}_+$. 
By Corollary~\ref{cor l and r hooks}, we further know that $\partial(\bar{\beta})=0$, and $\bar{\beta}$
is $p$-regular and not $p$-restricted, so that $\bar{\beta}_+$ does not exists. More precisely, we have $\bar{\beta}=(2(p-1),p,1)$.
Hence $\bar{\beta}$ is white, and is thus the largest $p$-regular white partition with $p$-core $\kappa_{\bar{B}}$. This identifies the 
positions of $\bar{\alpha}$, $\bar{\gamma}$ and $\bar{\beta}$. 
The remaining red vertex belongs to the $p$-regular (black) partition $\mu$
with $\bar{Z}\cong D^\mu$. Since all other vertices connected to $\bar{\beta}$ are smaller than $\bar{\beta}$, Theorem~\ref{thm Loewy except} also
implies $\bar{Y}=\{0\}$. 

By Corollary~\ref{cor l and r hooks}, $\beta$ is $p$-regular and not $p$-restricted, so that $\beta_+$ is not defined. 
Moreover, we have $\alpha=(2(p-1)+1,p)=\langle p,p\rangle$ with $\partial(\alpha)=0$, and $\alpha$ is white. 
By Lemma~\ref{lemma Ext alpha} and Theorem~\ref{thm Loewy except}, we deduce
that $\alpha$ is only connected to $\beta$.  As for the lexicographic ordering, we have $\mu=(3p-2)=\langle p-1\rangle$, since this is the lexicographically largest (white) partition in $\bar{B}$. 
Thus $\hat{\mu}=(3p-1)=\langle p\rangle$; in particular, $\hat{\mu}>\alpha$.  Note also that $\hat{\mu}$
satisfies $Z\cong D^{\hat{\mu}}$. This now gives the corresponding part of the quiver of $B$.

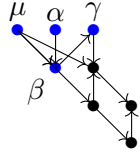
\begin{center}
\begin{tikzpicture}
\coordinate[label=above:$\alpha$] (A) at (0.5,0);
\coordinate[label=above:$\hat{\mu}$] (M) at (0,0);
\coordinate[label=above:$\gamma$] (G) at (1,0);
\coordinate[label=below left:$\beta$] (B) at (0.5,-0.5);
\draw[con] (M) node{\color{blue} $\bullet$} -- (B) node{\color{blue} $\bullet$};
\draw[con] (M)  -- (B);
\draw[con] (A)  node{\color{blue} $\bullet$} -- (B);
\draw[con] (B) -- (G) node{\color{blue} $\bullet$};
\draw[con] (G) -- (1,-0.5) node{$\bullet$};
\draw[con] (B) -- (1,-1) node{$\bullet$};
\draw[con] (1,-1) -- (1,-0.5);
\draw[con] (1,-0.5) -- (1.5,-1)node{$\bullet$};
\draw[con] (1,-1) -- (1.5,-1.5)node{$\bullet$};
\draw[con] (1.5,-1.5) -- (1.5,-1);
\draw[con] (M) -- (1,-0.5);
\end{tikzpicture}
\end{center}

The remaining information on the lexicographic order is again obtained from Theorem~\ref{thm Loewy except} and \cite[Theorem 6.1]{ChuangTan2001}. 
This shows that $B$ has quiver $Q_{k,l}(p)$.

\smallskip

To complete the case $k+l\leq p-1$, it remains to treat the blocks with $p$-cores $\kappa_{1,p-3}$ and $\kappa_{1,p-2}$. 
So let first $k=1$ and $l=p-3$. By Proposition~\ref{prop (2:1) pairs hooks}, $B$ forms a $(2:1)$-pair
with $\bar{B}:=B_{1,p-4}$, which has quiver $Q_{1,p-4}(p)$, by induction. By Lemma~\ref{lemma alpha max}
and Corollary~\ref{cor l and r hooks}, $\bar{\alpha}$ is the smallest $p$-regular partition of $n-1$ with $p$-core $(1^{p-3})$
and $\partial$-value 2. By \ref{noth partial Scopes}, we thus have to consider the following part of the quiver of $\bar{B}$, whose
vertices belong to the rows with $\partial$-values $0,1,2$ and $3$:

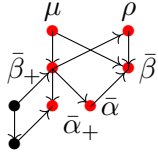
\begin{center}
\begin{tikzpicture}
\coordinate[label=left:$\bar{\beta}_+$] (B+) at (0,-0.5);
\coordinate[label=above:$\mu$] (M) at (0,0);
\coordinate[label=above:$\rho$] (R) at (1,0);
\coordinate[label=right:$\bar{\beta}$] (B) at (1,-0.5);
\coordinate[label=below right:$\bar{\alpha}_+$] (A+) at (0,-1);
\coordinate[label=right:$\bar{\alpha}$] (A) at (0.5,-1);
\draw[con] (M) node{\color{red} $\bullet$} -- (B) node{\color{red} $\bullet$};
\draw[con] (M)  -- (B+) node{\color{red} $\bullet$};
\draw[con] (B+)  -- (R) node{\color{red} $\bullet$};
\draw[con] (B+)  -- (A) node{\color{red} $\bullet$};
\draw[con] (A) -- (B);
\draw[con] (R) -- (B);
\draw[con] (B+)  -- (A+) node{\color{red} $\bullet$};
\draw[con] (-0.5,-1) node{$\bullet$} -- (B+);
\draw[con] (-0.5,-1) -- (-0.5,-1.5) node{$\bullet$};
\draw[con] (-0.5,-1.5) -- (A+);
\end{tikzpicture}
\end{center}

By Corollary~\ref{cor l and r hooks}, $\bar{\gamma}$ is $p$-singular, while $\bar{\alpha}$ and $\bar{\beta}$
are both $p$-regular and $p$-restricted. Thus $\bar{\alpha}_+$ and $\bar{\beta}_+$ exist. As well, $\partial(\bar{\beta})=1$. This identifies
the positions of $\bar{\alpha}$, $\bar{\beta}$, $\bar{\alpha}_+$ and $\bar{\beta}_+$. By Theorem~\ref{thm Loewy except} and \cite{ChuangTan2001}, we 
further deduce that $\bar{Y}\cong D^\mu\oplus D^\rho$ and $\bar{Z}=\{0\}$.

By Corollary~\ref{cor l and r hooks}, $\gamma$ is $p$-regular and $p$-restricted, and by Theorem~\ref{thm Loewy except}, we have $\gamma_+=\alpha$ with
$\partial$-value 1. So from Theorem~\ref{thm Loewy except} we deduce the following information on the corresponding part of the quiver of $B$:

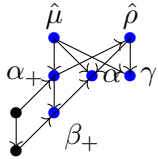
\begin{center}
\begin{tikzpicture}
\coordinate[label=left:$\alpha_+$] (A+) at (0,-0.5);
\coordinate[label=above:$\hat{\mu}$] (M) at (0,0);
\coordinate[label=above:$\hat{\rho}$] (R) at (1,0);
\coordinate[label=right:$\gamma$] (G) at (1,-0.5);
\coordinate[label=below right:$\beta_+$] (B+) at (0,-1);
\coordinate[label=right:$\alpha$] (A) at (0.5,-0.5);
\draw[con] (M) node{\color{blue} $\bullet$} -- (G) node{\color{blue} $\bullet$};
\draw[con] (M)  -- (A+) node{\color{blue} $\bullet$};
\draw[con] (A+)  -- (R) node{\color{blue} $\bullet$};
\draw[con] (M)  -- (A) node{\color{blue} $\bullet$};
\draw[con] (R) -- (A);
\draw[con] (R) -- (G);
\draw[con] (A+)  -- (B+) node{\color{blue} $\bullet$};
\draw[con] (-0.5,-1) node{$\bullet$} -- (A+);
\draw[con] (-0.5,-1) -- (-0.5,-1.5) node{$\bullet$};
\draw[con] (-0.5,-1.5) -- (B+);
\draw[con] (B+) -- (A);
\end{tikzpicture}
\end{center}

Here $Y\cong D^{\hat{\mu}}\oplus D^{\hat{\rho}}$.  By Lemma~\ref{lemma Ext alpha}, we also conclude that $\hat{\rho}$,
$\hat{\mu}$ and $\beta_+$ are the only neighbours of $\alpha$. Thus $B$ has quiver $Q_{1,p-3}(p)$.
The information concerning the lexicographic ordering follows again from Theorem~\ref{thm Loewy except} and
\cite[Theorem~6.1]{ChuangTan2001}.

\smallskip

Now let $k=1$ and $l=p-2$. Then $B$ forms a $(2:1)$-pair with $\bar{B}:=B_{1,p-3}$, which has quiver
$Q_{1,p-3}(p)$, as we have just seen. By Corollary~\ref{cor l and r hooks},
we know that $\bar{\beta}$ and $\bar{\gamma}$ are $p$-singular, but $\bar{\beta}$ is $p$-restricted, so that $\bar{\beta}_+$ exists.
Moreover, $\bar{\alpha}$ is both $p$-regular and $p$-restricted, so that $\bar{\alpha}_+$ exists. As well, $\bar{\alpha}$
has $\partial$-value $1$ and is, by Lemma~\ref{lemma alpha max}, the smallest $p$-regular partition of $n-1$ with
$p$-core $(1^{p-2})$ and $\partial$-value $1$. We also get $\bar{\beta}=(2^p,1^{p-2})$, which has $\partial$-value $0$
and is black; in particular, also $\bar{\beta}_+$ is black. This gives the following relevant part of the quiver of
$\bar{B}$, which lies in the rows with $\partial$-values $0,1,2$ and $3$:

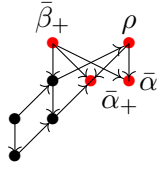
\begin{center}
\begin{tikzpicture}
\coordinate[label=above:$\bar{\beta}_+$] (B+) at (0,0);
\coordinate[label=above:$\rho$] (R) at (1,0);
\coordinate[label=right:$\bar{\alpha}$] (A) at (1,-0.5);
\coordinate[label=below right:$\bar{\alpha}_+$] (A+) at (0.5,-0.5);
\draw[con] (B+) node{\color{red} $\bullet$} -- (A) node{\color{red} $\bullet$};
\draw[con] (B+)  -- (0,-0.5) node{$\bullet$};
\draw[con] (0,-0.5)  -- (R) node{\color{red} $\bullet$};
\draw[con] (B+)  -- (A+) node{\color{red} $\bullet$};
\draw[con] (R) -- (A+);
\draw[con] (R) -- (A);
\draw[con] (0,-0.5)  -- (0,-1) node{$\bullet$};
\draw[con] (0,-1) -- (A+);
\draw[con] (-0.5,-1) node{$\bullet$} -- (-0.5,-1.5)node{$\bullet$};
\draw[con] (-0.5,-1.5) -- (0,-1);
\draw[con] (-0.5,-1) -- (0,-0.5);
\end{tikzpicture}
\end{center}

By Theorem~\ref{thm Loewy except}, we have $D^\rho\cong \bar{Z}$.  Note that $\alpha=(3,2^{p-1},1^{p-2})=\langle 2,2\rangle$,
which has $\partial$-value $0$ and is black. Displaying all partitions of $B$ on a $[2,3^{p-1}]$-abacus, we deduce that the $p$-regular
partitions of $B$ with $\partial$-value 0 are precisely $\langle p\rangle,\langle m,m-1\rangle$, for $m\in\{3,\ldots,p\}$, and $\langle 2,2\rangle=\alpha$. 
Hence $\alpha$ is the lexicographically smallest $p$-regular partition of $B$ with $\partial$-value 0.
With \ref{noth partial Scopes}, Lemma~\ref{lemma Ext alpha}
 and Theorem~\ref{thm Loewy except}, we get the following information on the corresponding part of the quiver of $B$:

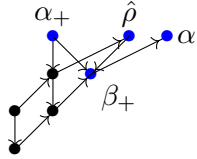
\begin{center}
\begin{tikzpicture}
\coordinate[label=above:$\alpha_+$] (A+) at (0,0);
\coordinate[label=above:$\hat{\rho}$] (R) at (1,0);
\coordinate[label=right:$\alpha$] (A) at (1.5,0);
\coordinate[label=below right:$\beta_+$] (B+) at (0.5,-0.5);
\draw[con] (A+) node{\color{blue}$\bullet$} -- (0,-0.5) node{$\bullet$};
\draw[con] (0,-0.5)  -- (R) node{\color{blue} $\bullet$};
\draw[con] (A+)  -- (B+) node{\color{blue} $\bullet$};
\draw[con] (R) -- (B+);
\draw[con] (B+) -- (A) node{\color{blue}$\bullet$};
\draw[con] (0,-0.5)  -- (0,-1) node{$\bullet$};
\draw[con] (0,-1) -- (B+);
\draw[con] (-0.5,-1) node{$\bullet$} -- (-0.5,-1.5)node{$\bullet$};
\draw[con] (-0.5,-1.5) -- (0,-1);
\draw[con] (-0.5,-1) -- (0,-0.5);
\end{tikzpicture}
\end{center}

Here $D^{\hat{\rho}}\cong Z$.
It remains to verify that $\alpha$ is only connected to $\beta_+$. This is not immediate at this stage, since we do 
not know $\bar{Y}$ and $Y$. However, from \cite{BO1998}, we deduce that the Mullineux conjugate partition of $\alpha$
is $(2(p-1)+1,p)$. We have seen above that this partition has, in the quiver of its block $B_{p-1,0}$, only
one neighbour. Therefore, $\alpha$ also has only one neighbour, which must be $\beta_+$.  
This completes the proof in the
case $k=1$ and $l=p-2$, showing that $B$ then has quiver $Q_{1,p-2}(p)$.

 \medskip

To summarize, we have now proven Theorem~\ref{thm main1} in the case where $k+l\leq p-1$. Next suppose that $k+l=p+1$. Then $B_{k,l}$ is Scopes equivalent to 
$B_{k-1,l-1}$, by Proposition~\ref{prop (2:1) pairs hooks}. By induction $B_{k-1,l-1}$ has quiver $Q_{k-1,l-1}(p)$. Since the Scopes
equivalence preserves the lexicographic ordering as well as $\partial$-values and colours of the corresponding partitions, we deduce
that also $B_{k,l}$ has quiver $Q_{k-1,l-1}(p)$, as claimed. 

Lastly, we need to treat the case where $p+2\leq k+l\leq 2p-1$. To show that $B_{k,l}$ then has quiver $Q_{k-1,l-1}(p)$, we again argue by induction on $k+l$. Since the arguments used in the case $2\leq k+l\leq p-1$ above
translate almost literally, we leave the details to the reader.
\end{proof}

As a consequence of Theorem~\ref{thm main1 details} and Proposition~\ref{prop graph isos}, we now also get the
following result, which immediately implies Theorem~\ref{thm main2}.

\begin{prop}\label{prop main2}
Let $p\geq 5$, and let $k,k',l,l'\in \NN_0$ be such that $(k,1^l)$ and $(k',1^{l'})$ are $p$-cores. Then the blocks
$B_{k,l}$ and $B_{k',l'}$ are Morita equivalent if and only if one of the following cases occurs:

\begin{itemize}
\item[{\rm (i)}] $(k,1^l)=(k',1^{l'})$;
\item[{\rm (ii)}] $(k,1^l)'=(k',1^{l'})$;
\item[{\rm (iii)}] $k+l=p-1$, $k'+l'=p+1$, $(k,1^l)=(k'-1,1^{l'-1})$;
\item[{\rm (iv)}] $k+l=p-1$, $k'+l'=p+1$, $(k,1^l)'=(k'-1,1^{l'-1})$.
\end{itemize}
\end{prop}

\begin{proof}
If $(k,1^l)'=(k',1^{l'})$, then $B_{k',l'}$ is isomorphic, hence Morita equivalent, to 
$B_{k,l}$, an isomorphism being given by tensoring with the sign representation. If $k+l=p-1$, $k'+l'=p+1$ and $(k,1^l)=(k'-1,1^{l'-1})$, then
$B_{k,l}$ and $B_{k',l'}$ are Scopes, hence Morita, equivalent, by Proposition~\ref{prop (2:1) pairs hooks}(c).

So, conversely, suppose that $B_{k,l}$ is Morita equivalent to $B_{k',l'}$. Then the Ext-quivers of $B_{k,l}$ and $B_{k',l'}$
are isomorphic as undirected graphs. We distinguish three cases. If $k+l\leq p-1$ and $k'+l'\leq p-1$, then $B_{k,l}$ has Ext-quiver
$Q_{k,l}(p)$ and $B_{k',l'}$ has Ext-quiver $Q_{k',l'}(p)$, by Theorem~\ref{thm main1 details}. 
By Proposition~\ref{prop graph isos}, we must have $(k,l)=(k',l')$ or $(k',l')=(l+1,k-1)$. In the latter case, $(k',1^{l'})=(k,l)'$.

If $p+1\leq k+l\leq 2p-1$ and $p+1\leq k'+l'\leq 2p-1$, then
$B_{k,l}$ has Ext-quiver
$Q_{k-1,l-1}(p)$ and $B_{k',l'}$ has Ext-quiver $Q_{k-1',l-1'}(p)$, by Theorem~\ref{thm main1 details}. Thus Proposition~\ref{prop graph isos} again
implies $(k,l)=(k',l')$, or $(k',1^{l'})=(k,l)'$.

Lastly, suppose that $0\leq k+l\leq p-1$ and $p+1\leq k'+l'\leq 2p-1$. Then $B_{k,l}$ has Ext-quiver
$Q_{k,l}(p)$ and $B_{k',l'}$ has Ext-quiver $Q_{k-1',l-1'}(p)$, by Theorem~\ref{thm main1 details}. So this time
Proposition~\ref{prop graph isos} implies $(k,l)=(k'-1,l'-1)$ or $(k,l)=(l',k'-2)$; in particular, $k+l=p-1$ and $k'-l'=p+1$.
If $(k,l)=(l',k'-2)$, then $(k,1^l)'=(k'-1,1^{l'-1})'$.

This completes the proof of the proposition.
\end{proof}

\begin{noth}{\bf The case $p=3$.}\,\label{noth p=3}
To conclude this section, we now also consider the case
$p=3$. By \cite{Scopes1991}, there are five Scopes classes of blocks
of weight $2$, representatives of these being given by the blocks with $3$-cores $\emptyset, (1),(2),(1^2)$, and $(3,1^2)$. Note that all of these cores are hook partitions. The decomposition numbers of these blocks are well known; see \cite{JK1981}. As well, the Loewy structures of the respective Specht modules can easily be determined.
Recall further that Theorem~\ref{thm Loewy except} also holds for $p=3$. Thus, examining the principal blocks of $F\mathfrak{S}_6$ and $F\mathfrak{S}_7$ and then arguing 
inductively
as in the proof of Theorem~\ref{thm main1 details}, we see that the Ext-quivers of these five blocks are as drawn below. Here, as before, we order
partitions with respect to their $\partial$-values, and the partitions with the same $\partial$-values with respect to the lexicographic ordering, from left to right. Again,
an arrow $\lambda\to \mu$ in the quivers indicates that $\lambda>\mu$ and $\Ext^1(D^\lambda,D^\mu)\neq \{0\}$.

\medskip

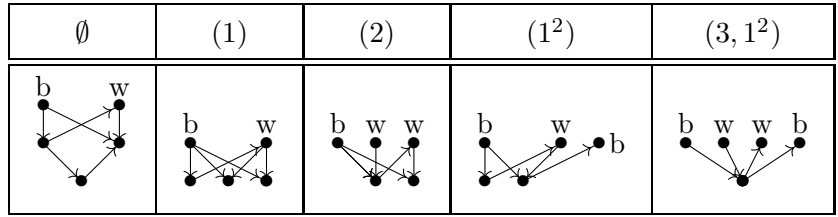
\begin{center}
\begin{tabular}{|c|c|c|c|c|}\hline
$\emptyset$& $(1)$& $(2)$ &$(1^2)$& $(3,1^2)$\\\hline\hline
\begin{tikzpicture}
\coordinate (B+) at (0,-0.5);
\coordinate[label=above:b] (M) at (0,0);
\coordinate[label=above:w] (R) at (1,0);
\coordinate (B) at (1,-0.5);
\coordinate (A) at (0.5,-1);
\draw[con] (M) node{$\bullet$} -- (B) node{$\bullet$};
\draw[con] (M)  -- (B+) node{$\bullet$};
\draw[con] (B+)  -- (R) node{$\bullet$};
\draw[con] (B+)  -- (A) node{$\bullet$};
\draw[con] (A) -- (B);
\draw[con] (R) -- (B);
\end{tikzpicture}&
\begin{tikzpicture}
\coordinate (A+) at (0,-0.5);
\coordinate[label=above:b] (M) at (0,0);
\coordinate[label=above:w] (R) at (1,0);
\coordinate (G) at (1,-0.5);
\coordinate(A) at (0.5,-0.5);
\draw[con] (M) node{$\bullet$} -- (G) node{$\bullet$};
\draw[con] (M)  -- (A+) node{$\bullet$};
\draw[con] (M)  -- (A) node{$\bullet$};
\draw[con] (A+)  -- (R) node{$\bullet$};
\draw[con] (B) -- (G) node{$\bullet$};
\draw[con] (R) -- (G);
\draw[con] (R) -- (A);
\end{tikzpicture}
&
\begin{tikzpicture}
\coordinate[label=above:w] (A) at (0.5,0);
\coordinate[label=above:b] (M) at (0,0);
\coordinate[label=above:w] (G) at (1,0);
\coordinate (B) at (0.5,-0.5);
\draw[con] (M) node{$\bullet$} -- (B) node{ $\bullet$};
\draw[con] (M)  -- (B);
\draw[con] (A)  node{$\bullet$} -- (B);
\draw[con] (B) -- (G) node{$\bullet$};
\draw[con] (G) -- (1,-0.5) node{$\bullet$};
\draw[con] (M) -- (1,-0.5);
\end{tikzpicture}
&
\begin{tikzpicture}
\coordinate[label=above:b] (A+) at (0,0);
\coordinate[label=above:w] (R) at (1,0);
\coordinate[label=right:b] (A) at (1.5,0);
\coordinate (B+) at (0.5,-0.5);
\draw[con] (A+) node{$\bullet$} -- (0,-0.5) node{$\bullet$};
\draw[con] (0,-0.5)  -- (R) node{$\bullet$};
\draw[con] (A+)  -- (B+) node{$\bullet$};
\draw[con] (R) -- (B+);
\draw[con] (B+) -- (A) node{$\bullet$};
\end{tikzpicture}&
\begin{tikzpicture}
\coordinate[label=above:b] (A) at (0,0);
\coordinate[label=above:w] (B) at (0.5,0);
\coordinate[label=above:w] (C) at (1,0);
\coordinate[label=above:b] (D) at (1.5,0);
\coordinate (E) at (0.75,-0.5);
\draw[con] (A) node{$\bullet$} -- (E) node{ $\bullet$};
\draw[con] (B) node{$\bullet$} -- (E);
\draw[con] (E) node{$\bullet$} -- (C)node{$\bullet$};
\draw[con] (E) node{$\bullet$} -- (D)node{$\bullet$};
\end{tikzpicture}\\\hline
\end{tabular}
\end{center}

\medskip

The blocks with $3$-cores $(2)$ and $(1^2)$, respectively, are isomorphic via tensoring with the sign representation, and their quivers
are isomorphic as undirected graphs. The above table also shows that the quivers of the blocks labelled by $\emptyset, (1), (2)$ and $(3,1^2)$ are
pairwise non-isomorphic as undirected graphs. In particular, these partitions parametrize the four Morita equivalence classes of $3$-blocks  of symmetric groups of weight 2.
\end{noth}

\begin{appendix}

\section{Abacus Combinatorics}\label{sec abacus}

The purpose of this short section is to collect some useful 
abacus combinatorics that we use repeatedly in this article. Most of this is 
well known to the experts and can easily be verified. We, therefore, omit most of the details, but 
present some illustrating examples. Throughout this section, let $p$ be a prime. Our notation will be chosen in 
accordance with Section~\ref{sec pre}.

\begin{noth}{\bf Hook lengths and abacus displays.}\,\label{noth hooks diagram}
Suppose that $\lambda$ is a partition of $n$ with $p$-core $\kappa=(\kappa_1,\ldots,\kappa_t)$ and $p$-weight $w$. We display $\lambda$ on an
abacus with $p$ runners, labelled from $1$ to $p$. We denote this abacus by $\Gamma_\lambda$.
By \cite[2.7.13]{JK1981}, there is a bijection between the entries in the hook diagram of $\lambda$ divisible by $p$ and the set of pairs $((r,i),(s,i))$ such that
there is a bead on runner $i$ in row $r$ and a gap on runner $i$ in row $s<r$ of $\Gamma_\lambda$. The entry in the hook diagram
of $\lambda$ then equals $p(r-s)$. Moreover, one can also read off  that the leg length $l$ of the hook in question
equals  the number of beads passed when moving the bead from position
$(r,i)$ to position $(s,i)$. 
\end{noth}

\begin{expl}\label{expl hook lengths}
Suppose that $p=3$ and $\lambda=(6,3^3,2^2)$. Then $\lambda$ has $p$-core $\kappa=(3,1)$ and $p$-weight $5$. We consider 
the abacus display $\Gamma_\lambda$ with six beads as well as the hook diagram $H_\lambda$ of $\lambda$:

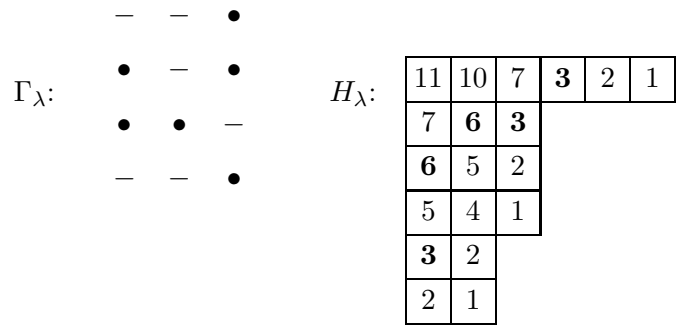
\begin{center}
$\Gamma_\lambda$: \quad \begin{tabular}{ccc}
$-$&$-$&$\bullet$\\
$\bullet$&$-$&$\bullet$\\
$\bullet$&$\bullet$&$-$\\
$-$&$-$&$\bullet$
\end{tabular}
\quad\quad 
$H_\lambda$:\quad \begin{ytableau}
11&10&7&{\bf 3}&2&1\cr
7&{\bf 6}&{\bf 3}\\
{\bf 6}&5&2\\
5&4&1\\
{\bf 3}&2\\
2&1
\end{ytableau}
\end{center}

\medskip

So there are five entries in $H_\lambda$ that are divisible by $3$. We record, on the one hand, their positions in $H_\lambda$, the lengths (hl) as well as
the leg lengths (ll) of the corresponding hooks and, on the other hand, the respective pairs of positions in $\Gamma_\lambda$
under the above-mentioned bijection:

\medskip

\begin{center}
\begin{tabular}{|c|c|c|c|}\hline
position in $H_\lambda$& hl & ll& pair of positions in $\Gamma_\lambda$\\\hline\hline
$(1,4)$& $3$& $0$& $((4,3),(3,3))$\\\hline
$(2,2)$& $6$&$4$&$((3,2),(1,2))$\\\hline
$(2,3)$&$3$&$2$&$((3,2),(2,2))$\\\hline
$(3,1)$&$6$&$4$&$((3,1),(1,1))$\\\hline
$(5,1)$&$3$&$1$&$((2,1),(1,1))$\\\hline
\end{tabular}
\end{center}
\end{expl}

\medskip

\begin{noth}{\bf Colours and $\partial$-values of weight-$2$ partitions.}\,\label{noth colour weight 2} 
Suppose now that $p\geq 3$ and that $\lambda=(\lambda_1,\ldots,\lambda_s)$ is a partition of $n$ 
with $p$-core $\kappa=(\kappa_1,\ldots,\kappa_t)$ and $p$-weight $2$. Recall from \ref{noth partial} the definition
of $\partial(\lambda)$.
We consider an $[m_1,\ldots,m_p]$-abacus display $\Gamma_\lambda$ of $\lambda$ with at least $2p+t$ beads. In the notation of \ref{noth abacus},
there are three possibilities: $\lambda=\langle i\rangle$, $\lambda=\langle i,i\rangle$, or $\lambda=\langle i,j\rangle$, for
some $1\leq i<j\leq p$.   One of our next aims is to show how to determine the colour of $\lambda$ in the case $\partial(\lambda)=0$, using the
abacus display $\Gamma_\lambda$.
This information was needed, for instance, in the proof of Lemma~\ref{lemma l and r}.

Suppose that $\lambda=\langle i\rangle$. Then \ref{noth hooks diagram} shows that $H_\lambda$ has an entry equal to $2p$, and the leg length
of the corresponding hook equals the number of beads passed when moving the (unique) movable bead on runner $i$ two positions up. Moreover, $H_\lambda$
has an entry equal to $p$. The  leg length
of the corresponding hook equals the number of beads passed when moving the movable bead on runner $i$ one position up.

If $\lambda=\langle i,i\rangle$, then, by \ref{noth hooks diagram}, $H_\lambda$ also has an entry equal to $2p$. The leg length of
the corresponding hook equals the number of beads passed when putting the lower of the two movable beads on runner $i$ two positions up.
As well, $H_\lambda$ has an entry equal to $p$. The leg length of the corresponding hook equals the number
of beads passed when moving the upper of the two movable beads on runner $i$ one position up.

Lastly, suppose that $\lambda=\langle i,j\rangle$, for
some $i<j$. By \ref{noth hooks diagram}, $H_\lambda$ then has two entries equal to $p$.
The hook lengths of the corresponding hooks equal the number of beads passed when moving the movable bead on runner $i$ (respectively,
the movable bead on runner $j$) one position up.

From now on, suppose, in addition, that $\partial(\lambda)=0$. Suppose, moreover, that the movable beads lie in positions $(x,i)$ and $(y,j)$ of $\Gamma_\lambda$. If $x<y$, then
we consider  rows $x-1,\ldots,y$ of $\Gamma_\lambda$:

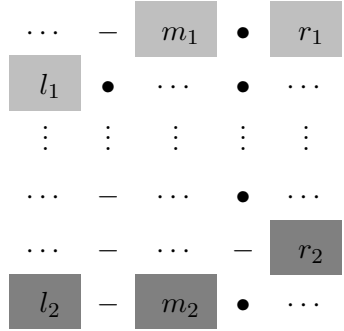
\begin{center}
\begin{tabular}{ccccc}
$\cdots$&$-$&{\cellcolor{lightgray} $m_1$}&$\bullet$&{\cellcolor{lightgray} $r_1$}\\
{\cellcolor{lightgray} $l_1$}&$\bullet$&$\cdots$&$\bullet$&$\cdots$\\
$\vdots$&$\vdots$&$\vdots$&$\vdots$&$\vdots$\\
$\cdots$&$-$&$\cdots$&$\bullet$&$\cdots$\\
$\cdots$&$-$&$\cdots$&$-$&{\cellcolor{gray} $r_2$}\\
{\cellcolor{gray} $l_2$}&$-$&{\cellcolor{gray} $m_2$}&$\bullet$&$\cdots$\\
\end{tabular}
\end{center}

Here as well as in all subsequent abacus displays, $m,m_1,m_2,l_1,l_2,r_1,r_2,l,r$ are the numbers of beads in the respective parts of $\Gamma_\lambda$, as shown in the diagrams.
Since $\lambda$ has weight $2$, we must have $m_1\geq m_2$, $l_1\geq l_2$ and $r_1\geq r_2$.
But then, by \ref{noth hooks diagram} we get that $\partial(\lambda)=|(l_1+r_1+m_1+1)-(l_2+m_2+r_2)|>0$, a contradiction.

If $x>y+1$, then rows $y-1,\ldots,x$ of $\Gamma_\lambda$ have shape

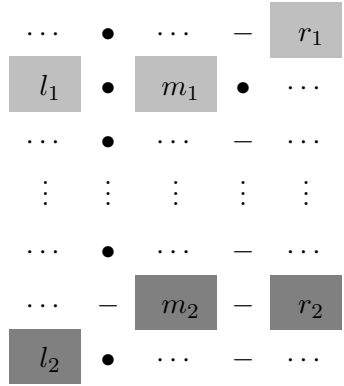
\begin{center}
\begin{tabular}{ccccc}
$\cdots$&$\bullet$&$\cdots$&$-$&{\cellcolor{lightgray} $r_1$}\\
{\cellcolor{lightgray} $l_1$}&$\bullet$&{\cellcolor{lightgray} $m_1$}&$\bullet$&$\cdots$\\
$\cdots$&$\bullet$&$\cdots$&$-$&$\cdots$\\
$\vdots$&$\vdots$&$\vdots$&$\vdots$&$\vdots$\\
$\cdots$&$\bullet$&$\cdots$&$-$&$\cdots$\\
$\cdots$&$-$&{\cellcolor{gray} $m_2$}&$-$&{\cellcolor{gray} $r_2$}\\
{\cellcolor{gray} $l_2$}&$\bullet$&$\cdots$&$-$&$\cdots$\\
\end{tabular}
\end{center}

Thus, we get $\partial(\lambda)=|(m_1+l_1+r_1+1)-(l_2+r_2+m_2)|>0$, again a contradiction.

If $x=y+1$, then rows $x-2,x-1,x$ of $\Gamma_\lambda$ have shape

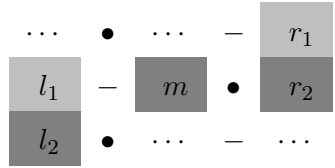
\begin{center}
\begin{tabular}{ccccc}
$\cdots$&$\bullet$&$\cdots$&$-$&{\cellcolor{lightgray} $r_1$}\\
{\cellcolor{lightgray} $l_1$}&$-$&{\cellcolor{gray} $m$}&$\bullet$&{\cellcolor{gray} $r_2$}\\
{\cellcolor{gray} $l_2$}&$\bullet$&$\cdots$&$-$&$\cdots$\\
\end{tabular}
\end{center}

We get $\partial(\lambda)=|(l_1+r_1+m)-(l_2+m_2+m+1)|$, thus $l_1+r_1=l_2+m_2+1$.

Lastly, if the movable bead on runner $i$ lies in the same row $x$ as the movable bead on runner $j$, then we consider 
rows $x$ and $x-1$ of $\Gamma_\lambda$:

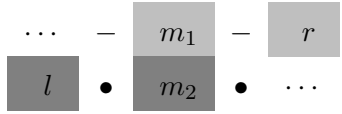
\begin{center}
\begin{tabular}{ccccc}
$\cdots$& $-$ &{\cellcolor{lightgray} $m_1$}&$-$&{\cellcolor{lightgray} $r$}\\
{\cellcolor{gray} $l$}& $\bullet$ &{\cellcolor{gray} $m_2$}&$\bullet$&$\cdots$\\
\end{tabular}

\end{center}

Note that still $m_1\geq m_2$, since there is no movable bead on any runner different from $i$ and $j$.
So there are two ways to obtain $[\kappa]$ from $[\lambda]$, depending on which bead is moved first. If 
we move the bead on runner $i$ first, then from \ref{noth hooks diagram} we get
$0=\partial(\lambda)=|(l+r+m_1)-(l+r+m_2)|$, thus $m_1=m_2$.
Analogously, also in the case that we first move the bead on runner $j$, we get
$0=\partial(\lambda)=|(l+r+m_2+1)-(l+r+m_1+1)|$, thus $m_1=m_2$.
\end{noth}

\begin{expl}\label{expl colour weight 2}
To illustrate the combinatorics in \ref{noth colour weight 2}, we consider $p=3$ and the principal block of $F\mathfrak{S}_7$ with $3$-core 
$(1)$. Moreover, we consider a $[2,3,2]$-abacus, for every partition of this block.

\smallskip

(a)\, For $\lambda=(5,2)=\langle 3\rangle$, we have

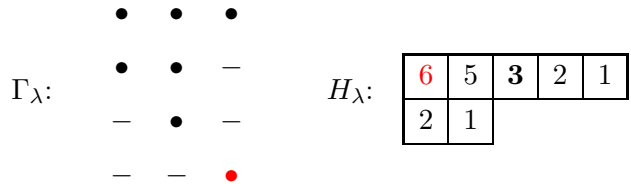
\begin{center}
$\Gamma_\lambda$: \quad \begin{tabular}{ccc}
$\bullet$&$\bullet$&$\bullet$\\
$\bullet$&$\bullet$&$-$\\
$-$&$\bullet$&$-$\\
$-$&$-$&{\color{red} $\bullet$}
\end{tabular}
\quad\quad 
$H_\lambda$:\quad \begin{ytableau}
{\color{red} 6}&5&{\bf 3} &2&1\cr
2&1
\end{ytableau}
\end{center}

The leg length of the hook of length $6$ is $1$, the leg length of the hook of length $3$ is $0$. Moving the red bead on runner 3 one position up, we do not
pass any bead, moving this bead two positions up we pass one bead.

\smallskip

(b)\, For $\lambda=(2^2,1^3)=\langle 3,3\rangle$, we have

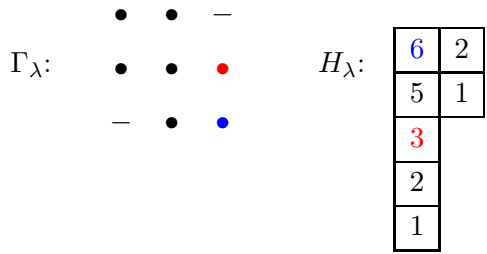
\begin{center}
$\Gamma_\lambda$: \quad \begin{tabular}{ccc}
$\bullet$&$\bullet$&$-$\\
$\bullet$&$\bullet$&{\color{red} $\bullet$}\\
$-$&$\bullet$&{\color{blue} $\bullet$}\\
\end{tabular}
\quad\quad 
$H_\lambda$:\quad \begin{ytableau}
{\color{blue} 6}&2\cr
5&1\cr
{\color{red} 3}\cr
2\cr
1
\end{ytableau}
\end{center}

The hook of length 6 has leg length 4, the hook of length 3 has leg length 2. Moving the blue bead two positions up we pass four beads, moving
the red bead one position up we pass two beads.

\smallskip

(c)\, For $\lambda=(4,3)=\langle 3,2\rangle$, we have

\begin{center}
$\Gamma_\lambda$: \quad \begin{tabular}{ccc}
$\bullet$&$\bullet$&$\bullet$\\
$\bullet$&$\bullet$&$-$\\
$-$&$-$&{\color{blue} $\bullet$}\\
$-$&{\color{red} $\bullet$}&$-$
\end{tabular}
\quad\quad 
$H_\lambda$:\quad \begin{ytableau}
5&4&{\color{red} 3}&1\cr
{\color{blue} 3}&2&1
\end{ytableau}
\end{center}

The hook of length 3 corresponding to the red entry in $H_\lambda$ has leg length 1, the hook of length 3 corresponding to the blue entry in $H_\lambda$ has leg length 0. 
Moroever, $\partial(\lambda)=0$ and, since the larger leg length of the hooks just mentioned is odd, $\lambda$ is white.
Moving the red bead one position up we pass one bead, moving the blue bead one position up we do not pass any bead.

\smallskip

(d)\, For $\lambda=(2^3,1)=\langle 3,1\rangle$, we have

\begin{center}
$\Gamma_\lambda$: \quad \begin{tabular}{ccc}
$\bullet$&$\bullet$&$\bullet$\\
$-$&$\bullet$&$-$\\
{\color{red} $\bullet$}&$\bullet$&{\color{blue} $\bullet$}\\
\end{tabular}
\quad\quad 
$H_\lambda$:\quad \begin{ytableau}
5&{\color{blue} 3}\cr
4&2\cr
{\color{red} 3}&1\cr
1
\end{ytableau}
\end{center}

The hook of length 3 corresponding to the red entry in $H_\lambda$ has leg length 1, the hook of length 3 corresponding to the blue entry in $H_\lambda$ has leg length 2. 
Moroever, $\partial(\lambda)=0$ and, since the larger leg length of the hooks just mentioned is even, $\lambda$ is black.
Moving the red bead one position up we pass one bead, moving the blue bead one position up we pass two beads.
\end{expl}

%%%%%%%%%%%%%%%%%%%%%%%%%%%%%%%%%%%%%%%%%%%%%%%%%%%%%%%%%%%%%%%%%%%%

\section{The Principal $p$-Blocks of $\mathfrak{S}_{2p}$ and $\mathfrak{S}_{2p+1}$}\label{sec B0}

Let  $p\geq 5$ be a prime. 
In the following we shall show that the principal block $B_{0,0}:=B_\emptyset$ of $F\mathfrak{S}_{2p}$ has
Ext-quiver $Q_{0,0,}(p)$, and the principal block $B_{1,0}:=B_{(1)}$ of $F\mathfrak{S}_{2p+1}$ has Ext-quiver $Q_{1,0}(p)$. We should also like to emphasize that the structure of the Ext-quivers of $B_{0,0}$ and $B_{1,0}$ can be deduced
from \cite{Scopes1995} and \cite{ChuangTan2001}. As well, these quivers appear in work of Martin \cite{Martin1989,Martin1990}. 

We shall give an elementary, self-contained proof here, the most important 
information being given by the decomposition matrix of $B_{0,0}$. The latter has been used in several publications, such as \cite{Martin1989,Scopes1995}, that
refer to the book \cite{Robinson}, which, however, does not provide too many details. We shall, thus, give a brief account in the form we shall
use it. Our strategy will then be to induce indecomposable projective $F\mathfrak{S}_{2p-1}$-modules to $F\mathfrak{S}_{2p}$.

\begin{noth}\label{noth abacus and order}
{\bf Abacus labelling and order on partitions.}\, Consider the principal block $B_{0,0}$ of $F\mathfrak{S}_{2p}$. We display the $p$-core $\emptyset$ of $B_{0,0}$ as well as all
partitions of $B_{0,0}$ on a $[2^p]$-abacus. As in \ref{noth abacus}(b) we identify partitions with their
respective abacus displays. 

\smallskip

(a)\, With the above notation, the lexicographic  ordering on partitions translates as follows:

\smallskip

\quad (i)\, For all $p\geq j>i\geq 1$ and $p\geq a\geq b\geq 1$, we have $\langle j\rangle >\langle i\rangle >\langle a,b\rangle$.

\quad (ii)\, For all $p\geq j>  i\geq 1$ and $p\geq j' >  i'\geq 1$, we have 
$$\langle j,i\rangle>\langle j',i'\rangle \Leftrightarrow j>j'\,, \text{ or } j=j' \text{ and } i>i'\,.$$

This leaves to determine the positions of the 
partitions  $\langle i, i\rangle$ with respect to the lexicographic ordering.
 
\quad (iii)\, For all $p\geq i\geq 3$, we have $\langle i,1\rangle>\langle i,i\rangle> \langle i-1, i-2\rangle$
 and $\langle 2, 1\rangle >   \langle  2, 2 \rangle > \langle 1, 1\rangle$.

\medskip

(b)\, The total number of partitions of $B_{0,0}$ equals $2p+\binom{p}{2}$, of which $\binom{p+1}{2}-1$ are $p$-regular.
Moreover, the following $p+1$ partitions are $p$-singular:
\begin{equation}\label{eqn p-sing}
\langle 2,1\rangle \text{ and } \langle j,j\rangle\,, \text{ for } 1\leq j\leq p\,.
\end{equation}
All partitions in (\ref{eqn p-sing}) are $p$-restricted.
\end{noth}

We now start to induce indecomposable projective $F\mathfrak{S}_{2p-1}$-modules to $B_{0,0}$.

\begin{noth}\label{noth induce from w=0}
{\bf Inducing from a block of $F\mathfrak{S}_{2p-1}$ of weight 0.}\, 
Consider  the block $B_{(p,1^{p-1})}$ of $F\mathfrak{S}_{2p-1}$ of weight $0$.
It has a unique simple module (up to isomorphism), namely
$D^{(p,1^{p-1})}\cong S^{(p,1^{p-1})}\cong P^{(p,1^{p-1})}$. The induced module $\ind_{\mathfrak{S}_{2p-1}}^{\mathfrak{S}_{2p}}(P^{(p,1^{p-1})})$
is of course also projective and, by the Branching Rules, it admits a Specht filtration whose Specht quotients are labelled
by those partitions of $2p$ that are obtained by adding a node to $(p,1^{p-1})$. Thus
$\ind_{\mathfrak{S}_{2p-1}}^{\mathfrak{S}_{2p}}(P^{(p,1^{p-1})})$ has a Specht filtration with quotients labelled by $(p+1,1^{p-1})$,
$(p,1^p)$ and $(p,2,1^{p-2})$. All of these partitions belong to $B_{0,0}$, so that
$$\ind_{\mathfrak{S}_{2p-1}}^{\mathfrak{S}_{2p}}(P^{(p,1^{p-1})})=P^{(p,1^{p-1})}\uparrow^{B_{0,0}}\,.$$

Now assume that $P:= P^{(p,1^{p-1})}\uparrow^{B_{0,0}}$ was decomposable. Then this module would
have an
indecomposable (projective) direct summand isomorphic to one of the Specht modules labelled
by $(p+1,1^{p-1})$, $(p,1^p)$, or $(p,2,1^{p-2})$. But such a Specht module would then be simple and projective, which is not possible
in a block of weight greater than $0$. Note that this argument uses the hypothesis $p\geq 5$, so that, by
\cite{HN2004} the
multiplicity of a Specht module in any Specht filtration of $P$ is unique. Alternatively, one could
 also examine the endomorphism algebra of $P$.
 
Hence $P$
must be indecomposable. The labels of its  Specht quotients with respect to the
$[2^p]$-abacus are:

$$(p+1, 1^{p-1}) = \langle 1\rangle, \ (p, 2, 1^{p-2}) = \langle p, 1\rangle, \ \ (p, 1^p) = \langle p, p\rangle\,.
$$

Recall from \ref{noth Specht filtration}, that  if $S^\mu$ is isomorphic to a subquotient of any Specht filtration of $P^\lambda$, then
$\mu=\lambda$ or $\lambda>\mu$. Given the Specht quotients of $P$, this shows that 
$$P=P^{(p,1^{p-1})}\uparrow^{B_{0,0}}\cong P^{(p+1,1^{p-1})}=P^{\langle 1\rangle}\,.$$
\end{noth}

\begin{noth}\label{noth induce from w=1}
{\bf Inducing from a block of $F\mathfrak{S}_{2p-1}$ of weight 1.}\, 
(a)\, We discuss first an appropriate parametrization of the partitions in blocks of $F\mathfrak{S}_{2p-1}$ of weight 1  that are relevant 
for our investigations, so that we end up with our fixed
labelling. For $B_{0, 0}$ we have previously used the $[2^p]$-abacus;  we have numbered the runners as $1, 2, \ldots, p$ from left to right, and used this to parametrize the
partitions.

We can also represent all partitions in $B_{0,0}$ on  the $[3^a, 2^{p-a}]$-abacus, for any $1\leq a <p$. 
More precisely, starting with the $[2^p]$-abacus, we insert one bead at each of the positions $1, 2, \ldots, a$, and then shift
all other beads and gaps by $a$ places. This means that a gap on  runner $i$ in the $[2^p]$-abacus is now on the $(i+a)$th runner (taking $i+a$ modulo $p$).
For example, a gap on runner $1$ of the $[2^p]$-abacus becomes a gap on runner $1+a$ of the $[3^a, 2^{p-a}]$-abacus.  
We label the runners of the new abacus  cyclically,  so that the rightmost runner has label $p-a$, that is,
the labels of the runners of this new abacus are  
$$(p-a+1, p-a+2, \ldots, p, 1, 2, \ldots, p-a). 
$$
Then a gap on the runner {\it  labelled with $j$} comes from a gap that {\it  is}  on  runner $j$ of the $[2^p]$-abacus. 
We use this observation in the following.

\smallskip

(b)\, Now we need to determine which blocks $B$ of $F\mathfrak{S}_{2p-1}$ of weight 1 contain a Specht module $S^\lambda$
with $S^\lambda\uparrow^{B_{0,0}}\neq\{0\}$. Equivalently, such a block $B$ satisfies $S^\mu\downarrow_B\neq \{0\}$, for some
Specht module $S^\mu$ in $B_{0,0}$. Thus, let  $S^\mu$ be a Specht module in $B_{0,0}$ and represent $\mu$ on the
$[2^p]$-abacus. To obtain the (uniquely determined) multiplicity of a Specht $F\mathfrak{S}_{2p-1}$-module
in any Specht filtration of $S^\mu\downarrow_{\mathfrak{S}_{2p-1}}$, we distinguish three cases:
suppose first that any gap occurs on some runner $j$ with $2\leq j \leq p-1$.  Then any bead which can be moved one place to the left occurs on some runner 
$i$ with $2\leq i\leq p$ and can be moved to runner $i-1$. In each case, the corresponding Specht quotient $S^{\lambda}$ of $S^\mu\downarrow_{\mathfrak{S}_{2p-1}}$ 
belongs  to a block whose $p$-core is represented on the $[2^k,3,1,2^l]$-abacus, for some
$k,l\in\{0,\ldots,p\}$ with $k+l=p-2$. These are precisely the hook partitions of $p-1$.

Second, suppose that there is a gap on runner $1$, that is $\mu=\langle 1\rangle$ or $\mu=\langle 1,i\rangle$, for some $i\in\{1,\ldots,p\}$.
If $\mu=\langle 1,2\rangle=(2^p)$, then $S^\mu\downarrow_{\mathfrak{S}_{2p-1}}\cong S^{(2^{p-1},1)}$, which belongs to the block
labelled by the hook partition $(2,1^{p-3})$ of $p-1$.

If $\mu=\langle 1,1\rangle=(1^{2p})$, then $S^\mu\downarrow_{\mathfrak{S}_{2p-1}}\cong S^{(1^{2p-1})}$, which belongs to the block
labelled by the hook partition $(1^{p-1})$ of $p-1$.

If $\mu=\langle 1\rangle=(p+1,1^{p-1})$, then $S^\mu\downarrow_{\mathfrak{S}_{2p-1}}$ has a Specht quotient isomorphic to $S^{(p+1,1^{p-2})}$ in the block
with $p$-core $(1^{p-1})$, and a Specht quotient isomorphic to $S^{(p,1^{p-1})}$ in the block of weight 0.

If $\mu=\langle 1,i\rangle$ with $i\geq 3$, then the Young diagram of $\mu$ has three removable nodes, which correspond to beads on runners $2$ and $i$ and $i+1$ (when $i<p$).
Moving the bead one place to the left in each case  yields a partition in a weight-1 block labelled by a hook partition, as we have just seen above. 
Let $i=p$, so that  $\langle 1, p\rangle = (p, 2, 1^{p-2})$, then  the restriction to  
$\mathfrak{S}_{2p-1}$ has a Specht quotient labelled by  the partition $(p,1^{p-1})$ of $2p-1$ belonging to the block of weight 0, which has to be the case not yet covered. 

Third, suppose there is a gap on runner $p$ with $\mu$ not yet considered, then $\mu = \langle i, p\rangle$ for $2\leq i\leq p$ or $\mu = \langle p\rangle$. 

If $\mu = \langle i, p\rangle$ for $2\leq i < p-1$ then  the Young diagram has three removable nodes. The abacus presentation yields  three
beads which can be moved by one place to the left and in each case the corresponding Specht quotient of the restriction lies in a weight-1 block labelled by a hook partition.

If $\mu = \langle p-1, p\rangle = (p^2)$ there is only one removable node, and the restriction is isomorphic to a  Specht module in  a block whose core is a hook partition, similarly
for $\mu = \langle p \rangle = (2p)$.

Let $\langle p, p\rangle = (p, 1^p)$, its restriction 
 to $\mathfrak{S}_{2p-1}$ has Specht quotients labelled by $(p-1, 1^p)$, in the block with core $(p-1)$,  and $(p, 1^{p-1})$, 
in a block of weight $0$. 
\smallskip

To summarize, we now know that, whenever, $S^\lambda$ is a Specht $F\mathfrak{S}_{2p-1}$-module in a block of weight 1
with $S^\lambda\uparrow^{B_{0,0}}\neq\{0\}$, then the $p$-core of the corresponding block is one of the $p-1$ hook partitions of $p-1$.
Therefore, for $s\in\{2,\ldots,p\}$,
we from now on denote the block of $F\mathfrak{S}_{2p-1}$ with $p$-core $(s-1,1^{p-s})$
by $B_s$. The $p$-core as well as all partitions of $B_s$ will be represented on a $[4,2,3^{p-2},2^{s-2}]$-abacus.
As above, we label the runners of this abacus from left to right, by $s-1,s,s+1,\ldots,p,1,2,\ldots,s-2$. Thus, in particular, for $s<p$, the rightmost runner with three beads has label $p$. 

With this notation, the partitions of $B_s$ are obtained by moving exactly one bead on some runner $i$ of the $[4,2,3^{p-2},2^{s-2}]$-abacus
of $(s-1,1^{p-s})$ one position down; we shall denote the resulting partition by $\langle i\rangle$ (when $s$ is fixed). 
Then,
for a fixed $s\in\{1,\ldots,p\}$, and with \ref{noth abacus and order}, we obtain the following lexicographic ordering on the partitions of $B_s$:

\begin{itemize}

\item[(i)] $\langle 1\rangle=\langle s-1\rangle>\langle p\rangle>\langle p-1\rangle>\cdots >\langle 3\rangle>\langle 2\rangle \,,\text{ for } s=2\,;$

\item[(ii)] $\langle s-1\rangle =\langle p\rangle >\langle p-1\rangle >\cdots >\langle s+1\rangle >\langle s-2\rangle >\cdots \langle 1\rangle >\langle s\rangle \,,\text{ for } 3\leq s<p\,;$

\item[(iii)] $\langle p-1\rangle =\langle p-2\rangle >\cdots >\langle 1\rangle >\langle p\rangle \,,\text{ for } s=p\,.$
 
\end{itemize}

\smallskip

(c)\, For $s\in\{2,\ldots,p\}$, the decomposition matrix of $B_s$ is well known from the theory of blocks with cyclic defect groups. Thus,
by Brauer Reciprocity and \ref{noth Specht filtration}, one also knows the Specht factors occurring in any Specht filtration of any indecomposable projective $B_s$-module.
More precisely, for every $p$-regular partition $\lambda$ of $B_s$, one has $P^\lambda\sim S^\lambda\oplus S^{\tilde{\lambda}}$, where
$\tilde{\lambda}$ denotes the lexicographically next smaller partition of $B_s$. So, with (i)-(iii) above, this gives the following information that
will be crucial for the proof of Theorem~\ref{thm induce from w=1} below:

\smallskip

\begin{center}
\begin{tabular}{|c|c|c|}\hline
$s=2$ & $3\leq s\leq p-1$ &$s=p$\\\hline\hline
$P^{\langle s-1\rangle}\sim S^{\langle s-1\rangle}\oplus S^{\langle p\rangle}$&$P^{\langle s-1\rangle}\sim S^{\langle s-1\rangle}\oplus S^{\langle s-2\rangle}$&$P^{\langle s-1\rangle}\sim S^{\langle s-1\rangle}\oplus S^{\langle s-2\rangle}$\\
$P^{\langle p\rangle}\sim S^{\langle p\rangle}\oplus S^{\langle p-1\rangle}$&$P^{\langle p\rangle}\sim S^{\langle p\rangle}\oplus S^{\langle p-1\rangle}$&$P^{\langle s-2\rangle}\sim S^{\langle s-2\rangle}\oplus S^{\langle s-3\rangle}$\\
$\vdots$&$\vdots$&$\vdots$\\
$\vdots$&$P^{\langle s+1\rangle}\sim S^{\langle s+1\rangle}\oplus S^{\langle s-2\rangle}$&$\vdots$\\
$\vdots$&$\vdots$&$\vdots$\\
$P^{\langle 3\rangle}\sim S^{\langle 3\rangle}\oplus S^{\langle 2\rangle}$&$P^{\langle 1\rangle}\sim S^{\langle 1\rangle}\oplus S^{\langle s\rangle}$&$P^{\langle 1\rangle}\sim S^{\langle 1\rangle}\oplus S^{\langle s\rangle}$\\\hline
\end{tabular}
\end{center}

\smallskip

(d)\, By the Branching Theorem \cite[Theorem~9.2]{James1978}, whenever $s\in\{2,\ldots,p\}$ and $i\in\{1,\ldots,p\}$, the $F\mathfrak{S}_{2p}$-module $\ind_{\mathfrak{S}_{2p-1}}^{\mathfrak{S}_{2p}}(S^{\langle i\rangle})$
has a Specht filtration. The Specht factors occurring in any such filtration are unique up to isomorphism, since $p\geq 5$,
and their labelling partitions are obtained by moving a bead on some runner of $\langle i\rangle$ one position to the right.
In particular, the block component $S^{\langle i\rangle}\uparrow^{B_{0,0}}$ has a Specht filtration, and the Specht factors occurring
are labelled by those partitions that are obtained by moving a bead from runner $s-1$ of $\langle i\rangle$ to runner $s$. Thus, for
each $s\in\{2,\ldots,p\}$, we get

$$
S^{\langle j\rangle}\uparrow^{B_{0,0}}\sim \begin{cases} S^{\langle s-1\rangle}\oplus S^{\langle s\rangle}\,, & \text{ if } j=s-1\,,\\
S^{\langle s-1,s-1\rangle}\oplus S^{\langle s,s\rangle}\,, &\text{ if } j=s\,,\\
S^{\langle i,s-1\rangle}\oplus S^{\langle i,s\rangle}\,, &\text{ if } s<j\leq p\,,\\
 S^{\langle s-1,i\rangle}\oplus S^{\langle s,i\rangle}\,, &\text{ if } 1\leq j<s-1\,.\end{cases}$$
Alternatively, one may use arguments as in \cite[(2.6)]{Donkin}, which also work when $p=3$. 
\end{noth}

\begin{thm}\label{thm induce from w=1}
Up to isomorphism there are $\binom{p+1}{2}-2$ indecomposable projective $B_{0,0}$-modules
that are induced from a block of $F\mathfrak{S}_{2p-1}$ of weight $1$. Their labelling partitions 
in $[2^p]$-abacus notation and their
Specht factors are as follows

\begin{itemize}
\item[\rm{(a)}] $P^{\langle p\rangle}\sim S^{\langle p\rangle}\oplus S^{\langle p-1\rangle}\oplus S^{\langle p,p-2\rangle}\oplus S^{\langle p-1,p-2\rangle}$;
\item[\rm{(b)}] $P^{\langle s\rangle}\sim S^{\langle s\rangle}\oplus S^{\langle s-1\rangle}\oplus S^{\langle p,s\rangle}\oplus S^{\langle p,s-1\rangle}$, for $s\in \{2,\ldots,p-1\}$;
\item[\rm{(c)}] $P^{\langle s,1\rangle}\sim S^{\langle s,1\rangle}\oplus S^{\langle s-1,1\rangle}\oplus S^{\langle s,s\rangle}\oplus S^{\langle s-1,s-1\rangle}$, for $s\in \{3,\ldots,p\}$;
\item[\rm{(d)}] $P^{\langle s+1,s\rangle}\sim S^{\langle s+1,s\rangle}\oplus S^{\langle s+1,s-1\rangle}\oplus S^{\langle s,s-2\rangle}\oplus S^{\langle s-1,s-2\rangle}$, for $s\in \{3,\ldots,p-1\}$;
\item[\rm{(e)}] $P^{\langle 3,2\rangle}\sim S^{\langle 3,2\rangle}\oplus S^{\langle 3,1\rangle}\oplus S^{\langle 2,2\rangle}\oplus S^{\langle 1,1\rangle}$;
\item[\rm{(f)}] $P^{\langle r,s\rangle}\sim S^{\langle r,s\rangle}\oplus S^{\langle r-1,s\rangle}\oplus S^{\langle r,s-1\rangle}\oplus S^{\langle r-1,s-1\rangle}$,
for $p\geq r>s>1$ and $r-s>1$.
\end{itemize}
\end{thm}

\begin{proof}
Since $F\mathfrak{S}_{2p}\cong \ind_{\mathfrak{S}_{2p-1}}^{\mathfrak{S}_{2p}}(F\mathfrak{S}_{2p-1})$, every indecomposable projective
$F\mathfrak{S}_{2p}$-module is isomorphic to a direct summand of the induction of some indecomposable projective $F\mathfrak{S}_{2p-1}$-module.
As we have seen in \ref{noth induce from w=0}, $P^{\langle 1\rangle}$  is the unique indecomposable projective $B_{0,0}$-module that is induced from
a block of  $F\mathfrak{S}_{2p-1}$ of weight 0; moreover, it has precisely three Specht factors. 

By \ref{noth induce from w=1} (c), (d), we obtain precisely $x:=\binom{p+1}{2}-2$ pairwise non-isomorphic
projective $B_{0,0}$-modules 
$R_1,\ldots,R_x$ that are obtained by inducing the indecomposable $F\mathfrak{S}_{2p-1}$-modules in blocks
of weight 1 to $B_{0,0}$. Furthermore, each of these block inductions has precisely four pairwise non-isomorphic Specht factors; the lexicographically largest labelling
partition $\lambda_i$ of $R_i$ is always $p$-regular, and $\lambda_i\neq \lambda_j$, for $i\neq j$.
Note that $P^{\langle 1\rangle}$ cannot be isomorphic to a direct summand of any $R_i$, since otherwise there would be a projective
$B_{0,0}$-module with only one Specht factor, which is impossible in a block of weight 2.

We may suppose that $\lambda_1>\lambda_2>\cdots>\lambda_x$. We show that $R_i\cong P^{\lambda_i}$,
for $i\in\{1,\ldots,x\}$. To do so, we first show that

\begin{equation}\label{eqn R dec}
R_i\cong P^{\lambda_i}\oplus Q_i\,,
\end{equation}
where $Q_i$ is a direct sum of indecomposable projective $B_{0,0}$-modules whose labelling partitions are belong
to $\{\lambda_{i+1},\ldots,\lambda_x\}$.
So let $i\in\{1,\ldots,x\}$.
There is a $p$-regular partition $\mu\in\{\lambda_1,\ldots,\lambda_x\}$ with $P^\mu\mid R_i$ and
$(P^\mu:S^{\lambda_i})=1$. Thus $\mu\geq \lambda_i$, by \ref{noth Specht filtration}. Since $(P^\mu:S^\mu)=1$, we also have
$(R_i:S^\mu)\neq 0$, hence $\lambda_i\geq \mu$ and then $\mu=\lambda_i$. 
If $\rho\neq \lambda_i$ is 
a $p$-regular partition of $B_{0,0}$ with $P^\rho\mid R_i$, then also $\rho\in\{\lambda_1,\ldots,\lambda_x\}$. Since
$(P^\rho:S^\rho)=1$, also $(R_i:S^\rho)\neq 0$, so that we must have $\lambda_i>\rho$ and $\rho\in\{\lambda_{i+1},\ldots,\lambda_x\}$.

This proves (\ref{eqn R dec}). Now we show that $Q_i=\{0\}$, for $i\in\{1,\ldots,x\}$. To do so, we argue by reverse induction on $i$.
For $i=x$, the assertion if clearly true. So let $i<x$, and assume that $Q_i\neq \{0\}$. Then, by (\ref{eqn R dec}), we would have
$P^{\lambda_j}\mid Q_i\mid R_i$, for some $j>i$. But, by induction, $P^{\lambda_j}\cong R_j$.
 Since both $R_j$ and $R_i$ have
precisely four Specht factors, this would imply $R_i\cong R_j$, a contradiction.

Now the assertion of the theorem follows from \ref{noth induce from w=1} (c), (d).
\end{proof}

To summarize, by \ref{noth induce from w=0} and Theorem~\ref{thm induce from w=1},
we have completely determined the columns of the decomposition matrix of the block $B_{0,0}$.
From this information, it is now straightforward to read off the rows of the decomposition matrix of $B_{0,0}$, that is, the
decomposition numbers of $B_{0,0}$, as well.
In Corollary~\ref{cor dec matrix} below, we shall in fact write down the Loewy
structures of the Specht modules in $B_{0,0}$. Before doing so,  we mention one last bit of information
concerning the $\partial$-values and colours of the partitions of $B_{0,0}$, which is immediate from
\ref{noth colour weight 2}.

\begin{lemma}\label{lemma B0 partial}
Identifying every partition of $B_{0,0}$ with its $[2^p]$-abacus, the partitions of $B_{0,0}$
have the following $\partial$-values:

\smallskip

\begin{center}
\begin{tabular}{|c|c|c|}\hline
$\partial$& $p$-regular & $p$-singular\\\hline\hline
$0$ (black)& $\langle p\rangle$& $\langle 2,1\rangle$\\
            & $\langle i+1,i\rangle$, $i\in\{2,\ldots,p-1\}$ \text{ odd }& \\\hline
            $0$ (white) & $\langle i+1,i\rangle$, $i\in\{2,\ldots,p-1\}$ \text{ even }& $\langle 1,1\rangle$\\\hline
$1$& $\langle p-1\rangle$ & $\langle 2,2\rangle$\\
      & $\langle i+2,i\rangle$, $i\in\{1,\ldots,p-2\}$&\\\hline
$2$& $\langle p-2\rangle$ & $\langle 3,3\rangle$\\
      & $\langle i+3,i\rangle$, $i\in\{1,\ldots,p-3\}$&\\\hline
$d\in\{3,\ldots,p-2\}$&     $\langle p-d\rangle$ & $\langle d+1,d+1\rangle$\\
      & $\langle i+d+1,i\rangle$, $i\in\{1,\ldots,p-d-1\}$&\\\hline  
$p-1$ & $\langle 1\rangle$ & $\langle p,p\rangle$\\\hline    
\end{tabular}
\end{center}
\end{lemma}

\begin{cor}\label{cor dec matrix}
Identifying every partition of $B_{0,0}$ with its $[2^p]$-abacus, the Specht modules in $B_{0,0}$ have 
the following Loewy structures:

\begin{itemize}
\item[{\rm (a)}] For $i\in\{1,\ldots,p-1\}$, one has $S^{\langle i\rangle}\ \approx \ \begin{matrix} D^{\langle i\rangle}\\ D^{\langle i+1\rangle}\end{matrix}$. Moreover, $S^{\langle p\rangle}\cong D^{\langle p\rangle}$.

\item[{\rm(b)}] For $i\in\{3,\ldots,p-1\}$, one has $S^{\langle i,i\rangle}\ \approx \ \begin{matrix} D^{\langle i,1\rangle}\\ D^{\langle i+1,1\rangle}\end{matrix}$. Moreover,
$$S^{\langle p,p\rangle}\ \approx \  \begin{matrix}D^{\langle p,1\rangle}\\ D^{\langle 1\rangle}\end{matrix}\,,\quad 
S^{\langle 2,2\rangle}\ \approx \  \begin{matrix}D^{\langle 3,2\rangle}\\ D^{\langle 3,1\rangle}\end{matrix}\,,\quad
\text{ and} \quad S^{\langle 1,1\rangle}\cong D^{\langle 3,2\rangle}.$$

\item[{\rm (c)}] For $1\leq i<j\leq p$ with $j-i\geq 3$, one has
$$S^{\langle j,i\rangle}\ \approx \  \begin{cases}
\begin{matrix} D^{\langle j,i\rangle}\\ D^{\langle j,i+1\rangle}\oplus D^{\langle j+1,i\rangle}\\ D^{\langle j+1,i+1\rangle}\end{matrix} &\text{ if } j\neq p\,,\\
&\\
\begin{matrix}D^{\langle p,i\rangle}\\ D^{\langle p,i+1\rangle}\oplus D^{\langle i\rangle}\\ D^{\langle i+1\rangle}\end{matrix} &\text{ if } j= p\,.
\end{cases}$$

\item[{\rm (d)}] For $1\leq i<j\leq p$ with $j-i=2$, one has
$$S^{\langle j,i\rangle}\ \approx \  \begin{cases}
\begin{matrix}D^{\langle j,i\rangle}\\ D^{\langle j,i+1\rangle}\oplus D^{\langle j+1,i\rangle}\oplus D^{\langle j+1,i+1\rangle}\\ D^{\langle j+1,i+2\rangle}\end{matrix} &\text{ if } j\neq p\,,\\
&\\
\begin{matrix}D^{\langle p,p-2\rangle}\\ D^{\langle p,p-1\rangle}\oplus D^{\langle p\rangle}\oplus D^{\langle p-2\rangle} \\ D^{\langle p-1\rangle}\end{matrix} &\text{ if } j= p\,.
\end{cases}$$

\item[{\rm (e)}] For $1\leq i<j\leq p$ with $j-i=1$, one has
$$S^{\langle j,i\rangle}\ \approx \  \begin{cases}
\begin{matrix}D^{\langle j,i\rangle}\\ D^{\langle j+1,i\rangle}\\ D^{\langle j+2,i+2\rangle}\end{matrix} &\text{ if } i\notin\{1,p-2,p-1\}\,,\\
&\\
\begin{matrix} D^{\langle p-1,p-2\rangle}\\ D^{\langle p,p-2\rangle}\\ D^{\langle p\rangle}\end{matrix} &\text{ if } i=p-2\,,\\
&\\
\begin{matrix} D^{\langle p,p-1\rangle}\\ D^{\langle p-1\rangle}\end{matrix} &\text{ if } i=p-1\,,\\
&\\
\begin{matrix}D^{\langle 3,1\rangle}\\ D^{\langle 4,3\rangle}\end{matrix} &\text{ if } i=1\,.\\
\end{cases}$$
\end{itemize}
\end{cor}

\begin{proof}
By Theorem~\ref{thm induce from w=1} and \ref{noth induce from w=0}, we know the columns
of the decomposition matrix of $B_{0,0}$. Given these, it is straightforward to determine the rows of this decomposition
matrix, that is, the composition factors of the Specht modules in $B_{0,0}$.  In particular, we see that 
$S^{\langle p\rangle}\cong D^{\langle p\rangle}$ and $S^{\langle 1,1\rangle}\cong D^{\langle 3,2\rangle}$.

Next, by \cite[Proposition 6.2]{ChuangTan2001}, a Specht module $S^\lambda$ in $B_{0,0}$ has Loewy length at most 3, and has Loewy length 3 if and only if $\lambda$ is both $p$-regular and $p$-restricted. If so, $S^\lambda$ has head isomorphic to $D^\lambda$, by
\cite[Corollary 12.2]{James1978} and socle isomorphic to $D^{\lambda_+}$, by \ref{noth part}(b) and \ref{noth partial}(c). Since
$\lambda_+$ has the same $\partial$-value and, if $\partial(\lambda)=0$, the same colour as $\lambda$, we deduce
the Loewy structures of the Specht modules in (c) and (d) as well as those of $S^{\langle i+1,i\rangle}$ with
$i\in\{2,\ldots,p-2\}$ from Lemma~\ref{lemma B0 partial}.

If $\lambda$ is a $p$-regular partition of $B_{0,0}$ such that $S^\lambda$ has exactly two composition factors, then
$S^\lambda$ is of course uniserial with head isomorphic to $D^\lambda$. 

Hence, it remains to establish the Loewy structures of the Specht modules labelled by the $p$-singular
partitions $\langle i,i\rangle$ with $i\in\{2,\ldots, p\}$, and $\langle 2,1\rangle$, respectively.
By \ref{noth part}, we know that, whenever $\lambda$ is $p$-restricted, $S^\lambda$ has socle isomorphic to $D^{\mathbf{m}(\lambda')}$, and
$\mathbf{m}(\lambda')=\lambda_+$, by \ref{noth partial}(b). Since $\lambda_+$ has the same $\partial$ value and, if
$\partial(\lambda)=0$, the same colour as $\lambda$, Lemma~\ref{lemma B0 partial}
implies $\Soc(S^{\langle i,i\rangle})\cong D^{\langle i+1,1\rangle}$, for
$i\in\{2,\ldots,p-1\}$, $\Soc(S^{\langle p,p\rangle})\cong D^{\langle 1\rangle}$, and $\Soc(S^{\langle 2,1\rangle})\cong D^{\langle 4,3\rangle})$. 

This completes the proof of the corollary.
\end{proof}

\begin{rem}\label{rem peel}
We emphasize that the partitions treated in part (a) and (b) of Corollary~\ref{cor dec matrix} are precisely the hook partitions 
of $B_{0,0}$, so that these assertions also follow from Peel's results in \cite{Peel}. 
\end{rem}

We are now in the position to describe how the Ext-quiver of $B_{0,0}$ is encoded in the decomposition
matrix of $B_{0,0}$, using general information from \cite{Scopes1995, Richards1996, ChuangTan2001}. As already mentioned at the
beginning of this section, this quiver was first computed by Martin \cite{Martin1989}. Although Richards' work had not yet been
available at that time, Martin's way to draw the quiver is the same that we shall now describe.

\begin{noth}\label{noth quiver B0}
{\bf The Ext-quiver of $B_{0,0}$.}\, 
(a)\, We draw a quiver with $p$ rows, which we label by $0,1,\ldots,p-1$, from top to bottom. In row $i\in\{0,\ldots,p-1\}$ we draw a 
vertex for each $p$-regular partition of $B_{0,0}$ that has $\partial$-value $i$. We order the partitions in row $i\in\{1,\ldots,p-1\}$
with respect to the lexicographic ordering $>$, from left to right. In row $0$, the 
leftmost vertex corresponds to the partition $\langle p\rangle$, which is the largest black partition. Next we draw
$\langle p,p-1\rangle$, which is the largest white partition. From then on, black and white partitions alternate in decreasing
lexicographic ordering, from left to right.

\smallskip

(b)\, Let $\lambda$ be a $p$-regular partition of $B_{0,0}$ corresponding to a vertex in row $i$ of the quiver we have just drawn, that is, 
$\partial(\lambda)=i$. By \cite[Theorem~6.1]{ChuangTan2001}, $\lambda$ can only be connected to a vertex
$\mu$ that lies in row $i-1$ or in row $i+1$. Moreover, if $\mu$ is a vertex in row $i-1$ or $i+1$, then
$\mu$ and $\lambda$ are connected if and only if one of the following holds:

\quad (i)\, $\lambda>\mu$ and $[S^\mu:D^\lambda]=1$, or

\quad (ii)\, $\mu>\lambda$ and $[S^\lambda:D^\mu]=1$.

In case (i) we draw an arrow $\lambda\to \mu$, in case (ii) we draw an arrow $\mu\to \lambda$.
Thus, representing every partition of $B_{0,0}$ on a $[2^p]$-abacus and invoking Corollary~\ref{cor dec matrix} and Lemma~\ref{lemma B0 partial}, we obtain the following quiver, which
equals $Q_{0,0}(p)$ in Section~\ref{sec quiv}, with respect to the lexicographic ordering on partitions:

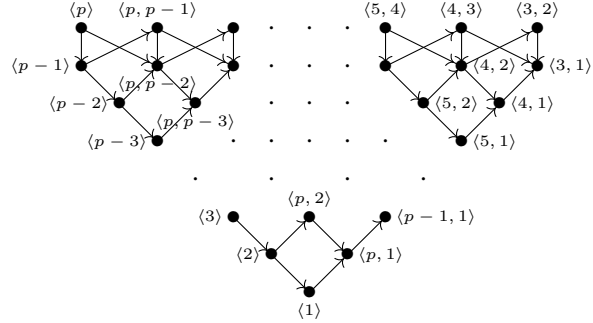
\begin{center}
\begin{tikzpicture}
\coordinate[label=above:{\tiny $\langle p\rangle$}] (0,0) at (0,0);
\coordinate[label=above:{\tiny $\langle p,p-1\rangle$}] (1,0) at (1,0);
\coordinate  (2,0) at (2,0);
\coordinate (2.5,0) at (2.5,0);
\coordinate (3,0) at (3,0);
\coordinate (3.5,0) at (3.5,0);
\coordinate[label=above:{\tiny $\langle 5,4\rangle$}] (4,0) at (4,0);
\coordinate[label=above:{\tiny $\langle 4,3\rangle$}] (5,0) at (5,0);
\coordinate[label=above:{\tiny $\langle 3,2\rangle$}] (6,0) at (6,0);

\coordinate[label=left:{\tiny $\langle p-1\rangle$}] (0,-0.5) at (0,-0.5);
\coordinate[label=below:{\tiny $\langle p,p-2\rangle$}] (1,0.5) at (1,-0.5);
\coordinate (2,-0.5) at (2,-0.5);
\coordinate (2.5,-0.5) at (2.5,-0.5);
\coordinate (3,-0.5) at (3,-0.5);
\coordinate (3.5,-0.5) at (3.5,-0.5);
\coordinate (4,-0.5) at (4,-0.5);
\coordinate[label=right:{\tiny $\langle 4,2\rangle$}] (5,-0.5) at (5,-0.5);
\coordinate[label=right:{\tiny $\langle 3,1\rangle$}] (6,-0.5) at (6,-0.5);

\coordinate[label=left:{\tiny $\langle p-2\rangle$}] (0.5,-1) at (0.5,-1);
\coordinate[label=below:{\tiny $\langle p,p-3\rangle$}] (1.5,-1) at (1.5,-1);
\coordinate (2,-1) at (2,-1);
\coordinate (2.5,-1) at (2.5,-1);
\coordinate (3,-1) at (3,-1);
\coordinate (3.5,-1) at (3.5,-1);
\coordinate (4,-1) at (4,-1);
\coordinate[label=right:{\tiny $\langle 5,2\rangle$}] (4.5,-1) at (4.5,-1);
\coordinate[label=right:{\tiny $\langle 4,1\rangle$}] (5.5,-1) at (5.5,-1);

\coordinate[label=left:{\tiny $\langle p-3\rangle$}] (1,-1.5) at (1,-1.5);
\coordinate (2,-1.5) at (2,-1.5);
\coordinate (2.5,-1.5) at (2.5,-1.5);
\coordinate (3,-1.5) at (3,-1.5);
\coordinate (3.5,-1.5) at (3.5,-1.5);
\coordinate (4,-1.5) at (4,-1.5);
\coordinate[label=right:{\tiny $\langle 5,1\rangle$}] (5,-1.5) at (5,-1.5);

\coordinate (1.5,-2) at (1.5,-2);
\coordinate (2.5,-2) at (2.5,-2);
\coordinate (3,-2) at (3,-2);
\coordinate (3.5,-2) at (3.5,-2);
\coordinate (4.5,-2) at (4.5,-2);

\coordinate[label=left:{\tiny $\langle 3\rangle$}] (2,-2.5) at (2,-2.5);
\coordinate[label=above:{\tiny $\langle p,2\rangle$}] (3,-2.5) at (3,-2.5);
\coordinate[label=right:{\tiny $\langle p-1,1\rangle$}] (4,-2.5) at (4,-2.5);

\coordinate[label=left:{\tiny $\langle 2\rangle$}]  (2.5,-3) at (2.5,-3);
\coordinate[label=right:{\tiny $\langle p,1\rangle$}]  (3.5,-3) at (3.5,-3);

\coordinate[label=below:{\tiny $\langle 1\rangle$}] (3,-3.5) at (3,-3.5);

\draw (0,0) node{$\bullet$};
\draw (1,0) node{$\bullet$};
\draw (2,0) node{$\bullet$};
\draw (2.5,0) node{$\cdot$};
\draw (3,0) node{$\cdot$};
\draw (3.5,0) node{$\cdot$};
\draw (4,0) node{$\bullet$};
\draw (5,0) node{$\bullet$};
\draw (6,0) node{$\bullet$};

\draw (0,-0.5) node{$\bullet$};
\draw (1,-0.5) node{$\bullet$};
\draw (2,-0.5) node{$\bullet$};
\draw (2.5,-0.5) node{$\cdot$};
\draw (3,-0.5) node{$\cdot$};
\draw (3.5,-0.5) node{$\cdot$};
\draw (4,-0.5) node{$\bullet$};
\draw (5,-0.5) node{$\bullet$};
\draw (6,-0.5) node{$\bullet$};

\draw[con] (0,0) -- (0,-0.5);
\draw[con] (0,0) -- (1,-0.5);
\draw[con] (0,-0.5) -- (1,0);
\draw[con] (1,0) -- (1,-0.5);
\draw[con] (1,-0.5) -- (2,0);
\draw[con] (1,0) -- (2,-0.5);
\draw[con] (2,0) -- (2,-0.5);
\draw[con] (4,0) -- (4,-0.5);
\draw[con] (4,0) -- (5,-0.5);
\draw[con] (4,-0.5) -- (5,0);
\draw[con] (5,0) -- (5,-0.5);
\draw[con] (5,-0.5) -- (6,0);
\draw[con] (6,0) -- (6,-0.5);
\draw[con] (5,0) -- (6,-0.5);

\draw (0.5,-1) node{$\bullet$};
\draw (1.5,-1) node{$\bullet$};
\draw (2.5,-1) node{$\cdot$};
\draw (3,-1) node{$\cdot$};
\draw (3.5,-1) node{$\cdot$};
\draw (4.5,-1) node{$\bullet$};
\draw (5.5,-1) node{$\bullet$};

\draw[con] (0,-0.5) -- (0.5,-1);
\draw[con] (0.5,-1) -- (1,-0.5);
\draw[con] (1,-0.5) -- (1.5,-1);
\draw[con] (1.5,-1) -- (2,-0.5);

\draw[con] (4,-0.5) -- (4.5,-1);
\draw[con] (4.5,-1) -- (5,-0.5);
\draw[con] (5,-0.5) -- (5.5,-1);
\draw[con] (5.5,-1) -- (6,-0.5);

\draw (1,-1.5) node{$\bullet$};
\draw (2,-1.5) node{$\cdot$};
\draw (2.5,-1.5) node{$\cdot$};
\draw (3,-1.5) node{$\cdot$};
\draw (3.5,-1.5) node{$\cdot$};
\draw (4,-1.5) node{$\cdot$};
\draw (5,-1.5) node{$\bullet$};

\draw[con] (0.5,-1) -- (1,-1.5);
\draw[con] (1,-1.5) -- (1.5,-1);
\draw[con] (4.5,-1) -- (5,-1.5);
\draw[con] (5,-1.5) -- (5.5,-1);

\draw (1.5,-2) node{$\cdot$};
\draw (2.5,-2) node{$\cdot$};
\draw (3.5,-2) node{$\cdot$};
\draw (4.5,-2) node{$\cdot$};

\draw (2,-2.5) node{$\bullet$};
\draw (3,-2.5) node{$\bullet$};
\draw (4,-2.5) node{$\bullet$};

\draw (2.5,-3) node{$\bullet$};
\draw (3.5,-3) node{$\bullet$};

\draw (3,-3.5) node{$\bullet$};

\draw[con] (2.5,-3) -- (3,-3.5);
\draw[con] (3,-3.5) -- (3.5,-3);

\draw[con] (2,-2.5) -- (2.5,-3);
\draw[con] (2.5,-3) -- (3,-2.5);
\draw[con] (3,-2.5) -- (3.5,-3);
\draw[con] (3.5,-3) -- (4,-2.5);
\end{tikzpicture}
\end{center}

\end{noth}

\begin{thm}\label{thm quiver B0}
Let $p\geq 5$. The principal block $B_{0,0}$ of $F\mathfrak{S}_{2p}$ has Ext-quiver $Q_{0,0}(p)$ shown in (\ref{eqn (36)}), and
the principal block $B_{1,0}$ of $F\mathfrak{S}_{2p+1}$ has Ext-quiver $Q_{1,0}(p)$ shown in (\ref{eqn (37)}), with respect to the
lexicographic ordering on partitions.
\end{thm}

\begin{proof}
The assertion concerning $B_{0,0}$ has just been proved in \ref{noth quiver B0}.

We now consider the block $B_{1,0}$, which forms a $(2:1)$-pair with $B_{0,0}$, by Proposition~\ref{prop (2:1) pairs hooks}.
We represent all partitions of $B_{1,0}$ on a $[2,3,2^{p-2}]$-abacus, and all partitions of
$B_{0,0}$ on a $[3,2^{p-1}]$-abacus.
As usual, we denote the exceptional partitions of $B_{0,0}$ and $B_{1,0}$ by $\bar{\alpha}$, $\bar{\beta}$, $\bar{\gamma}$ and
$\alpha$, $\beta$, $\gamma$, respectively,
By Corollary~\ref{cor l and r hooks} and Lemma~\ref{lemma alpha max}, 
we know that $\bar{\alpha}$ is the largest partition of $2p$ in $B_{0,0}$ with $\partial$-value $p-1$.
So, by \ref{noth partial Scopes}, it suffices to consider the last three rows
of $Q_{0,0}$, whose vertices have $\partial$-values $p-1,p-2$ and $p-3$. This gives

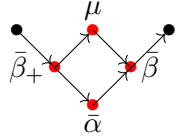
\begin{center}
\begin{tikzpicture}
\coordinate[label=left:$\bar{\beta}_+$] (B+) at (0.5,-0.5);
\coordinate[label=above:$\mu$] (M) at (1,0);
\coordinate[label=right:$\bar{\beta}$] (B) at (1.5,-0.5);
\coordinate[label=below:$\bar{\alpha}$] (A) at (1,-1);
\draw (B+) node{\color{red} $\bullet$};
\draw[con] (B+) -- (M) node{\color{red} $\bullet$};
\draw (B) node{\color{red} $\bullet$};
\draw[con] (M)  -- (B);
\draw[con] (B+) -- (A) node{\color{red} $\bullet$};
\draw[con] (A)  -- (B);
\draw[con] (0,0) node{$\bullet$} -- (B+);
\draw[con] (B) -- (2,0) node{$\bullet$};
\end{tikzpicture}
\end{center}

We need to identify the red vertices. By Corollary~\ref{cor l and r hooks} (in case (3)), $\bar{\gamma}$ is $p$-singular, while 
$\bar{\beta}$ is both $p$-regular and $p$-restricted. So $\bar{\beta}_+$ exists. Moreover, $\partial(\bar{\beta})=\partial(\bar{\beta}_+)=p-2$.
By Theorem~\ref{thm Loewy except}, $\bar{\beta}$ and
$\bar{\beta}_+$ are connected to $\bar{\alpha}$. This identifies $\bar{\beta}$ and $\bar{\beta}_+$.
Since $\mu>\bar{\beta}$, we must have $[S^{\bar{\beta}}:D^\mu]\neq 0$, by \cite[Theorem 6.1]{ChuangTan2001},
thus $D^\mu\cong \bar{Y}$, in the notation of Theorem~\ref{thm Loewy except}.

Now, by \ref{noth partial Scopes}, Corollary~\ref{cor l and r hooks} and Theorem~\ref{thm Loewy except} again, we 
obtain the following information on the rows of the quiver of $B_{1,0}$ whose vertices have $\partial$-values
$p-2$ and $p-3$:

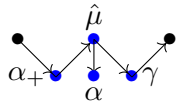
\begin{center}
\begin{tikzpicture}
\coordinate[label=left:$\alpha_+$] (A+) at (0.5,-0.5);
\coordinate[label=above:$\hat{\mu}$] (M) at (1,0);
\coordinate[label=right:$\gamma$] (G) at (1.5,-0.5);
\coordinate[label=below:$\alpha$] (A) at (1,-0.5);
\draw (A+) node{\color{blue} $\bullet$};
\draw[con] (A+) -- (M) node{\color{blue} $\bullet$};
\draw (G) node{\color{blue} $\bullet$};
\draw[con] (M)  -- (G);
\draw[con] (M) -- (A) node{\color{blue} $\bullet$};
\draw[con] (0,0) node{$\bullet$} -- (A+);
\draw[con] (G) -- (2,0) node{$\bullet$};
\end{tikzpicture}
\end{center}

Here $D^{\hat{\mu}}\cong Y$. By Corollary~\ref{cor l and r hooks}, $\beta$ is $p$-singular. Hence, by Lemma~\ref{lemma Ext alpha} and
Theorem~\ref{thm Loewy except}, $\alpha$ is only connected to $D^{\hat{\mu}}$. Since
$[S^\alpha:D^{\hat{\mu}}]\neq 0$, we have $\hat{\mu}>\alpha$. This proves, that $B_{1,0}$ has quiver $Q_{1,0}(p)$.
\end{proof}

%%%%%%%%%%%%%%%%%%%%%%%%%%%%%%%%%%%%%%%%%%%%%%%%%%%%%%%%%%%%%%%%%%%%

\section{The Quivers}\label{sec quiv}

Suppose that $n\in\NN$ with $n\geq 5$. 
First we construct a quiver $Q_{0,0}(n)$ with $n$ rows, and $n-1$ vertices in the top row that has the following shape:

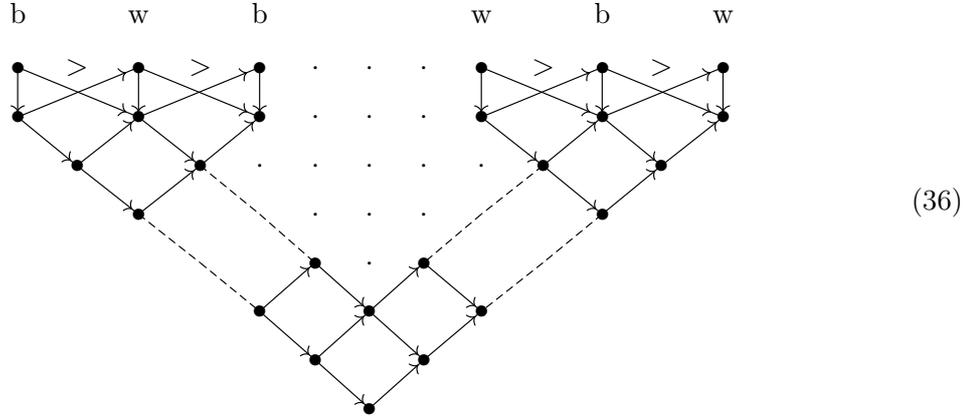
\begin{equation}\label{eqn (36)}
\begin{tikzcd}[row sep = 0.5em, column sep = 0.5em, ampersand replacement=\&] 
\text{b}\&\phantom{x}\&\text{w}\&\phantom{x}\&\text{b}\& \&  \& \&\text{w}\&\phantom{x}\&\text{b}\&\phantom{x}\&\text{w}\\
\bullet\ar[con]{d}\ar[con]{drr}\&>\&\bullet\ar[con]{d}\ar[con]{drr}\&>\&\bullet\ar[con]{d}\& \cdot \& \cdot \& \cdot \&\bullet\ar[con]{d}\ar[con]{drr}\&>\&\bullet\ar[con]{d}\ar[con]{drr}\&>\&\bullet\ar[con]{d}\\
\bullet      \ar[con]{dr}        \ar[con]{urr}   \&\phantom{x}\&\bullet     \ar[con]{dr}   \ar[con]{urr}       \&\phantom{x}\&\bullet        \& \cdot \& \cdot \& \cdot \&\bullet       \ar[con]{dr}     \ar[con]{urr}       \&\phantom{x}\&\bullet    \ar[con]{dr}                 \ar[con]{urr}     \&\phantom{x}\&\bullet\\
\&\bullet\ar[con]{dr}\ar[con]{ur}\&\&\bullet\ar[dashed]{dr}\ar[con]{ur}\&\cdot\&\cdot\&\cdot\&\cdot\&\cdot\&\bullet\ar[con]{dr}\ar[dashed]{dl}\ar[con]{ur}\&\&\bullet\ar[con]{ur}\&\\
\&\&\bullet\ar[dashed]{dr}\ar[con]{ur}\&\& \phantom{x}\ar[dashed]{dr}\&\cdot\&\cdot\& \cdot\&\phantom{x}\ar[dashed]{dl}\&\&\bullet\ar[dashed]{dl}\ar[con]{ur}\&\&\\
\&\&\&\phantom{x}\ar[dashed]{dr}\&\&\bullet\ar[con]{dr}\&\cdot\&\bullet\ar[con]{dr}\&\&\phantom{x}\ar[dashed]{dl}\&\&\&\\
\&\&\&\&  \bullet\ar[con]{dr}\ar[con]{ur}\& \&\bullet\ar[con]{dr}\ar[con]{ur}\&\&\bullet           \&\&\&\&\\
\&\&\&\&\&\bullet\ar[con]{dr}\ar[con]{ur}\&\&\bullet\ar[con]{ur} \&\&\&\&\&\\
\&\&\&\&\&\&\bullet\ar[con]{ur}\&\&\&\&\&\&
\end{tikzcd}
\end{equation}

The rows are labelled by $0,1,\ldots,n-1$, from top to bottom.
The vertices in the top row are equipped with a colour, which is either white (w) or black (b). Moreover,  we assume that there is 
a total ordering $>$ on the vertices of $Q_{0,0}(n)$ and that, whenever there is an arrow from vertex $x$ to vertex $y$, we have $x>y$.
In addition, the ordering on the vertices in a given row decreases from left to right. Unfortunately, the way our quivers are drawn,
one can, in general, not read off how to compare a vertex in the top row with a vertex in one of the lower rows. For the applications to
the main results of this paper this is irrelevant.

Next we modify the quiver $Q_{0,0}(n)$ to define $Q_{1,0}(n)$ as the following quiver with $n-1$ rows, and $n-1$ vertices in the top row. We call the part of $Q_{1,0}(n)$
consisting of the red and green vertices the \textit{left rim segment} of $Q_{1,0}(n)$, and those part
consisting of the blue vertices the  \textit{right rim segment} of $Q_{1,0}(n)$. Again we have a total ordering on the vertices
of $Q_{1,0}(n)$.

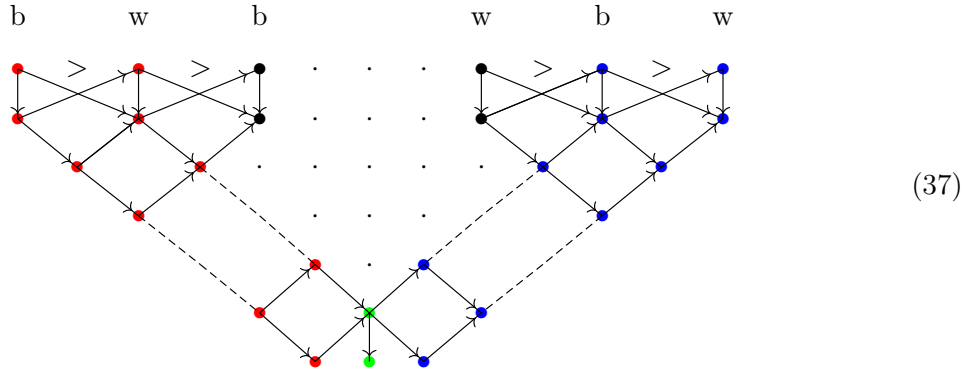
\begin{equation}\label{eqn (37)}
\begin{tikzcd}[row sep = 0.5em, column sep = 0.5em, ampersand replacement=\&] 
\text{b}\&\phantom{x}\&\text{w}\&\phantom{x}\&\text{b}\& \&  \& \&\text{w}\&\phantom{x}\&\text{b}\&\phantom{x}\&\text{w}\\
\color{red}{\bullet}\ar[con]{d}\ar[con]{drr}\&>\&\color{red}{\bullet}\ar[con]{d}\ar[con]{drr}\&>\&\bullet\ar[con]{d}\& \cdot \& \cdot \& \cdot \&\bullet\ar[con]{d}\ar[con]{drr}\&>\&\color{blue}{\bullet}\ar[con]{d}\ar[con]{drr}\ar{dll}\&>\&\color{blue}{\bullet}\ar[con]{d}\\
\color{red}{\bullet}      \ar[con]{dr}      \ar[con]{urr}      \&\phantom{x}\&\color{red}{\bullet}     \ar{dl}\ar[con]{dr}    \ar[con]{urr}      \&\phantom{x}\&\bullet       \& \cdot \& \cdot \& \cdot \&\bullet       \ar[con]{dr}     \ar[con]{urr}       \&\phantom{x}\&\color{blue}{\bullet}    \ar[con]{dr}  \ar[con]{urr}                    \&\phantom{x}\&\color{blue}{\bullet}\\
\&\color{red}{\bullet}\ar[con]{dr}\ar[con]{ur}\&\&\color{red}{\bullet}\ar[con]{ur}\ar[dashed]{dr}\&\cdot\&\cdot\&\cdot\&\cdot\&\cdot\&\color{blue}{\bullet}\ar[con]{dr}\ar[dashed]{dl}\ar[con]{ur}\&\&\color{blue}{\bullet}\ar[con]{ur}\&\\
\&\&\color{red}{\bullet}\ar[dashed]{dr}\ar[con]{ur}\&\& \phantom{x}\ar[dashed]{dr}\&\cdot\&\cdot\& \cdot\&\phantom{x}\ar[dashed]{dl}\&\&\color{blue}{\bullet}\ar[dashed]{dl}\ar[con]{ur}\&\&\\
\&\&\&\phantom{x}\ar[dashed]{dr}\&\&\color{red}{\bullet}\ar[con]{dr}\&\cdot\&\color{blue}{\bullet}\ar[con]{dr}\&\&\phantom{x}\ar[dashed]{dl}\&\&\&\\
\&\&\&\& \color{red}{\bullet}\ar[con]{dr}\ar[con]{ur}\& \&\color{green}{\bullet}\ar[con]{dr}\ar[con]{d}\ar[con]{ur}\&\&\color{blue}{\bullet}           \&\&\&\&\\
\&\&\&\&\&\color{red}{\bullet}\ar[con]{ur}\&\color{green}{\bullet}\&\color{blue}{\bullet}\ar[con]{ur} \&\&\&\&\&\\
\end{tikzcd}
\end{equation}

\medskip

Now, for $(i,j)\in\{1,\ldots,n-1\}\times\{0,\ldots,n-2\}$ with $(i,j)\neq (1,0)$, 
we define the quiver $Q_{i,j}(n)$ that is obtained by modifying the 
left and right rim segments of $Q_{1,0}(n)$ in the following way:

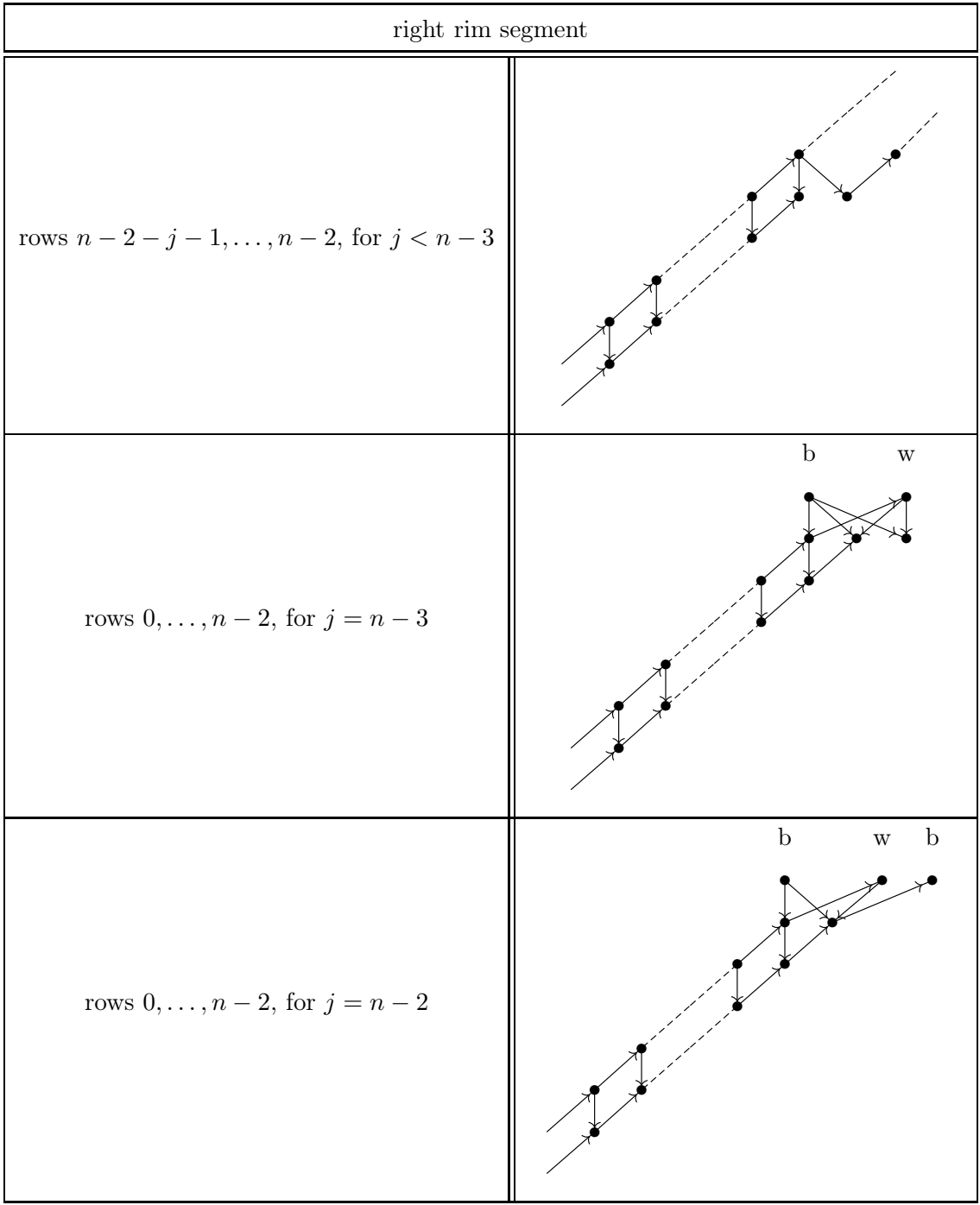
\begin{tabular}{|c||c|}\hline
\multicolumn{2}{|c|}{left rim segment}\\\hline\hline
rows 
$n-2-i,\ldots,n-2$, for $i\leq n-3$&
$$\begin{tikzcd}[row sep = 0.5em, column sep = 0.5em, ampersand replacement=\&] 
\&\phantom{x}\ar[dashed]{dr}\&\&    \& \&\&\&\&\\  
\phantom{x}\ar[dashed]{dr}\&\&\phantom{x} \ar[dashed]{dr}\&    \& \&\&\&\&\\  
\&\bullet\ar[con]{dr}\& \&    \bullet\ar[con]{d}\ar[con]{dr} \& \&\&\&\&\\  
\&\&\bullet\ar[con]{ur}\&\bullet\ar[con]{dr}\&\bullet \ar[dashed]{dr}\&\&\&\&\\
\&\&\&\&\bullet\ar[dashed]{dr}\ar[con]{u}\&\phantom{x}\ar[dashed]{dr}\&\&\&\\
\&\&\&\&\&\phantom{x}\ar[dashed]{dr}\&\bullet\ar[con]{dr}\&\&\\
\&\&\&\&\&\&\bullet\ar[con]{dr}\ar[con]{u}\&\bullet\ar[con]{dr}\&\\
\&\&\&\&\&\&\&\bullet\ar[con]{dr}\ar[con]{u}\&\bullet\\
\&\&\&\&\&\&\&\&\bullet\ar[con]{u}\\
\end{tikzcd}$$\\\hline
rows 
$0,\ldots,n-2$, for $i=n-2$&
$$\begin{tikzcd}[row sep = 0.5em, column sep = 0.5em, ampersand replacement=\&] 
\text{b}\&   \&\text{w}\&\&\&\&\&\\
\bullet\ar[con]{d}\ar[con]{dr}\ar[con]{drr}\&   \&\bullet \ar[con]{d}\ar[con]{dl}\&\&\&\&\&\\
\bullet\ar[con]{urr}\& \bullet \ar[con]{dr}\&    \bullet\ar[con]{dr} \& \&\&\&\&\\  
\&\&\bullet\ar[con]{dr}\ar[con]{u}\&\bullet \ar[dashed]{dr}\&\&\&\&\\
\&\&\&\bullet\ar[dashed]{dr}\ar[con]{u}\&\phantom{x}\ar[dashed]{dr}\&\&\&\\
\&\&\&\&\phantom{x}\ar[dashed]{dr}\&\bullet\ar[con]{dr}\&\&\\
\&\&\&\&\&\bullet\ar[con]{dr}\ar[con]{u}\&\bullet\ar[con]{dr}\&\\
\&\&\&\&\&\&\bullet\ar[con]{dr}\ar[con]{u}\&\bullet\\
\&\&\&\&\&\&\&\bullet\ar[con]{u}\\
\end{tikzcd}$$\\\hline
rows 
$0,\ldots,n-2$, for $i=n-1$&
$$\begin{tikzcd}[row sep = 0.5em, column sep = 0.5em, ampersand replacement=\&] 
\&\text{b}\&\text{w}  \&\text{w}\&\&\&\&\&\\
\&\bullet\ar[con]{dr}\ar[con]{drr}\& \bullet\ar[con]{d} \&\bullet \ar[con]{d}\&\&\&\&\&\\
\&\& \bullet\ar[con]{ur} \ar[con]{dr} \&    \bullet\ar[con]{dr} \& \&\&\&\&\\  
\&\&\&\bullet\ar[con]{dr}\ar[con]{u}\&\bullet \ar[dashed]{dr}\&\&\&\&\\
\&\&\&\&\bullet\ar[con]{u}\ar[dashed]{dr}\&\phantom{x}\ar[dashed]{dr}\&\&\&\\
\&\&\&\&\&\phantom{x}\ar[dashed]{dr}\&\bullet\ar[con]{dr}\&\&\\
\&\&\&\&\&\&\bullet\ar[con]{dr}\ar[con]{u}\&\bullet\ar[con]{dr}\&\\
\&\&\&\&\&\&\&\bullet\ar[con]{dr}\ar[con]{u}\&\bullet\\
\&\&\&\&\&\&\&\&\bullet\ar[con]{u}\\
\end{tikzcd}$$\\\hline

\end{tabular}

\medskip

\begin{tabular}{|c||c|}\hline
\multicolumn{2}{|c|}{right rim segment}\\\hline\hline
rows 
$n-2-j-1,\ldots,n-2$, for $j<n-3$ &
$$\begin{tikzcd}[row sep = 0.5em, column sep = 0.5em, ampersand replacement=\&] 
\&\&\&\&\&\&\& \phantom{x}\ar[dashed]{dl}\&\\
\&\&\&\& \&            \&     \phantom{x} \ar[dashed]{dl} \&                \&  \ar[dashed]{dl}\\
\&\&\&\&\&\bullet  \ar[con]{dr} \ar[con]{d} \&                \&   \bullet   \&\\
\&\&\&\&\bullet \ar[con]{ur}  \ar[con]{d} \ar[dashed]{dl} \&     \bullet  \&  \bullet  \ar[con]{ur}  \&\&\\
\&\&\& \phantom{x} \ar[dashed]{dl}\& \bullet \ar[dashed]{dl}\ar[con]{ur}\&\&\&\&\\
\&\&\bullet \ar[con]{d}\& \phantom{x} \ar[dashed]{dl}\&\&\&\&\&\\
\&\bullet\ar[con]{d}\ar[con]{ur}\&\bullet \&\&\&\&\&\&\\
\phantom{x}\ar[con]{ur}\&\bullet\ar[con]{ur}\&\&\&\&\&\&\&\\
\phantom{x}\ar[con]{ur}\&\&\&\&\&\&\&\&\\
\end{tikzcd}$$\\\hline
rows 
$0,\ldots,n-2$, for $j=n-3$&
$$\begin{tikzcd}[row sep = 0.5em, column sep = 0.5em, ampersand replacement=\&] 
\&\&\&\&\&\text{b} \&             \&   \text{w}  \&\\
\&\&\&\&\&\bullet \ar[con]{d}  \ar[con]{dr} \ar[con]{drr}\&             \&   \bullet \ar[con]{d}\ar[con]{dl}  \&\\
\&\&\&\&\&\bullet \ar[con]{urr} \ar[con]{d} \&         \bullet      \&   \bullet    \&\\
\&\&\&\&\bullet   \ar[con]{d} \ar[dashed]{dl} \ar[con]{ur}\&     \bullet \ar[con]{ur} \&    \&\&\\
\&\&\& \phantom{x} \ar[dashed]{dl}\& \bullet \ar[dashed]{dl}\ar[con]{ur}\&\&\&\&\\
\&\&\bullet \ar[con]{d}\& \phantom{x} \ar[dashed]{dl}\&\&\&\&\&\\
\&\bullet\ar[con]{ur}\ar[con]{d}\&\bullet \&\&\&\&\&\&\\
\phantom{x}\ar[con]{ur}\&\bullet\ar[con]{ur}\&\&\&\&\&\&\&\\
\phantom{x}\ar[con]{ur}\&\&\&\&\&\&\&\&\\
\end{tikzcd}$$\\\hline
rows 
$0,\ldots,n-2$, for $j=n-2$&
$$\begin{tikzcd}[row sep = 0.5em, column sep = 0.5em, ampersand replacement=\&] 
\&\&\&\&\&\text{b} \&       \&   \text{w}\&\text{b}\\
\&\&\&\&\&\bullet\ar[con]{d}\ar[con]{dr} \&       \&   \bullet\ar[con]{dl} \&\bullet\\
\&\&\&\&\&\bullet \ar[con]{urr} \ar[con]{d} \&         \bullet  \ar[con]{urr}   \&      \&\&\\
\&\&\&\&\bullet   \ar[con]{d} \ar[dashed]{dl}\ar[con]{ur} \&     \bullet \ar[con]{ur}  \&  \&  \&\&\\
\&\&\& \phantom{x} \ar[dashed]{dl}\& \bullet \ar[dashed]{dl}\ar[con]{ur}\&\&\&\&\&\\
\&\&\bullet \ar[con]{d}\& \phantom{x} \ar[dashed]{dl}\&\&\&\&\&\&\\
\&\bullet\ar[con]{d}\ar[con]{ur}\&\bullet \&\&\&\&\&\&\&\\
\phantom{x}\ar[con]{ur}\&\bullet\ar[con]{ur}\&\&\&\&\&\&\&\&\\
\phantom{x}\ar[con]{ur}\&\&\&\&\&\&\&\&\&\\
\end{tikzcd}$$\\\hline
\end{tabular}

\bigskip

By Theorem~\ref{thm main1} and Theorem~\ref{thm main1 details}, the above quivers appear as Ext-quivers of
blocks of symmetric groups of weight 2 whose cores are hooks partitions. In Proposition~\ref{prop main2} and Theorem~\ref{thm main2}, we characterize possible
Morita equivalences between different such blocks. Since Morita equivalent blocks have isomorphic Ext-quivers, 
the following result is our key ingredient in the proof of Proposition~\ref{thm main2}.

\pagebreak

\begin{prop}\label{prop graph isos}
Suppose that $(i,j),(i',j')\in \{(0,0)\}\cup (\{1,\ldots,n-1\}\times\{0,\ldots,n-2\})$.
Then $Q_{i,j}(n)$ is isomorphic to $Q_{i,'j'}(n)$, as an undirected graph, if and only if one of the following cases occurs:

\begin{itemize}
\item[{\rm (i)}] $(i,j)=(i',j')$;
\item[{\rm (ii)}] $i=j'+1$ and  $j=i'-1$.
\end{itemize}
\end{prop}

\begin{proof} 
The number of vertices with a given valency is invariant under a graph isomorphism. Moreover, 
we  call a pair $v\neq w$ of vertices of any of the above graphs $Q_{ij}(n)$ {\it exceptional} if $v$ and $w$ have valency at most $2$, and in addition have distance $2$ in $Q_{ij}(n)$. A  graph isomorphism permutes exceptional pairs.
Next we define the {\it boundary}  of $Q_{i,j}(n)$
to be the full subgraph of $Q_{i,j}(n)$ whose vertices are precisely the vertices with valency at most $3$. The boundary is also invariant under graph isomorphisms.

\medskip

We fix $n$ and  write $Q_{i,j}$ instead of $Q_{i,j}(n)$.
We note that $Q_{0,0}$ and $Q_{1,0}$ cannot be isomorphic to any of the other graphs, since they are the only graphs with no 
exceptional pair, or with two exceptional pairs sharing a vertex, respectively.

Note that in case (ii) the (undirected) graphs $Q_{i,j}$ and $Q_{i,'j'}$ are clearly isomorphic,
since $Q_{i,j}$ is then simply obtained by reflecting $Q_{i',j'}$ along the middle axis.
Hence, we may from now on suppose that $2\leq i+j$, and show that $Q_{i,j}$ is not isomorphic
to any graph $Q_{i',j'}$ with $(i,j)\neq (i',j')\neq (i-1,j+1)$. 

\smallskip

We first consider the case that $Q_{i,j}$ has a vertex of valency 1. Then $i=n-1$ or $j=n-2$. 
Since  $Q_{n-1,n-2}$ is the only graph with two vertices of valency 1, it is not isomorphic to any other graph under consideration. Therefore we may suppose that $(i,j)\neq (n-1,n-2)$.
Then $i=n-1$ and $j\in\{0,\ldots,n-3\}$, or $j=n-2$ and $i\in\{1,\ldots,n-2\}$. For $j\in\{0,\ldots,n-3\}$, the graph $Q_{n-1,j}$ is isomorphic to
$Q_{j+1,n-2}$, as mentioned above. 

Thus it suffices to show that if $Q_{n-1,j}$ is  isomorphic to $Q_{n-1, m}$, for
$j,m\in\{0,\ldots,n-3\}$, then $j=m$.
Based on the diagrams,  we describe the graph structure of the boundary of  $Q_{n-1, j}$. 
It has a unique exceptional pair, which  is drawn in the right rim segment  when $j\geq 1$, and when  $j=0$, one vertex is in the left rim segment and the
other is the lowest vertex of the right
rim segment.
The vertices of the exceptional pair belong to two disjoint connected components of the boundary.  
When $j=0$,  one of these components is a line with $n-4$ vertices, and one component is
isomorphic to the Dynkin diagram $D_{n}$. For $1\leq j< n-3$, one component is isomorphic to $D_{n-j}$. The other component is a tree 
 branching at the vertex drawn in the lowest row
in the left rim segment. One arm of this tree has $n-4$ segments, one has just one segment, and the third has $j$ segments.
All other components of the boundary are isolated vertices. If $j=n-3$, then the boundary of $Q_{n-1, j}$ is one connected component.

The boundary of $Q_{n-1, m}$ has a similar structure. 
A graph isomorphism restricts to a graph isomorphism of the boundaries, so that $j=m$.

\smallskip

The graph $Q_{n-2,n-3}$ is the only graph with precisely two vertices of valency $2$, hence cannot be isomorphic to any of the others.

\smallskip
Now we consider the case that
$i=n-2$ and $j\in\{0,\ldots,n-4\}$, or  $j=n-3$ and $i\in\{1,\ldots,n-3\}$. This covers those cases where $Q_{i,j}$ has precisely four vertices of valency $2$.

Let $j\in\{0,\ldots,n-4\}$. Since $Q_{n-2, j}$ is isomorphic to $Q_{j+1, n-3}$, it suffices to show if $Q_{n-2, j}$ is  isomorphic to $Q_{n-2, m}$, for $m\in\{0,\ldots,n-4\}$, then $j=m$.
The graph $Q_{n-2, j}$ has  a unique exceptional pair, which  is drawn in the right rim segment  when $j\geq 1$, and when  
$j=0$, one vertex is in the left rim segment and the
other is the lowest vertex of the right
rim segment.
The vertices of the exceptional pair belong to two disjoint connected components of the boundary.  
When $j=0$,  these components are a line with $n-1$ segments, and one component
isomorphic to the (Dynkin) diagram $D_{n}$. 
For $j\geq 1$, one component is isomorphic to $D_{n-j}$, and the other is a tree 
 branching at the vertex drawn in the lowest row
in the left rim segment. This tree has one arm with $n-1$ segments, one arm with just one segment and one arm with $j$ segments.
All other components of the boundary are isolated vertices. 
The boundary of $Q_{n-2, m}$ has a similar structure. 
A graph isomorphism restricts to a graph isomorphism of the boundaries, so that $j=m$.

\smallskip
Now we consider  the graphs  $Q_{i,j}$ with $i\in\{1,\ldots, n-3\}$ and $j\in\{0,\ldots,n-4\}$. These have precisely six vertices of valency 2.
Since $Q_{ij}$ is isomorphic to $Q_{i-1,j+1}$, for  $i\in \{1,\ldots, n-3\}$ and $j\in \{0,\ldots, n-4\}$,
it suffices to show
that if $Q_{i,j}$ is  isomorphic to $Q_{i,m}$, for  $i\in \{1,\ldots, n-3\}$ and $j,m\in \{0,\ldots, n-4\}$, then $j = m$.

Suppose first that $i>1$. The graph $Q_{i,j}$ has  precisely two exceptional pairs. One is of then is drawn in the left rim segment and the other is drawn in the right rim segment, for $j\geq 1$.
For 
$j=0$, one vertex is in the left rim segment and the
other is the lowest vertex of the right
rim segment.
The vertices of the exceptional pair belong to {\it three}  disjoint connected components of the boundary.  
When $j=0$, they are a line with $i-1$ segments, then a component isomorphic to $D_{n-(i-1)}$ and a component isomorphic to $D_n$.

When $j\geq 1$, the boundary of $Q_{i,j}$  has a component that is a   tree branching at the vertex drawn in the lowest row in the left rim segment, whose arms have
$i-1$ segments,  one segment
and $j$ segments respectively. The other components of the boundary containing vertices from exceptional pairs are  
 isomorphic to $D_{n-(i-1)}$, and to $D_n$.
The boundary of $Q_{i,m}$ has a similar structure.  As before, if there is a graph isomorphism between $Q_{i,j}$ and
$Q_{i,m}$, then it restricts to an isomorphism of the boundaries, and we deduce
$j=m$.

It remains to consider the graphs $Q_{1, j}$, with $j\in\{1,\ldots,n-4\}$. Each of these has two exceptional pairs: one of them is drawn between the lowest two vertices 
in the left rim segment, and the other in the right rim segment. They belong to three disjoint connected components of the boundary: one is  isomorphic to $D_n$, 
one is a line with $j$ segments, and the third is a component isomorphic to $D_{n-j}$. Hence if $m,j\in\{1,\ldots,n-4\}$ and if $Q_{1, j}$ is isomorphic to
$Q_{1,m}$, then $j=m$.
\end{proof}

\end{appendix}

\medskip

\noindent (Susanne Danz) \  {\sc FB Mathematik, Katholische Universit\"{a}t Eichst\"{a}tt-Ingolstadt, 
85072 Eichst\"{a}tt, Germany}\\
{\it Email} \ susanne.danz@ku.de

\medskip

\noindent (Karin Erdmann) \ {\sc Mathematical Institute, University of Oxford, OX2 6GG, UK}\\
{\it Email} \ erdmann@maths.ox.ac.uk

\end{document}